\documentclass[reqno,12pt]{amsart}
\newtheorem{thm}{Theorem}[section]
\newtheorem{define}[thm]{Definition}

\newtheorem{lem}[thm]{Lemma}

\newtheorem{pro}[thm]{Propostion}
\newtheorem{rem}[thm]{Remark}

\usepackage[letterpaper]{geometry}
\usepackage{setspace}
\usepackage{lmodern}
\usepackage[T1]{fontenc}
\usepackage{amssymb}
\usepackage{mdwlist}
\usepackage{multirow}
\usepackage{amsmath,amssymb,epsfig}
\usepackage{amsmath}
\usepackage{type1cm}
\usepackage{url}
\usepackage{enumerate}
\usepackage[all]{xy}
\usepackage{epsfig}
\usepackage{color}
\newcounter{aa}

\theoremstyle{definition}
\newcommand{\w}{\mathfrak{p}}
\newcommand{\ww}{|\mathfrak{p}|}
\newcommand{\m}{\mathbf}

\newcommand{\ti}{\tilde}

\newcommand{\dha}{\delta^{\frac{1}{2}}}
\newcommand{\dne}{\delta^{-1}}
\newcommand{\dnh}{\delta^{-\frac{1}{2}}}

\newcommand{\gpg}{\mathrm{G}}
\newcommand{\gph}{\mathrm{H}}
\newcommand{\gpm}{\mathrm{M}}
\newcommand{\gpi}{\mathrm{I}}
\newcommand{\gpj}{\mathrm{J}}
\newcommand{\gpk}{\mathrm{K}}
\newcommand{\gpb}{\mathrm{B}}
\newcommand{\gpu}{\mathrm{U}}

\newcommand{\gpt}{\mathrm{T}}
\newcommand{\gpw}{\mathrm{W}}
\newcommand{\gpx}{\mathrm{X}}
\newcommand{\gpy}{\mathrm{Y}}
\newcommand{\gpz}{\mathrm{Z}}
\newcommand{\gpn}{\mathrm{N}}
\newcommand{\funf}{\mathrm{F}}
\newcommand{\funk}{\mathcal{K}}
\newcommand{\aveg}{\mathcal{A}_{\gpw_{\gpg}}}
\newcommand{\avem}{\mathcal{A}_{\gpw_{\gpm}}}

\newcommand{\ud}{\,\mathrm{d}}
\newcommand{\fifi}{\varphi_{1}}
\newcommand{\fise}{\varphi_{2}}
\newcommand{\pair}[2]{\lpair (#1,#2)}
\newcommand{\ppair}[2]{\langle #1,#2 \rangle}
\newcommand{\chag}{\chi}
\newcommand{\cham}{\xi}

\newcommand{\whi}{\mathrm{I}}

\newcommand{\repfi}{\mathrm{Ind}^{\gpg}_{\gpb_{\gpg}}(\chag)}
\newcommand{\repse}{\mathrm{Ind}^{\gpm^{\gpj}}_{\gpb_{\gpm^{\gpj}}}(\cham,\psi)}

\newcommand{\wsf}{\mathrm{W}_{\chag,\cham,\psi}^{0}}
\newcommand{\wsfb}{\mathrm{W}_{\chag,\cham,\psi}}
\newcommand{\iii}{\mathrm{I}_{\chag,\cham,\psi}^{0}}
\newcommand{\iiib}{\mathrm{I}_{\chag,\cham,\psi}}
\newcommand{\ws}{\mathcal{WS}_{\chag,\cham,\psi}}

\newcommand{\iwag}{\varPhi_{1}}
\newcommand{\iwam}{\Psi_{1}}
\newcommand{\iwagw}[1]{\varPhi_{#1}}
\newcommand{\iwamw}[1]{\Psi_{#1}}
\newcommand{\lpair}{l_{\chag,\cham,\psi}}

\newcommand{\schh}{\mathrm{\tilde{T}}}

\begin{document}
\title{The Whittaker-Shintani functions for symplectic groups}
\author{Xin Shen}
\address{School of mathematics, University of Minnesota}
\address{520 Vincent Hall, 206 Church st. SE, Minneapolis, MN, 55455, USA}
\address{shenx125@umn.edu}

\maketitle

\begin{abstract}
  In this note, we give a formula for the Whittaker-Shintani functions for the p-adic symplectic groups, which is a generalization of the Zonal spherical functions and Whittaker functions. We then use the formula to give an alternative proof of a conjecture given by T.Shintani on the unramified calculation of L-functions for $\mathrm{Sp_{2m}}\times \mathrm{GL_{1}}$.

\end{abstract}

%\setlength{\abovedisplayskip}{12pt}
%\setlength{\belowdisplayskip}{12pt}

%\onehalfspacing

\section{Introduction}

\subsection{}
\label{sec:1.1}

Let $\mathbb{G}$ be a connected reductive linear algebraic group defined over a number field $F$ with its ring of Adeles $\mathbf{A}$, $\Pi=\otimes_{v}\Pi_{v}$  an irreducible unitary automorphic cuspidal representation of $\mathbb{G}(\mathbf{A})$, and $r$ a finite dimensional representation of the $L$-group $^{L}\mathbb{G}$ of $\mathbb{G}$. Following Langlands, one may define the partial L-function as 
\begin{equation}
  \label{eq:aa1}
  L_{S}(s,\Pi,r)=\prod_{v\notin S}L(s,\Pi_{v},r_{v})
\end{equation}
where $S$ is a finite set of places of $F$ outside of which both $\mathbb{G}$ and $\Pi_{v}$ are unramified. Langlands conjectured that this partial L-function continues to a meromorphic function in $\mathbb{C}$ which has only finitely many poles and satisfies a standard functional equation relating its value at $s$ to $1-s$. One of the successful approaches to this conjecture, the Rankin-Selberg method, is by constructing a global zeta-integral plus an Euler product expansion, and equating the unramified local zeta-integrals with the ``Langlands factors'' $L(s,\Pi_{v},r_{v})$ (referred to as ``unramified computation'').

One of the interesting cases of the above conjecture is the partial tensor L-function, where $(\mathbb{G}, \Pi, r)=(Sp_{2n}\times GL_{k}, \pi\otimes \tau, \text{"standard"}.)$. Here $\pi$ and $\tau$ are irreducible cuspidal automorphic representations of $Sp_{2n}$ and $GL_{k}$ respectively. The purpose of this paper is to give an explicit formula for the Whittkaker-Shintani functions, which is one of the key steps towards the unramified computation, for the case when $\pi$ is non-generic and $k<n$, following the global construction of zeta-integral in \cite{fourierjacobi}. When $\pi$ is generic the unramified calculation is completed in \cite{MR1675971} using the Casselman-Shalika formula (\cite{MR581582}) for the Whittaker functions, and the formula for Whittkaker-Shintani functions play a parallel role in the non-generic case.

The main idea of this paper actually comes from \cite{MR581582} for the Casselman-Shalika formula for Whittaker functions, and \cite{MR1956080} where the  Whittaker-Shintani functions for orthogonal groups are defined in a similar way and an explicit formula is given. However, the calculation in our case is more technical since the Jacobi group we are dealing with is not reductive. While the general formula we obtain is not as explicit as in the orthogonal case due to its nature, we still have an explicit formula (\ref{whatweneed}) when restricted to the torus, which is enough for the unramified computation. 

The author would like to express his sincere appreciation to his adviser, Prof. Dihua Jiang, for his support and advice in writing this paper. He also thanks Prof. David Ginzburg for his suggestion on the unramified computation. This paper would be part of the author's thesis.

\subsection{}
\label{sec:1.2}

Let $\gpg$ and $\gpm$ be symplectic groups, defined over a non-archimedean local field, of rank $n$ and $m$ respectively with $n\geq m+1$. Let $\repfi$ be the unramified principle series of $\gpg$. Let $\gpm^{\gpj}$ be the Jacobi group and $\gpb_{\gpm^{\gpj}}$ its Borel subgroup as defined in (\ref{notgroup}) in \textbf{Section \ref{Notation}}, and let  $\repse$ be an unramified principle series of $\gpm^{\gpj}$ as defined in (\ref{notrep}) in \textbf{Section \ref{Notation}}. Let $\gpu$ be the unipotent radical of a parabolic subgroup $P_{1}^{n-m-1}$ of $\gpg$ and $\psi_{\gpu}$ be a character on $\gpu$ which is stablized by $\gpm^{\gpj}$ (see (\ref{equ}) and (\ref{eqpsiu})). Then one can define an $\gpm^{\gpj}$-invariant, $(\gpu,\psi_{\gpu})$-covariant pairing $\lpair$ between $\repfi$ and $\repse$. Let $\funf_{\chag}^{0}$ and $\funf_{\cham,\psi}^{0}$ be the normalized spherical vectors in $\repfi$ and $\repse$, and we define
\begin{equation*}
  \wsfb(g)=\lpair(R(g)\funf_{\chag}^{0},\funf_{\cham,\psi}^{0}).
\end{equation*}
This function is a Whittaker-Shintani function attached to $(\chag,\cham,\psi)$ (see \textbf{Definition \ref{defwhi}}). We will show later that for given $(\chag,\cham,\psi)$ (unramified) such function is unique up to a scalar. Denote by $\wsf$ the normalized Whittaker-Shintani function which is equal to 1 at the identity. In this paper we show the following two theorems (for the definition of $\gpx^{0}, \gpz, \gpk_{\gpm^{\gpj}}, \w^{\m d}, \lambda, \w^{\m f}, \gpk_{\gpg}$, see \textbf{Section \ref{Notation}}).
\begin{thm}
  Let $(\chag,\cham)\in \mathbb{C}^{n}\times \mathbb{C}^{m}$, and let $\m d\in \Lambda^{+}_{m}$, $\m f\in \Lambda^{+}_{n}$. Let $\wsf$ be the normailzed Whittaker-Shintani function attached to $(\chag,\cham,\psi)$. Then we have 
  \begin{equation}
    \begin{split}
    \int_{\gpx^{0}}\ud &x\wsf(\w^{\mathbf{d}} x \w^{\mathbf{f}})=\zeta(1)^{-m}\prod_{i=1}^{m}\zeta(2i)\\
    &\cdot \sum_{w\in W_{\gpg},w'\in W_{\gpm}}b(w\chag,w'\cham)d(w\chag)d'(w'\cham)((w\chag)^{-1}\dha)(\w^{\m f})((w'\cham)^{-1}\dha)(\w^{\m d}) .
  \end{split}
  \end{equation}
where
\begin{equation*}
  \mathrm{d}(\chag)=\prod_{1 \leq a< b \leq n }\zeta(\chag_{a}\pm\chag_{b})\prod_{i=1}^{n}\zeta(\chag_{i}), \qquad  \mathrm{d'}(\cham)=\prod_{1\leq a<b \leq m}\zeta(\cham_{a}\pm\cham_{b})\prod_{j=1}^{m}\zeta(2\cham_{j}),
\end{equation*}
and
\begin{equation*}
  \begin{split}
    \mathrm{b}(\chag,\cham)=&\prod_{i<j+n-m}\zeta^{-1}(\chag_{i}-\cham_{j}+\frac{1}{2})\cdot 
\prod_{i>j+n-m}\zeta^{-1}(-\chag_{i}+\cham_{j}+\frac{1}{2})\\
    &\prod_{ 1\leq j \leq m}\zeta^{-1}(\cham_{j}+\frac 1 2)\prod_{\substack{1 \leq i \leq n\\ 1\leq j \leq m}}\zeta^{-1}(\chag_{i}+\cham_{j}+\frac{1}{2})
  \end{split}   
\end{equation*}
\end{thm}

\begin{thm}
  Under the same notation and assumption as in the previous theorem, the support of $\wsf$ is on
  \begin{equation}
    \bigcup_{\m d\in \Lambda^{+}_{m}, \m f\in \Lambda^{+}_{n}}\gpz \gpu \gpk_{\gpm^{\gpj}}(\w^{\m d}\lambda \w^{\m f})\gpk_{\gpg}.
  \end{equation}
 If we let $\mathcal{L}(\m d', \m f')=  \int_{\gpx^{0}}\ud x\wsf(\w^{\mathbf{d'}} x \w^{\mathbf{f'}})$, then there exists $a(\m d')\geq 0$ independent of $(\chag,\cham, \psi)$ such that
\begin{equation}
  \wsf(\w^{\mathbf{d}} \lambda \w^{\mathbf{f}})=\sum_{\m d'}a(\m d')\mathcal{L}(\m d', \mathbf{f+d-d'} ).
\end{equation}
where $\m d'$ runs over the set $\{\m d'\mid\mathbf{d'\in }\Lambda^{+}_{m}, \mathbf{f+d-d'}\in \Lambda^{+}_{n}, \mathbf{d'\leq d}\}$ and $a(\m d)>0$. In particular, we have
\begin{equation}
\label{whatweneed}
  \wsf(\w^{\mathbf{f}})= \mathcal{L}(\m 0, \mathbf{f})
\end{equation}

\end{thm}

The paper is organized as follows. In \textbf{Section \ref{Notation}} we give the notation we use in this paper. In \textbf{Section \ref{sec:step0}} we use the Rankin-Selberg convolution to find an integral expression of the pairing $\lpair$ when $(\chag,\cham)$ belongs to $\mathcal{Z}_{c}\subset \mathbb{C}^{n}\times \mathbb{C}^{m}$ which contains a Hausdoff open set. In \textbf{Section \ref{sec:step1}} we show that the pairing $\lpair$ between $\repfi$ and $\repse$ satisfying Condition A (see \textbf{Definition \ref{conditiona}}) is unique up to a scalar. Then in  \textbf{Section \ref{sec:step3}} we apply the Bernstein's theorem to extend this pairing defined by the integral to generic $(\chag,\cham)$. In \textbf{Section \ref{sec:step4}} we discuss the double cosets of $\gpg$ on which the Whittaker-Shintani function is supported. By considering the vectors invariant under certain open compact subgroups (in \textbf{Section \ref{sec:step5}}) and applying the intertwining operators (in \textbf{Section \ref{sec:step6}}) we give an explicit formula in \textbf{Section \ref{sec:step7}} for the Whittaker-Shintani function attached to generic $(\chag,\cham)$, and we obtain its value at the identity by an combinatorial argument in \textbf{Section \ref{sec:step8}}. After showing the uniqueness of the normailized Whittaker-Shintani function in \textbf{Section \ref{sec:step9}}, we apply the Bernstein' theorem again to extend the formula to  all $(\chag,\cham)$ in \textbf{Section \ref{sec:step10}}. In \textbf{Section \ref{sec:13}} we use the formula we obtained to give an alternative proof of  in  \cite[Theorem 6.1]{MR1121142}, the unramified calculation of L-fucntions for $\mathrm{Sp}_{2n}\times \mathrm{GL_{1}}$.

The application in the last section is in fact a special case of \cite[Theorem 4.3]{fourierjacobi}, the unramified calculation of L-functions for $\mathrm{Sp_{2n}}\times \mathrm{GL}_{k}$. The proof for the general case will be shown in our \cite{xinshen2}. We also expect parellel results for unitary groups (with respect to skew hermitian forms). This will be covered in the future.

\section{Notation}
\label{Notation}
In this paper, we let $F$ be a non-archimedean local field of characteristic 0. Let $\mathcal O$ be its maximal compact subring and $\w$ the uniformizer. Suppose the order of the residue field is $q$ which is not a power of 2. All the groups are defined over $F$. Through out the paper we fix $\psi$ to be an additive character on $\mathbb F$ with conductor $0$. 
\begin{enumerate}[(A)]
\setlength{\itemsep}{12pt}
\item \textbf{Groups.} \label{notgroup}Let $\gpg =\mathrm{Sp}_{2n}$, $\gph=\mathrm{Sp}_{2m+2}$ and $\gpm =\mathrm{Sp}_{2m}$, where $m,n$ are two positive integers with $n\geq m+1$.
$\gpm$ (or $ \gph$) embedds to $\gpg$ as $diag(1_{n-m},g,1_{n-m})$ (or $diag(1_{n-m-1},g,1_{n-m-1})$) for $g\in \gpm$ (or $g\in \gph$). Let $\mathcal{M}_{2n\times 2n}(F)$ be the $2n$ by $2n$ matrix over $F$. For any subgroup $\ti {\gpg}$ of $\gpg$, and any $i\geq 0$, we define $\ti {\gpg}^{i}$ as
\begin{equation*}
  \ti {\gpg}^{i}=\ti {\gpg} \cap \left(I_{2n}+\mathcal M_{2n\times 2n}(\w^{i}\mathcal O)\right).
\end{equation*}
Let  $\gpk_{\gpg}=\gpg^{0}$, and $\gpk_{\gpm}=\gpm\cap \gpk_{\gpg}$. Let  $\gpj$ be the Heisenberg group of dimension $2m+1$ embedding to $\gph$ as 
\begin{equation*}
  \gpj (x,y,z)=
  \begin{pmatrix}
  1&x&y&z\\ &1&&^{t}y\\ & &1& -^{t}x\\ &&&1
\end{pmatrix}
\end{equation*}
where $x,y\in \mathbb F^{m}$, $z\in \mathbb F$. Let $\gpm^{\gpj}=\gpm\ltimes \gpj$, and $\gpk_{\gpm^{\gpj}}=\gpk_{\gpm}\ltimes \gpj^{0}$.  We let $\gpx(x)=\gpj (x,0,0)$, $\gpy(y)=\gpj (0,y,0)$, $\gpz(z)=\gpj (0,0,z)$ and let $\gpx,\gpy,\gpz$ be the group of them respectively. Let $\gpb_{\gpg}$, $\gpb_{\gph}$ and $\gpb_{\gpm}$ be the standard Borel subgroup of $\gpg,\gph,\gpm$, and $\gpb_{\gpm^{\gpj}}=\gpb_{\gpm}\ltimes (\gpy\times \gpz)$, and let $\gpn_{\gpg}, \gpn_{\gph}, \gpn_{\gpm},\gpn_{\gpm^{\gpj}}$ be their unipotent radical respectively. Let $\mathrm{T}_{\gpg}$ be the toral part of of $\gpb_{\gpg}$, and let
  \begin{equation*}
    \begin{split}
      &      \Lambda_{k}^{+}=\{(d_{1},...,d_{k})\in \mathbb{Z}^{k} \mid d_{1}\geq d_{2}\geq... \geq d_{k}\geq 0 \}\\
      &\mathrm{T}_{\gpg}^{+}=\{diag(t_{1},...,t_{n},t_{n}^{-1},...,t_{1}^{-1})\mid |t_{1}|\leq...\leq |t_{n}|\leq 1\}\\
      &\mathrm{T}_{\gpg}^{-}=\{t^{-1}\mid t\in T_{\gpg}^{+}\}
    \end{split}
  \end{equation*}
The definition of $T_{\gpm}^{+}$ and $T_{\gpm}^{-}$ are similar. Let $P^{n-m-1}_{1}$ be the standard parabolic subgroup of $\gpg$ with Levi decomposition
\begin{equation}
\label{equ}
  P_{1}^{n-m-1}=\mathrm{GL}_{1}^{n-m-1}\times \gph \ltimes \gpu.
\end{equation}
 Let $\psi_{\gpu}$ be the character on $\gpu$ given by
\begin{equation}
\label{eqpsiu}
 \psi_{\gpu}(u)=\psi(\sum_{i=1}^{n-m-1}u_{i,i+1}),
\end{equation}
which is stablized by $\gpm^{\gpj}$.
We denote by $\gpi_{\gpg}$, $\gpi_{\gpm}$ the Iwahori subgroups of $\gpg$ and $\gpm$, and $\overline I_{\gpm}=I_{\gpm}\ltimes \gpj^{0}$. Let $\gpw_{\gpg}$, $\gpw_{\gpm}$ be the Weyl group of $\gpg$ and $\gpm$ with respect to $\gpt_{\gpg}$ and $\gpt_{\gpm}$.
\item \textbf{Elements.} Let $w_{0}^{\gpg}$ be the longest Weyl element in $\gpg$. For $k\leq n$, and given $t_{1},...,t_{k}\in F^{*}$, we let $d_{k}(t_{1},...,t_{k})=diag(I_{n-k},t_{1},\ldots,t_{k},t_{k}^{-1},\ldots,t_{1}^{-1},I_{n-k})\in \gpt_{\gpg}$. Let $\mathbb{\overline Z}=\mathbb{Z}\cup \{\infty\}$, and let $v$ be the normalized valuation from $F$ to $\mathbb{\overline Z}$. For $\mathbf{a,b}\in \mathbb{\overline Z}^{k}$, we define an order in $ \mathbb{\overline Z}$ such that $\mathbf{a}\geq \mathbf{b}$ if and only if $\mathbf{a-b}\in \mathbb{(N\cup \{\infty\})}^{k}$. We define $\mathrm{min}\mathbf{(a,b)}=(\mathrm{min}(a_{1},b_{1}),\ldots , \mathrm{min} (a_{k},b_{k})).$ When $\mathbf{a}\in \mathbb{\overline Z}^{m}$, we let $\lambda(\mathbf{a})=\gpx(\w^{a_{1}},\ldots,\w^{a_{m}}).$ Here $\w^{\infty}=0.$ Let $\lambda=\lambda(\m 0)$. For $\mathbf{a}=(a_{1},...,a_{k})\in \mathbb{Z}^{k}$, we let $\w^{\mathbf{a}}=d_{k}(\w^{a_{1}},...,\w^{a_{k}})$.

\item \textbf{Representations.}\label{notrep} Let $\chag$ and $\cham$ be unramified characters on $\gpt_{\gpg}$ and $\gpt_{\gpm}$. We parametrize them as $\chag=(\chag_{1},...,\chag_{n})\in \mathbb{C}^{n}$ and $\cham=(\cham_{1},...,\cham_{m})\in \mathbb{C}^{m}$ such that $\chag(d_{n}(t_{1},...,t_{n}))=\prod_{i=1}^{n}|t_{i}|^{\chag_{i}}$ and $  \cham(d_{m}(t_{1},...,t_{m}))=\prod_{j=1}^{m}|t_{j}|^{\cham_{j}}.$ Let $\mathrm{Ind}^{\gpm^{\gpj}}_{\gpb_{\gpm^{\gpj}}}(\cham, \psi)$ be a representation of $\gpm^{\gpj}$ consistiting of smooth functions on $\gpm^{\gpj}$ such that
\begin{equation*}
 f(b_{\gpm}(0,y,z)m^{\gpj})=\cham \delta_{\gpb_{\gpm^{\gpj}}}^{\frac{1}{2}}(b_{m})\psi(z)f(m^{\gpj}),
\end{equation*}
with $\gpm^{\gpj}$ acting by right translation. Sometimes we write $\cham\psi$ as a character on $\gpb_{\gpm^{\gpj}}$ such that
\begin{equation*}
  \cham\psi(b_{\gpm}\gpj(0,y,z))=\cham(b_{\gpm})\psi(z).
\end{equation*}

\end{enumerate}
%\suspend{enumerate}
\setcounter{aa}{\value{enumi}}
\begin{rem}
  Although $\gpb_{\gpm^{\gpj}}\backslash \gpm^{\gpj}$ is not compact, functions in $\mathrm{Ind}^{\gpm^{\gpj}}_{\gpb_{\gpm^{\gpj}}}(\cham, \psi)$ are compactly supported on $\gpb_{\gpm^{\gpj}}\backslash \gpm^{\gpj}$ by smoothness. In fact by Iwasawa decomposition on $\gpm$, we have
  \begin{equation*}
    \gpm^{\gpj}=\gpb_{\gpm^{\gpj}}\gpx\gpk_{\gpm}.
  \end{equation*}
Suppose $f\in \mathrm{Ind}^{\gpm^{\gpj}}_{\gpb_{\gpm^{\gpj}}}(\cham, \psi)$ which is right $\gpk_{f}$ invariant for some open compact subgroup $\gpk_{f}$. Note that $\gpk_{f}\backslash \gpk_{\gpm^{\gpj}}$ is finite. Let $k$ be a representative in one of the cosets and suppose $f(xk)\neq 0$. Then by the smoothness of $f$, $f(xk)=f(xyk)$ when $y$ is in a neighbourhood of $0\in F^{m}$. But note that $f(xyk)=\psi(2\ppair{x}{y})f(xk)$, so $\psi(\ppair{x}{y})=1$ for all such $y$, which implies that $x$ belongs to a compact subset. 

In particular, if $f$ is $\gpk_{\gpm^{\gpj}}$-invariant, then $f(x)\neq 0$ implies $x\in \gpx^{0}$. So the spherical vector in $\mathrm{Ind}^{\gpm^{\gpj}}_{\gpb_{\gpm^{\gpj}}}(\cham, \psi)$ is supported on $\gpb_{\gpm^{\gpj}}\gpk_{\gpm^{\gpj}}$, and is unique up to a scalar.
\end{rem}

\begin{enumerate}
\setcounter{enumi}{\value{aa}}

\item\textbf{Functions and Functional.} \label{def17}
We denote by $\zeta(s)$ the local zeta function as $\zeta(s)=(1-q^{-s})^{-1}$. For any set $\mathcal{X}$ we denote by $Ch_{\mathcal{X}}$ its characteristic function.  For $\fifi\in \mathbf{C}^{\infty}_{c}(\gpg)$, we let
\begin{equation}
\label{eqfchag}
  \funf_{\chag}(\varphi_{1})(g)=\int_{\gpb_{\gpg}}\chag^{-1}\dha_{\gpb_{\gpg}}(b_{\gpg})\varphi_{1}(b_{\gpg}g)\ud_{l}b_{\gpg}.
\end{equation}
Then the map $\varphi_{1}\mapsto \funf_{\chag}(\varphi_{1})$ is surjective from $\mathbf{C}^{\infty}_{c}(\gpg)$ to $\repfi$.
Similarly for $\fise \in \mathbf{C}^{\infty}_{c}(\gpm^{\gpj})$, we let
\begin{equation}
\label{eqfcham}
  \funf_{\cham,\psi}(\fise)(m^{\gpj})=\int_{\gpb_{\gpm^{\gpj}}} (\cham\psi)^{-1}\dha_{\gpb_{\gpm^{\gpj}}}(b_{\gpm^{\gpj}})\fise(b_{\gpm^{\gpj}}m^{\gpj})\ud_{l}b_{\gpm^{\gpj}}.
\end{equation}
The map $\fise\mapsto \funf_{\cham,\psi}(\fise)$ is surjective from $\mathbf{C}^{\infty}_{c}(\gpm^{\gpj})$ to 
$\repse$. Let $\funk_{\chag,\cham,\psi}$ be a function defined on $\gpg$ such that
\begin{equation}
\label{def19}
    \funk_{\chag,\cham,\psi}(b_{\gpg}w_{0}^{\gpg}\lambda \gpj(0,y,z)b_{\gpm}u)=\chag^{-1}\delta^{\frac{1}{2}}(b_{\gpg})\cham\delta^{-\frac{1}{2}}(b_{\gpm})\psi(z)\psi_{\gpu}^{-1}(u),
\end{equation}
and $\funk_{\chag,\cham,\psi}(g)=0$ for all other $g$.  For $\fifi\in C_{c}^{\infty}(\gpg)$ and $\fise\in C_{c}^{\infty}(\gpm^{\gpj})$, let
\begin{equation}
\label{integral}
  \iiib(\fifi,\fise)(g)=\int_{\gpg}dg'\int_{\gpm^{\gpj}}dm^{\gpj}\varphi_{1}(g'){K_{\chag,\cham,\psi}}(g'g^{-1}(m^{\gpj})^{-1})\varphi_{2}(m^{\gpj}),
\end{equation}
and let  $\iii(g)=\iiib(Ch_{\gpk_{\gpg}},Ch_{\gpk_{\gpm^{\gpj}}})(g)$. Let $\funf_{\chag}^{0}=\funf_{\chag}(Ch_{\gpk_{\gpg}})$ and $\funf_{\cham,\psi}^{0}=\funf_{\cham,\psi}(Ch_{\gpk_{\gpm^{\gpj}}})$ be spherical elements in $\repfi$ and $\repse$ respectively. Let $\mathcal{H}_{\gpg}$ be the spherical Hecke algebra of $\gpg$, and $\mathcal{H}_{\gpm^{\gpj},\psi}$ be the spherical Hecke algebra of $\gpm^{\gpj}$ with respect to $\psi$ as defined in section 4 of \cite{MR1018057}, and let them act on $\repfi^{\gpk_{\gpg}}$ and $\repse^{\gpk_{\gpm^{\gpj}}}$ by characters $\omega_{\chag}$ and $\omega_{\cham}$ respectively. For any function $f$ on $\gpg$, let $(L(g_{0})f)(g)=f(g_{0}^{-1}g)$, and $(R(g_{0})f)(g)=f(gg_{0})$. 

\begin{define}
\label{conditiona}
  A pairing $\lpair$ between $\repfi$ and $\repse$ is called satisfying \textrm{Condition A} if
\begin{enumerate}[(i)]
\item $\lpair(\funf_{\chag},\funf_{\cham,\psi})=\lpair(R(m^{\gpj})\funf_{\chag},R(m^{\gpj})\funf_{\cham,\psi})$ for any $m^{\gpj}\in \gpm^{\gpj}$.
\item $\lpair(R(u)\funf_{\chag},\funf_{\cham,\psi})=\psi_{\gpu}(u)\lpair(\funf_{\chag},\funf_{\cham,\psi})$ for any $u\in \gpu$.
\end{enumerate}
\end{define}

\begin{define}
\label{defwhi}
  For $(\chag,\cham)\in \mathbb{C}^{n}\times \mathbb{C}^{m}$, a function $\wsfb \in \mathbf{C}^{\infty}(\gpg)$ is called a \textrm{Whittaker-Shintani Function} attached to $(\chag,\cham)$ if
\begin{enumerate}[(i)]
\item $\wsfb(zuk_{\gpm^{\gpj}}gk_{\gpg}) =\psi^{-1}(z)\psi_{\gpu}(u)\wsfb(g) $.\vspace{3pt}
\item $L(\varphi_{\gpm^{\gpj}})R(\varphi_{\gpg})\wsfb =\omega_{\cham}(\varphi_{\gpm^{\gpj}})\omega_{\chag}(\varphi_{\gpg})\cdot \wsfb $  for any $\varphi_{\gpm^{\gpj}}\in \mathcal{H}_{\gpm^{\gpj},\psi}$ and $\varphi_{\gpg}\in \mathcal{H}_{\gpg}$.
\end{enumerate} The space of Whittaker-Shintani functions attached to $(\chag,\cham,\psi)$ is denoted by $\ws$. Sometimes we omit $\psi$ because it is fixed in this paper.  A Whittaker-Shintani function is called a \textrm{Normalized Whittaker-Shintani function} if it equals 1 at the identity.
\end{define}
\end{enumerate}

\section{Integral expression for the pairing}
\label{sec:step0}

We first use the function  $\funk_{\chag,\cham,\psi}$, as defined in  (\ref{def19}), to construct a pairing between $\repfi$ and $\repse$ satisfying \textrm{Condition A}. For any element $g\in \gpb_{\gpg}w_{0}^{\gpg}\gpn_{\gpg}$, the way to express $g=bw_{0}^{\gpg}n$ with $b\in \gpb_{\gpg}$ and $n\in \gpn_{\gpg}$ is unique. From this it is not hard to see that the set $\gpb_{\gpg}w_{0}^{\gpg}\lambda \gpy \gpz \gpb_{\gpm}\gpu$ has the same property. So the function $\funk_{\chag,\cham,\psi}$ is well-defined. Moreover $\gpb_{\gpg}w_{0}^{\gpg}\lambda \gpy \gpz \gpb_{\gpm}\gpu$ is zarisky open in $\gpg$ by lemma (\ref{lemopen}) below. Let $\fifi\in \mathbf{C}_{c}^{\infty}(\gpg)$ and $\fise\in \mathbf{C}_{c}^{\infty}(\gpm^{\gpj})$, and let $\funf_{\chag}(\fifi)$ and $\funf_{\cham,\psi}(\fise)$ be defined as in (\ref{eqfchag}) and (\ref{eqfcham}). Then we let
\begin{equation*}
  \mathcal{E}(\funf_{\chag}(\fifi))(g)=\int_{\gpb_{\gpg}\backslash \gpg}\funf_{\chag}(\fifi)(\dot gg){\funk_{\chag,\cham,\psi}}(\dot g)\ud \dot g,
\end{equation*}
where $\ud\dot g$ is the right $\gpg$-invariant functional on $\mathrm{Ind}^{\gpg}_{\gpb_{\gpg}}(\delta_{\gpb_{\gpg}}^{\frac 1 2})$ determined by the Haar measure of $\gpg$. By direct calculation we have
\begin{equation*}
  \mathcal{E}(\funf_{\chag}(\fifi))(g)=\int_{\gpg}\fifi(g'g){\funk_{\chag,\cham,\psi}}(g')\ud g\ ,
\end{equation*}
which is convergent when $\funk_{\chag,\cham,\psi}$ is continuous on $\gpg$. $\mathcal{E}(\funf_{\chag}(\fifi))$ satisfies 
\begin{enumerate}
\item $\mathcal{E}(R(u)\funf_{\chag}(\fifi))(g)=\psi_{\gpu}(u)\mathcal{E}(\funf_{\chag}(\fifi))(g)$ when $u\in \gpu$.
\item When restricted to $\gpm^{\gpj}$, $\mathcal{E}(\funf_{\chag}(\fifi))\in \mathrm{Ind}_{\gpb_{\gpm^{\gpj}}}^{\gpm^{\gpj}}(\cham^{-1},\psi)$.
\end{enumerate}
So $\mathcal{E}$ actually gives an $\gpm^{\gpj}$-homomorphism from the twisted Jacquet-module \\$ \left(\repfi\right)_{\gpu,\psi_{\gpu}}$ to $\mathrm{Ind}_{\gpb_{\gpm^{\gpj}}}^{\gpm^{\gpj}}(\cham^{-1},\psi)$. Now we consider the integral
\begin{equation*}
  \pair{\funf_{\chag}(\fifi)}{\funf_{\cham,\psi}(\fise)}=\int_{\gpb_{\gpm^{\gpj}}\backslash \gpm^{\gpj}}\mathcal{E}(\funf_{\chag}(\fifi))(\dot m^{\gpj})\funf_{\cham,\psi}(\fise)(\dot m^{\gpj})\ud \dot m^{\gpj}
\end{equation*}
where $\ud \dot m^{\gpj}$ is the right $\gpm^{\gpj}$-invariant functional on $\mathrm{Ind}^{\gpm^{\gpj}}_{\gpb_{\gpm^{\gpj}}}(\delta_{\gpb_{\gpm^{\gpj}}}^{\frac 1 2})$ determined by the Haar measure of $\gpm^{\gpj}$. Substituting $\mathcal{E}$ and $\funf_{\cham,\psi}$ by definition, we have
\begin{equation}
\label{defl}
  \pair{\funf_{\chag}(\fifi)}{\funf_{\cham,\psi}(\fise)}=\int_{\gpg}\int_{\gpm^{\gpj}}\fifi(g'){\funk_{\chag,\cham,\psi}}(g'(m^{\gpj})^{-1})\fise(m^{\gpj})\ud m^{\gpj} \ud g'.
\end{equation}
Note that the right hand side is actually $\iiib(\fifi,\fise)(e)$ as defined in (\ref{integral}). 

It is easy to see that the pairing $\pair{\funf_{\chag}(\fifi)}{\funf_{\cham,\psi}(\fise)}$ satisfies \textrm{Condition A} if the integral is convergent, and the integral is convergent if $\funk_{\chag,\cham,\psi}$ is continuous on $\gpg$. In the rest of this section we will prove the following proposition. 
\begin{pro}
\label{pro3.1}
  Let $\mathcal{Z}_{c}$ be the set of unramified characters $(\chag,\cham)$ satisfying 
  \begin{equation}
    \begin{cases}
      Re(\chag_{i}-\chag_{i+1})\geq 1&\text{ for }1\leq i \leq n-m-1\\
      Re(\chag_{n-m-1+j}-\cham_{j})\geq \frac 1 2 &\text{ for }1\leq j \leq m\\
      Re(-\chag_{n-m+j}+\cham_{j})\geq \frac 1 2&\text{ for } 1\leq j \leq m\\
      Re(\chag_{n})\geq 1&
    \end{cases}
  \end{equation}
then when $(\chag,\cham)\in \mathcal{Z}_{c}$, the function $\funk_{\chag,\cham,\psi}$ is continuous on $\gpg$, and as a consequence, the integral (\ref{defl}) is convergent.
\end{pro}
 
Since $\funk_{\chag,\cham,\psi}$ is defined continuously on an Zariski open subset of $\gpg$ (we will see this soon) and extend by $0$ to $\gpg$, we only need to show the continuity outside the Zariski open set, for the function $|\funk_{\chag,\cham,\psi}|$. The method we use here is similar to that in \cite{MR1956080}.

First by the Bruhat decomposition we have
\begin{equation*}
  \gpg=\bigcup_{w\in \gpw_{\gpg}}\gpb_{\gpg}w\gpn_{\gpg}. 
\end{equation*}
And we know that $\gpb_{\gpg}w_{0}^{\gpg}\gpn_{\gpg}$ is zariski open in $\gpg$. In fact we have
\begin{lem}
\label{lemmaalpha}
  There exists $\alpha_{k}\in \mathfrak{o}[\gpg]$ for $1\leq k \leq n$  such that $$\gpb_{\gpg}w_{0}^{\gpg}\gpn_{\gpg}=\{g\mid \alpha_{k}(g)\neq 0\text{ for all $k$ }\}.$$ 
\end{lem}
\begin{proof}
For $g\in \gpg$, let its matrix be $g=\left(g_{ij}\right)_{1\leq i,j\leq 2n}$. Let $\mathcal{N}_{2n}=\{1,2,\ldots,2n\}$. For $I=(i_{1},\ldots,i_{k})$ and $J=(j_{1},\ldots,j_{k})$ both belonging to $\mathcal{N}_{2n}^{k}$, we Let $g_{IJ}=(g_{i_{s},j_{t}})_{1\leq s,t\leq k}$. We define
\begin{equation}
\label{eqIJ}
  \Delta_{IJ}(g)=\det g_{IJ}.
\end{equation}
For $1\leq k\leq n$, let $I_{k}=\{2n+1-k,2n+1-(k-1),\ldots, 2n\}$, and $J_{k}=\{1,2,\ldots,k\}$, and we take
\begin{equation}
\label{eqalpha}
  \alpha_{k}(g)=\Delta_{I_{k},J_{k}}(g).
\end{equation}
Then one can check that 
\begin{enumerate}
\item For any $n_{1},n_{2}\in \gpn_{\gpg}$, $\alpha_{k}(n_{1}gn_{2})=\alpha_{k}(g)$.
\item $\alpha_{k}(d_{n}(t_{1},\ldots,t_{n})gd_{n}(s_{1},\ldots, s_{n}))=\prod_{i=1}^{k}t_{i}^{-1}s_{i}\cdot \alpha_{k}(g)$.
\item Let $w\in \gpw_{\gpg}$. If $\alpha_{k}(w)\neq 0$ for all $1\leq k\leq n$, then $w=w_{0}^{\gpg}$. 
\end{enumerate}
Combining these properties with the Bruhat decomposition of $\gpg$, we have our lemma. 
\end{proof}

Next we have 

\begin{lem}
\label{lemopen}
  There exists $\beta_{l}\in \mathfrak{o}[\gpg]$ for $1\leq l\leq m$ such that
  \begin{equation*}
     \gpb_{\gpg}w_{0}^{\gpg}\lambda \gpy \gpz \gpb_{\gpm}\gpu=\{ g\in \gpg\mid \alpha_{k}(g)\neq 0, \beta_{l}(g)\neq 0\text{ for all }1\leq k \leq n, 1\leq l \leq m \},
  \end{equation*}
where $\alpha_{k}$ is as defined in lemma \ref{lemmaalpha}.
\end{lem}
\begin{proof}
   Note that for any $w\in\gpw$, $\gpb_{\gpg}w\gpn_{\gpg}=\gpb_{\gpg}w\gpx\gpu\gpn_{\gpm^{\gpj}}$. For any $\gpx(x_{1},\ldots,x_{m})\in \gpx$, we have $\gpx(x_{1},\ldots,x_{m})=\m s^{-1}\gpx(r_{1},\ldots,r_{m})\m s$, where $\m s=d_{m}(s_{1},\ldots, s_{m})\in \gpt_{\gpm}$ such that
\begin{equation*}
  (s_{i},r_{i})=\begin{cases}
    (x_{i},1)&\text{ if }x_{i}\neq 0;\\
    (1,0)& \text{ if }x_{i}=0
  \end{cases}
\end{equation*}
From this we can see that
\begin{equation}
\label{bruhat}
  \gpb_{\gpg}w\gpn_{\gpg}=\bigcup_{\m r\in \{0,1\}^{m}} \gpb_{\gpg}w \gpx(\m r)\gpb_{\gpm^{\gpj}}\gpu,
\end{equation}
and when $w=w_{0}^{\gpg}$, the union is disjoint. For $1\leq l\leq m$, we let $$J'_{l}=\{1,2,\ldots,\hat{(n-m)},\ldots,n-m+l\},$$ and we define
\begin{equation}
  \label{eqbeta}
  \beta_{l}(g)=\Delta_{I_{n-m+l-1},J'_{l}}(g).
\end{equation}
Then $\beta_{l}$ satisfies
\begin{enumerate}
  \item $\beta_{l}(n_{1}g)=\beta_{l}(g)$ for any $n_{1}\in \gpn_{\gpg}$.
  \item $\beta_{l}(gn_{2}u)=\beta_{l}(g)$ for any $n_{2}\in \gpn_{\gpm^{\gpj}}$ and $u\in \gpu$.
  \item $\beta_{l}(d_{n}(t_{1},\ldots,t_{n})gd_{m}(s_{1},\ldots,s_{m}))=\prod_{i=1}^{n-m}t_{i}^{-1}\cdot \prod_{j=1}^{l-1}t_{n-m+j}^{-1}\cdot \prod_{j=1}^{l}s_{j}\cdot \beta_{l}(g).$
  \item $\beta_{l}(w_{0}^{\gpg}\gpx(\m r))=\pm r_{l}$. The sign in front of $r_{l}$ depends on $l$, which is not impotant since we are only interested in $|\beta_{l}|$.
  \end{enumerate}
In fact, (1) is by the definition of $I_{k}$ while (3) and (4) are by direct calculation. For (2), note that if we only consider the first $n$ column of $(g_{ij})$, multiplying elements in $\gpn_{\gpm^{\gpj}}\gpu$ from the right corresponds to column operations adding multiples of column $k_{1}$ to column $k_{2}$ where $1\leq k_{1}<k_{2}\leq n$ with $k_{1}\neq n-m$. On the other hand elements in $J'_{l}$ are consecutive from $1$ to $n-m+l$ with $n-m$ missing, so $\Delta_{I_{n-m+l-1},J'_{l}}(g)$ is invariant under such column operations. 

So for $g\in \gpb_{\gpg}w \gpx(\m r)\gpb_{\gpm^{\gpj}}\gpu$, $\alpha_{k}(g)\neq 0$ and $\beta_{l}(g)\neq 0$ for all $1\leq k \leq n$ and $1\leq l \leq m$ if and only if $w=w_{0}^{\gpg}$ (by lemma \ref{lemmaalpha}) and $\gpx(\m r)=\lambda$ (by the property of $\beta_{l}$'s), completing our proof.
\end{proof}
 
\begin{rem}
  If we let $\varpi_{i}=\epsilon_{1}+\ldots+\epsilon_{i}\in \mathrm{Hom}(\gpt_{\gpg},\mathrm{GL_{1}})$ and $\varpi'_{j}=\epsilon_{1}'+\ldots +\epsilon_{j}'\in\mathrm{Hom}(\gpt_{\gpm},\mathrm{GL_{1}})$ be the dominant weights of $\gpg$ and $\gpm$ with respect to $\gpb_{\gpg}$ and $\gpb_{\gpm}$, then the properties of $\alpha_{k}$ and $\beta_{l}$ actually shows that under the $\gpb_{\gpg}\times\gpb_{\gpm^{\gpj}}$ action, $\alpha_{i} $ has the highest weight $(\varpi_{k},0)$ when $1\leq k\leq n-m$, and $(\varpi_{k},\varpi'_{k-(n-m)})$ when $n-m+1\leq k \leq n$, and $\beta_{l}$ has the highest weight $(\varpi_{n-m+l-1},\varpi_{l}')$.
\end{rem}
Now we can expressed $\funk_{\chag,\cham,\psi}$ by $\alpha_{k}$ and $\beta_{l}$. First we have 
\begin{lem}
  Let $g=d_{n}(t_{1},\ldots,t_{n})n_{\gpg}w_{0}^{\gpg}\lambda d_{m}(s_{1},\ldots, s_{m})n_{\gpm^{\gpj}}u\in \gpb_{\gpg}w_{0}^{\gpg}\lambda\gpb_{\gpm^{\gpj}}\gpu$, we have
  \begin{equation*}
    |t_{i}|=
    \begin{cases}
      |\alpha_{1}(g)^{-1}|&\text{ if }i=1;\\
      |\alpha_{i-1}\alpha_{i}^{-1}(g)|&\text{ if }2\leq i\leq n-m;\\
      |\beta_{i-(n-m)}\alpha_{i}^{-1}(g)|&\text{ if }i>n-m
    \end{cases}
  \end{equation*}
and
\begin{equation*}
  |s_{j}|=|\beta_{j}\alpha_{n-m+j-1}^{-1}(g)|\qquad \text{for }1\leq j\leq m
\end{equation*}
\end{lem}
By this we have 
\begin{lem}
\label{lemkk}
  For $g\in \gpb_{\gpg}w_{0}^{\gpg}\lambda\gpb_{\gpm^{\gpj}}\gpu$, we have
  \begin{equation*}
    \begin{split}
      |\funk_{\chag,\cham,\psi}(g)|=&\prod_{i=1}^{n-m-1}|(\chag_{i}\chag_{i+1}^{-1}|\cdot|^{-1})(\alpha_{i}(g))|\cdot \prod_{j=1}^{m}|(\chag_{n-m-1+j}\cham_{j}^{-1}|\cdot|^{-\frac 1 2}(\alpha_{n-m+j}(g))|\\
      &\cdot |(\chag_{n}|\cdot|^{-1})(\alpha_{n}(g))|\cdot \prod_{j=1}^{m}|(\chag_{n-m+j}^{-1}\cham_{j}|\cdot|^{-\frac 1 2})(\beta_{j}(g))|.
    \end{split}    
  \end{equation*}
\end{lem}
The proof of these are by direct calculation. Note that $\alpha_{k}$, $\beta_{l}$ are continuous functions on $\gpg$, so when the assumptions in proposition \ref{pro3.1} are satisfied, the extension of $|\funk_{\chag,\cham,\psi}(g)|$ by 0 to outside the set $\gpb_{\gpg}w_{0}^{\gpg}\lambda\gpb_{\gpm^{\gpj}}\gpu$ is continuous, and so $\funk_{\chag,\cham,\psi}(g)$ is continuous. 

\section{Uniqueness for the pairing for generic $(\chag,\cham)$ }
\label{sec:step1}

The pairing $\lpair$ satisfying \textrm{Condition A} corresponds to the homomorphism
\begin{equation}
\label{18}
  \mathrm{Hom}_{\gpb_{\gpm^{\gpj}}\gpu}(\mathrm{Ind}_{\gpb_{\gpg}}^{\gpg}(\chag),\cham^{-1}\psi^{-1}\dha_{\gpb_{\gpm^{\gpj}}}\otimes\psi_{\gpu}).
\end{equation}
In this section we prove that 

\begin{pro}
  For generic $(\chag,\cham)$,
  \begin{equation*}
    \mathrm{dim}\mathrm{Hom}_{\gpb_{\gpm^{\gpj}}\gpu}(\mathrm{Ind}_{\gpb_{\gpg}}^{\gpg}(\chag),\cham^{-1}\psi^{-1}\dha_{\gpb_{\gpm^{\gpj}}}\otimes\psi_{\gpu})\leq 1.
  \end{equation*}
\end{pro}

Let $\mathcal{U}_{d}$ be the union of double cosets of $\gpb_{\gpg}\backslash\gpg\slash \gpb_{\gpm^{\gpj}}\gpu$ with codimension $\leq d$. Then $\mathcal{U}_{0}=\gpb_{\gpg}w_{0}^{\gpg}\lambda\gpb_{\gpm^{\gpj}}\gpu$ is open in $\gpg$, and for any $d\geq 1$ we have the exact sequence
\begin{equation*}
  0\to \mathbf{S}(\mathcal{U}_{d-1})\to \mathbf{S}(\mathcal{U}_{d})\to \sum_{\text{codim }\mathcal{U}=d}\mathbf{S}(\mathcal{U})\to 0.
\end{equation*}
Obviously we have
\begin{equation*}
  \mathrm{dim\,Hom}_{\gpb_{\gpm^{\gpj}}\gpu}(\mathrm{Ind}_{\gpb_{\gpg}}^{\gpg}(\chag,\mathcal{U}_{0}),\cham^{-1}\psi^{-1}\dha_{\gpb_{\gpm^{\gpj}}}\otimes\psi_{\gpu})\leq 1,
\end{equation*}
so we only need to show that 
\begin{lem} Suppose $(\chag,\cham)$ is generic, and let $\mathcal{U}=\gpb_{\gpg} g \gpb_{\gpm^{\gpj}}\gpu$ be a double coset different from $\mathcal{U}_{0}$, then
  \begin{equation*}
    \mathrm{dim\,Hom}_{\gpb_{\gpm^{\gpj}}\gpu}(\mathrm{Ind}_{\gpb_{\gpg}}^{\gpg}(\chag,\mathcal{U}),\cham^{-1}\psi^{-1}\dha_{\gpb_{\gpm^{\gpj}}}\otimes\psi_{\gpu})=0.
  \end{equation*}
\end{lem}
\begin{proof}
When $\mathcal{U}=\gpb_{\gpg} g \gpb_{\gpm^{\gpj}}\gpu$, we let $\gpg_{g}= \gpb_{\gpm^{\gpj}}\gpu\cap g^{-1}\gpb_{\gpg}g$, then
\begin{equation*}
  \begin{split}
    &\mathrm{Hom}_{\gpb_{\gpm^{\gpj}}\gpu}(\mathrm{Ind}_{\gpb_{\gpg}}^{\gpg}(\chag,\mathcal{U}),\cham^{-1}\psi^{-1}\dha_{\gpb_{\gpm^{\gpj}}}\otimes\psi_{\gpu})\\
=&\mathrm{Hom}_{\gpg_{g}}(g^{-1}(\chag\dha_{\gpb_{\gpg}})\otimes \cham\psi\dnh_{\gpb_{\gpm^{\gpj}}}\otimes \psi_{\gpu}^{-1},\delta_{g})
  \end{split}
\end{equation*}
where $\delta_{g}$ is the modulus character of $\gpg_{g}$. So we need to show that
\begin{equation}
\label{equiv1}
  g^{-1}(\chag\dha_{\gpb_{\gpg}})\cdot \cham\psi\dnh_{\gpb_{\gpm^{\gpj}}}\cdot \psi_{\gpu}^{-1}\cdot \delta_{g}^{-1}\not\equiv 1
\end{equation}
when restricted to $\gpg_{g}$ for any generic $(\chag,\cham)$. Recall from (\ref{bruhat}) that
  \begin{equation*}
    \gpg=\bigcup_{w\in \gpw_{\gpg},\mathbf{r}\in \{0,1\}^{m}}\gpb_{\gpg}w\gpx(\mathbf{r})\gpb_{\gpm^{\gpj}}.
  \end{equation*}
So we can assume  $g=w\gpx(\mathbf{r})$ with either $w\neq w_{0}^{\gpg}$ or $\mathbf{r}\neq (1,1,\ldots, 1)$. What we need to find is $b_{1}\in \gpb_{\gpg}$ and $b_{2}\in \gpb_{\gpm^{\gpj}}\gpu$ such that
\begin{equation}
  \label{b1b2}
  \begin{split}
    &b_{1}g=gb_{2}\\
    &(\chag\dha_{\gpb_{\gpg}})(b_{1})\neq \cham\psi\dnh_{\gpb_{\gpm^{\gpj}}}\psi_{\gpu}\cdot \delta_{g}(b_{2}).
  \end{split}
\end{equation}

First suppose $\mathbf{r}\neq  (1,1,\ldots, 1)$. In this case we claim that $\gpt_{\gpg}\cap (g\gpt_{\gpm}g^{-1})$ contains a nontrivial torus.  Let $\mathbf{t}=d_{m}(t_{1},\ldots t_{m})\in \gpt_{\gpm}$. Note that $\gpt_{\gpg}$ is stablized by the adjoint action of $\gpw_{\gpg}$, so it suffices to show that there exists a nontrivial torus $\gpt_{s}$ of $\gpt_{\gpm}$ such that when $\mathbf{t}\in \gpt_{s}$, $\gpx(\mathbf{r})^{-1}\mathbf{t}\gpx(\mathbf{r})\in \gpt_{\gpg}$. Note that
\begin{equation*}
  \gpx(\mathbf{r})^{-1}\mathbf{t}\gpx(\mathbf{r})=\mathbf{t}\cdot \gpx((1-t_{1})r_{1},(1-t_{2})r_{2},\ldots,(1-t_{m})r_{m}),
\end{equation*}
so when $r_{j}=0$ for some $j$, we can let $\gpt_{s}=\{\mathbf{t}=d_{m}(1,\ldots,\overset{\text{j-th}}{t},\ldots, 1)\}$ be the torus we claimed. Then since $(\chag,\cham)$ is generic, one can find some $b_{2}\in \gpt_{s}$ and $b_{1}=gb_{2}g^{-1}\in \gpt_{\gpm}$ so that (\ref{b1b2}) is satisfied, completing the proof for this case. 

Now suppose $\mathbf{r}=(1,\ldots, 1)\in F^{m}$ and $w\neq w_{0}^{\gpg}$, so $\gpx(\mathbf{r})=\lambda$ by our notation. In this case there is a simple root $\alpha$ in $\gpg$ such that $w\gpn_{\alpha}w^{-1}\in \gpn_{\gpg}$. 

When $\alpha=e_{i}-e_{i+1}$ with $1\leq i \leq n-m-1$, note that $\lambda\in \gpm^{\gpj}$ stablizes $\psi_{\gpu}$, so $\psi_{\gpu}(\lambda^{-1}n_{\alpha}(t)\lambda)=\psi_{\gpu}(n_{\alpha}(t))\neq 1$ for some $t\in F$. On the other hand, $wn_{\alpha}(t)w^{-1}\in \gpn_{\gpg}$ by our assumption. So let $b_{1}=wn_{\alpha}(t)w^{-1}$ and $b_{2}=\lambda^{-1}n_{\alpha}(t)\lambda$ we have (\ref{b1b2}).

When $\alpha=e_{i}-e_{i+1}$ with $i\geq n-m$, we let $\mathbf{r}(t,i)=\gpx(1,\ldots, \overset{\text{i'-th}}{1},1+t,\ldots, 1)$, where $i'=i-(n-m)$.  Then for $t\neq -1$ we have
\begin{equation*}
  \begin{split}
    &wn_{\alpha}(t)\lambda=w\gpx(\mathbf{r}(t,i))n_{\alpha}(t)\\
=&(d_{m}^{-1}(\mathbf{r}(t,i))\,)^{w}w\lambda d_{m}(\mathbf{r}(t,i))n_{\alpha}(t).
  \end{split}
\end{equation*}
So let $b_{1}=(d_{m}^{-1}(\mathbf{r}(t,i))n_{\alpha}(t))^{w}$ and $b_{2}=d_{m}(\mathbf{r}(t,i))n_{\alpha}(t)$. For generic $(\chag,\cham)$, we can always find some $t\in F$ so that (\ref{b1b2}) is satisfied.

When $\alpha=2e_{n}$, we have
\begin{equation*}
  wn_{\alpha}(t)\lambda=w\lambda\gpy(t)\gpz(-t)n_{\alpha}(t).
\end{equation*}
So let $b_{1}=(n_{\alpha}(t))^{w}$ and $b_{2}=\gpy(t)\gpz(-t)n_{\alpha}(t)$, we have $(\chag\dha_{\gpb_{\gpg}})(b_{1})=1$, and we can find some $t\in F$ so that $\psi(\gpz(-t))\neq 1$, so (\ref{b1b2}) is satisfied.
\end{proof}

\section{Rationality(I)}
\label{sec:step3}
In this section we show that the pairing $\lpair$ defined in (\ref{defl}) can be extended to all generic $(\chag,\cham)$.
\begin{lem}
  Given $\fifi$ and $\fise$, the pairing $\lpair(\funf_{\chag}(\fifi),\funf_{\cham,\psi}(\fise))$ as defined in (\ref{defl}) can be extended as a rational function on $(\chag,\cham)$, and for generic $(\chag,\cham)$ it is the pairing between $\repfi$ and $\repse$ satisfying \textrm{Condition A}.
\end{lem}
\begin{proof}
  By lemma (\ref{lemflat}), which does not depend on the rationality of $\lpair$, we know that for $(\chag,\cham)\in \mathcal{Z}_{c}$,
\begin{equation}
\label{eq1}
  \pair{\funf_{\chag}(\mathrm{Ch}_{\gpi_{\gpg}})}{\funf_{\cham,\psi}(\mathrm{Ch}_{\gpi_{\gpm}})}=\mathrm{Vol}(\gpi_{\gpg})\mathrm{Vol}(\gpi_{\gpm}).
\end{equation}
Now we apply the Bernstein theorem. Note that by the last section, the pairing $\lpair$ satisfying \textrm{Condition A} and (\ref{eq1}) is unique for generic $(\chag,\cham)$. When $(\chag,\cham)\in \mathcal{Z}_{c}$, the pairing $\lpair$ defined by
  \begin{equation*}
    \pair{\funf_{\chag}(\fifi)}{\funf_{\cham,\psi}(\fise)}=\iiib(\fifi,\fise)
  \end{equation*}
satisfies these conditions. So it extends to a rational function in $(\chag,\cham)$ when $\fifi$ and $\fise$ are given. Moreover, it is regular at generic $(\chag,\cham)$ where the uniqueness is valid. 
\end{proof}

From now on we still denote by $l_{\chag,\cham,\psi}$ its own meromorphic continuation to all generic $(\chag,\cham)$, and similarly for $\iiib(g)$ and $\iii(g)$.

\section{Double coset decomposition}
\label{sec:step4}
Let $\wsfb(g)$ be a Whittaker-Shintani function. In this section we discuss the support of $\wsfb(g)$ on the double cosets $\gpz\gpu\gpk_{\gpm^{\gpj}}\backslash \gpg \slash \gpk_{\gpg}$. We denote by $g_{1}~g_{2}$ when $g_{1}$ and $g_{2}$ belong to the same double coset. We are going to show that
\begin{thm}
  \label{thm:step4}
  The support of $\wsfb(g)$ is contained in the double coset
  \begin{equation*}
    \bigcup_{\m d\in \Lambda^{+}_{m}, \m f\in \Lambda^{+}_{n}}\gpz \gpu \gpk_{\gpm^{\gpj}}(\w^{\m d}\lambda \w^{\m f})\gpk_{\gpg}.
  \end{equation*}

\end{thm}

First, by the Iwasawa decomposition on $\gpg$ and the Cartan decomposition on $\gpm$, we have
\begin{equation*}
  \gpg=\gpz\gpu\gpk_{\gpm^{\gpj}}(\gpx\gpy\gpt_{1}^{n-m}\gpt_{\gpm}^{+}) \gpk_{\gpg}.
\end{equation*}
Here $\gpt_{1}^{n-m}$ is the embedding of $\mathrm{GL}_{1}^{n-m}$ to $\gpg$ as
\begin{equation*}
  \gpt_{1}^{n-m}(t_{1},...,t_{n-m})=d_{n}(t_{1},...,t_{n-m}, 1,1,\ldots,1).
\end{equation*}
So we only need to consider the support of $\wsfb$ on the set $\gpx\gpy\w^{\m{(a;b)}}$ where $\mathbf{a}\in \mathbb{Z}^{n-m}$  and $\m b\in \Lambda^{+}_{m}$. We have
\begin{lem}
\label{lem:4.1}
   Suppose $\wsfb(xy\w^{\mathbf{(a,b)}})\neq 0$. Then $y\in \gpy^{0}$ and $\m a\in \Lambda_{n-m}$
\end{lem}
\begin{proof}
  The proof is given in lemma 2.1 in \cite{MR1121142}. First for any $z\in \gpz^{0}$, we have
  \begin{equation*}
    \wsfb(xy\w^{\mathbf{(a,b)}})=\wsfb(xy\w^{\mathbf{(a,b)}}z)=\psi(\w^{2a_{n-m}}z)\wsfb(xy\w^{\mathbf{(a,b)}}).
  \end{equation*}
  So $a_{n-m}\geq 0$.

  To show $y\in \gpy^{0}$, we argue by contradiction. Let $y=\gpy(y_{1},...,y_{m})$ with $|y_{j}|=\ww^{-r}$, $r>0$. We let $E_{\alpha}$ be the root subgroup in $\gpg$ of the root $\alpha$. Let $E_{\alpha}(t)$ be the canonical embedding of $F$ to $E_{\alpha}$. Define $E'_{\beta}(t)$ similarly on $\gpm$. Then for any $t\in \mathcal{O}^{*}$,
  \begin{equation*}
    \begin{split}
      &\wsfb(xy\w^{\mathbf{(a,b)}})=\wsfb(xy\w^{\mathbf{(a,b)}}E'_{-2e_{j}}(\w^{2b_{j}+r}t))\\
      =&\wsfb(xyE'_{-2e_{j}}(\w^{r}t))\w^{\mathbf{(a,b)}})=\psi(\w^{r}ty_{j}^{2})\wsfb(xy\w^{\mathbf{(a,b)}})
    \end{split}
  \end{equation*}
But note that $\w^{r}ty_{j}^{2}\in w^{-r}\mathcal{O}^{*}$, one can choose $t\in \mathcal{O}^{*}$ so that $\psi(\w^{r}ty_{j}^{2})\neq 1$, which is a contradiction. So $y\in \gpy^{0}$. 

To show $\m a\in \Lambda^{+}_{n-m}$, we also argue by contradiction. Since we already have $a_{n-m}\geq 0$, we assume that $a_{i}<a_{i+1}$ for some $i\leq n-m-1$. Note that when $i\leq n-m-1$, $E_{e_{i}-e_{i+1}}\in \gpu$. For any $t\in \mathcal{O}$, we have
\begin{equation*}
  \begin{split}
    &\wsfb(xy\w^{\mathbf{(a,b)}})=\wsfb(xy\w^{\mathbf{(a,b)}}E_{e_{i}-e_{i+1}}(t))\\
    =&\psi(\w^{a_{i}-a_{i+1}}t)\wsfb(xy\w^{\mathbf{(a,b)}})
  \end{split}
\end{equation*}
But when $a_{i}<a_{i+1}$, we can always find some $t\in \mathcal{O}$ such that $\psi(\w^{a_{i}-a_{i+1}}t)\neq 1$, contradicting the assumption. So we have $\m a\in \Lambda^{+}_{n-m}$. 

\end{proof}

By this lemma the support of $\wsfb$ is on $\gpz\gpu\gpk_{\gpm^{\gpj}}(\gpx \gpt_{n-m}^{+}\gpt_{m}^{+}) \gpk_{\gpg}$. Next we narrow the support further on the $\gpx$-part. Note that if $v(x_{i})=c_{i}$ for $1\leq i \leq m$, then $\gpx((x_{1},\ldots,x_{m}))t~\lambda(\m c)t$ for any $t\in \gpt_{\gpg}$. So we only need to consider for which $\m c$ that $\lambda(\m c)\gpt_{n-m}^{+}\gpt_{m}^{+}$ is contained in the support. 
\begin{lem}
\label{lem:zero}
  Suppose $x=\lambda(\m c)$, then for any $t\in \gpt_{n}$,
  \begin{equation*}
    x\,t~\lambda(min(\m c,\m 0))t
  \end{equation*}

\end{lem}
\begin{proof}
  Note that $\lambda\in \gpk_{\gpm^{\gpj}}$, so $x\,t~\lambda x\,t\in \gpt_{\gpm}^{0}\lambda(\mathrm{min}(\m c,\m 0))t\gpt_{\gpm}^{0}$. 
\end{proof}

By this lemma, we just need to consider the support of $\wsfb$ on $x\w^{(\m a;\m b)}$ with $x=\lambda(-\m d)$ for some $\m d\geq \m 0$. 

\begin{lem}
\label{lem:4.2}
  Let $x=\lambda(\mathbf{-d})$ with $\m d\geq \m 0 $. Let $\m a\in \Lambda^{+}_{n-m}$ and $\m b\in \Lambda^{+}_{m}$. Suppose $ \wsfb(x\w^{\mathbf{(a,b)}})\neq 0$, then $\m d \leq \m b$.
\end{lem}
\begin{proof}
  Suppose not, so we let $d_{j}>b_{j}$ for some $j$. Let $t\in \mathcal{O}^{*}$, then
  \begin{equation*}
    \begin{split}
      &\wsfb(x\w^{\mathbf{(a,b)}})=\wsfb(x\w^{\mathbf{(a,b)}}E'_{2e_{j}}(t))=\wsfb(xE'_{2e_{j}}(\w^{2b_{j}}t)\w^{\mathbf{(a,b)}})\\=&\psi(\w^{2b_{j}-2d_{j}}t)\wsfb(Y(0,\ldots,0,\underbrace{\w^{2b_{j}-d_{j}}}_{{j-th}},0,\ldots,0)x\w^{\mathbf{(a,b)}})
     \end{split}
  \end{equation*}
By lemma (\ref{lem:4.1}), when $\wsfb(x\w^{\mathbf{(a,b)}})\neq 0$, we have $\w^{2b_{j}-d_{j}}\in \mathcal{O}$, so the last formula equals $\psi(\w^{2b_{j}-2d_{j}}t)\wsfb(x\w^{\mathbf{(a,b)}})$. But when $d_{j}>b_{j}$ we can choose some $t\in \mathcal{O}^{*}$ such that $\psi(\w^{2b_{j}-2d_{j}}t)\neq 1$, contradicting the assumption.
\end{proof}
Note that when $x=\lambda(-\m d)$ with $\m d \geq \m 0$, $x\w^{\mathbf{(a,b)}}=\w^{\m d}\lambda \w^{\mathbf{(a,b-d)}}$. Let $\mathbf{r}=\mathbf{b-d}$. So by the previous lemmas,  the support of $\wsfb$ is on the union of $\gpz\gpu\gpk_{\gpm^{\gpj}}(\w^{\m d}\lambda \w^{\mathbf{(a,r)}})\gpk_{\gpg}$ for all $\m a\in \Lambda^{+}_{n-m}$, $\m d\geq \m 0$, $\mathbf{r\geq 0}$ and $\mathbf{d+r\in \Lambda^{+}_{m}}$.  The following two lemmas help us to narrow our choice of $\mathbf{a,b,d}$ so that we get theorem (\ref{thm:step4}). We also need to use them in the later calculations for the Whittaker-Shintani function. 
\begin{lem}
  \label{lem:induction}
  Suppose $g=\w^{\m d}\lambda \w^{\mathbf{(a,r)}}$ with $\m a\in \Lambda^{+}_{n-m}$, $\m d \geq \m 0$, $\m r\geq \m 0$, and $\mathbf{d+r}\in \Lambda^{+}_{m}$, then
  \begin{enumerate}
    \item Suppose $a_{n-m}<r_{1}$. Let $\mathbf{\overline r}=(a_{n-m},r_{2},\ldots,r_{m})$. Then
  \begin{equation*}
    \w^{\m d}\lambda \w^{\mathbf{(a,r)}}~\w^{\mathbf{d+r-\overline r}}\lambda \w^{\mathbf{(a,\overline r)}}.
  \end{equation*}
Moreover, if $\m d\in \Lambda^{+}_{m}$, then so is $\mathbf{d+r-\overline r}$. We call the process from $\mathbf{(d;a,r)}$ to $\mathbf{(d+r-\overline r;a,\overline r)}$ \textbf{Operation 1}.
  \item Fix $i$, let $\mathbf{\ti r}=(\ti r_{1},...,\ti r_{m})$ where
    \begin{equation*}\ti r_{j}=
      \begin{cases}
        r_{i}&\text{ if }j>i \text{ and } r_{j}>r_{i}\\
        r_{j}& \text{ otherwise},
      \end{cases}
    \end{equation*} 
    then
    \begin{equation*}
      \w^{\m d}\lambda \w^{\mathbf{(a,r)}}~\w^{\mathbf{d+r-\ti r}}\lambda \w^{\mathbf{(a,\ti r)}}.
    \end{equation*}
 Moreover, if $a_{n-m}\geq r_{1}$, then $a_{n-m}\geq \ti r_{1}$; if $\m d\in \Lambda^{+}_{m}$, then so is $\mathbf{d+r-\ti r}$. For given $i$, we call the process from $\mathbf{(d;a,r)}$ to $\mathbf{(d+r-\ti r;a,\ti r)}$ \textbf{Operation (2,i)}.
  \item Given $i$, let $\mathbf{\ti d}=(\ti d_{1},...,\ti d_{m})$ where
    \begin{equation*}
      \ti d_{j}=
      \begin{cases}
        d_{i}&\text{ if }j<i \text{ and } d_{j}<d_{i}\\
        d_{j}&\text{ otherwise},
      \end{cases}
    \end{equation*}
then
\begin{equation*}
  \wsfb(\w^{\m d}\lambda \w^{\mathbf{(a,r)}})=\wsfb(\w^{\mathbf{\ti d}}\lambda \w^{\mathbf{(a,r+d-\ti d)}}).
\end{equation*}
 Moreover, if $\mathbf{(a,r)}\in \Lambda^{+}_{n}$, then so is $\mathbf{(a,r+d-\ti d)}$. For given $i$, we call the process from $\mathbf{(d;a,r)}$ to $\mathbf{(\ti d;a,r+d-\ti d)}$ \textbf{Operation (3,i)}
  \end{enumerate}
\end{lem}
\begin{proof}
  For part 1, note that when $a<r_{1}$,
  \begin{equation*}
    \begin{split}
     & \w^{\m d}\lambda \w^{\mathbf{(a,r)}}~\w^{\m d}\lambda \w^{\mathbf{(a,r)}}\lambda\\
      ~&\w^{\m d}\lambda((a_{n-m}-r_{1},0,\ldots,0))\w^{\mathbf{(a,r)}}~\w^{\mathbf{d+r-\overline r}}\lambda \w^{\mathbf{(a,\overline r)}},
    \end{split}
   \end{equation*}
so we have part 1. For part 2, we fix $i$, then we have
\begin{equation*}
  \begin{split}
    & \w^{\m d}\lambda \w^{\mathbf{(a,r)}}~\w^{\m d}\lambda \w^{\mathbf{(a,r)}}\cdot\prod_{j\,|j>i,r_{j}>r_{i}}E'_{e_{i}-e_{j}}(1)\\
    ~&\prod_{j\,|j>i,r_{j}>r_{i}}E'_{e_{i}-e_{j}}(\w^{\mathbf{d_{i}+r_{i}-d_{j}-r_{j}}})\w^{\m d}\lambda(\mathbf{\ti r-r})\w^{\mathbf{(a,r)}}~\w^{\mathbf{d+r-\ti r}}\lambda \w^{\mathbf{(a,\ti r)}}
  \end{split}
\end{equation*}
the rest of part 2 is correct since $\mathbf{d+r}\in \Lambda^{+}_{m}$.
Part 3 is similar to part 2. For fix $i$, we have
\begin{equation*}
  \begin{split}
    & \w^{\m d}\lambda \w^{\mathbf{(a,r)}}~\prod_{j|j<i,d_{j}<d_{i}}E'_{-e_{j}+e_{i}}(1)\w^{\m d}\lambda \w^{\mathbf{(a,r)}}\\
    ~&\w^{\m d}\lambda(\mathbf{d-\ti d})\w^{\mathbf{(a,r)}}\prod_{j|j<i,d_{j}<d_{i}}E'_{e_{i}-e_{j}}(\w^{\mathbf{d_{j}+r_{j}-d_{i}-r_{i}}})~\w^{\mathbf{\ti d}}\lambda \w^{\mathbf{(a,r+d-\ti d)}}.
  \end{split}
\end{equation*}
The rest of part 3 is correct since $\mathbf{d+r}\in \Lambda^{+}_{m}$ and $\ti r_{1}\leq r_{1}$.
\end{proof}
\begin{pro}
\label{lem:closure}
  Suppose $\m a\in \Lambda^{+}_{n-m}$, $\m d \geq \m 0$, $\m r\geq \m 0$ and $\mathbf{d+r}\in \Lambda^{+}_{m}$. Starting from $\mathbf{(d;a,r)}$, if we take the \textbf{Operation 1}, and then \textbf{Operation (2,i)}, for $i$ from $1$ to $m$, and then \textbf{Opeartion (3,i)}, for $i$ from $m$ to $1$, we will end the process with $\mathbf{(\overline d;a,\overline r)}$ such that
  \begin{enumerate}
  \item $\mathbf{\overline d+\overline r}=\mathbf{d+r}.$
  \item $\mathbf{\overline d}\geq \mathbf{d}$, so equivalently, $\mathbf{\overline r\leq r}$.
  \item $\mathbf{\overline d}\in \Lambda^{+}_{m}$, and $\mathbf{(a,\overline r)}\in \Lambda^{+}_{n}$.
  \item For any triple $\mathbf{(d';a,r')}$ satisfying (1),(2) and (3) above, $\mathbf{\overline d\leq d'}$.
  \end{enumerate}
\end{pro}

In other words, for $\mathbf{(d;a,r)}$ such that $\m a\in \Lambda^{+}_{n-m}$, $\m d \geq \m 0$, $\m r\geq \m 0$ and $\mathbf{d+r}\in \Lambda^{+}_{m}$, there exists $\mathbf{(\overline d;a,\overline r)}$ satisfying (1),(2) and (3) above such that $\w^{\m d}\lambda \w^{\mathbf{(a,r)}}~\w^{\mathbf{\overline d}}\lambda \w^{\mathbf{(a,\overline r)}}$. Moreover, among all the triples satisfying (1), (2) and (3), $\mathbf{(\overline d;a,\overline r)}$ is the one with the smallest $\mathbf{\overline d}$.  

\begin{proof}
  Part (1) is obvious. Part (2) is true because in all the operations, either $\m d$ increases or $\m r$ decreases, and by (1) they are equivalent. Part (3) is true because by opeartion 1, $a_{n-m}\geq \overline r_{1}$, and after operation (2,i), $\overline r_{i}\geq \overline r_{j}$ for all $j>i$, and after operation (3,i), $\overline d_{i}\leq \overline d_{j}$ for all $j<i$. 

To prove part (4), we let $\mathbf{(d';a,r')}$ satisfy (1),(2) and (3). Note that  $\mathbf{(\overline d;a,\overline r)}$ is the result of $2i+1$ operations on $\mathbf{(d;a,r)}$. We show below that for any triple  $\mathbf{(D;a,R)}$ satisfying $\mathbf{D+R=d'+r'}$ and $\mathbf{D\leq d'}$, it still satisfies the same conditions after one of the operations in lemma (\ref{lem:induction}). Repeating it for $2i+1$ times we prove (4) since initially we have $\m d \leq \mathbf{d'}$.

 To be precise, let $\mathbf{(D;a,R)}$ be such a triple, and suppose after one of the operations in lemma (\ref{lem:induction}) it becomes $\mathbf{(D';a,R')}$. We need to show $\mathbf{D'+R'=d'+r'}$ and $\mathbf{D'\leq d'}$. The former one is obvious by the definition of the operations. For the latter first suppose the operation we took is (3,i), then for any $j$, either $\mathbf{ D'_{j}=D_{j}\leq d'_{j}}$, or $\mathbf{D'_{j}=D_{i}}$, in which case $j<i$, implying  $\mathbf{ D'_{j}=D_{i}\leq d'_{i}\leq d'_{j}}$. So $\mathbf{ D'_{j}\leq d'_{j}}$ anyway and hence $\mathbf{ D'\leq d'}$. If the operation we took is (1) or (2,i), we can similarly prove that $\mathbf{R'\geq r'}$, which implies, by $\mathbf{D'+R'=d'+r'}$, that $\mathbf{D'\leq d'}$. So in any case $\mathbf{D'\leq d'}$. 
\end{proof}
Following lemma (\ref{lem:induction}) and (\ref{lem:closure}), theorem (\ref{thm:step4}) is implied.

\section{Vectors invariant under certain open compact subgroups}
\label{sec:step5}

\subsection{ The $\gpi_{\gpg}$-invariant vectors in $\repfi$}\hfill\\

Let $\gpi_{\gpg}$ be the Iwahori subgroup of $\gpg$. For $w\in W_{\gpg}$, let $\iwag^{w\chag}$ be the element in $\mathrm{Ind}_{\gpb_{\gpg}}^{\gpg}(w\chag)$ defined as
\begin{equation*}
  \iwag^{w\chag}(g)=\int_{\gpb_{\gpg}}\ud_{l}b\ \mathrm{Ch}_{\gpi_{\gpg}}(bg)(w\chag)^{-1}\dha_{\gpb_{\gpg}}(b).
\end{equation*}
Then $\iwag^{w\chag}$ is $\mathrm{I}_{\gpg}$-invariant. And we have
\begin{equation*}
  \iwag^{w\chag}(1)=\mathrm{Vol}(\gpb_{\gpg}\cap \mathrm{I}_{\gpg})=1.
\end{equation*}
For a Weyl element $w$ and a character ${\chag}$ on the $\gpt_{\gpg}$, let $\mathrm{T}_{w,\chag}$ be the Intertwining operator from $\mathrm{Ind}^{\gpg}_{\gpb_{\gpg}}(\chag)$ to $\mathrm{Ind}^{\gpg}_{\gpb_{\gpg}}(w\chag)$ defined as

\begin{equation*}
  \mathrm{T}_{w,\chag}\mathrm{f}(g)=\int_{\gpn\cap w\gpn w^{-1}}\ud n\ \mathrm{f}(w^{-1}ng).
\end{equation*}
The integral is convergent when $Re(\chag)$ is sufficiently large, and by \cite{casselman's_notes} it continues holomorphically to generic $\chag$. We write $\mathrm{T}_{w,\chag}$ as $\mathrm{T}_{w}$ when there is no risk of confusion. For generic ${\chag}$, the $\gpg$-intertwining operator from $\mathrm{Ind}^{\gpg}_{\gpb_{\gpg}}(\chag)$ to $\mathrm{Ind}^{\gpg}_{\gpb_{\gpg}}(w\chag)$ is unique for every $w\in \gpw_{\gpg}$. So
\begin{equation*}
  \mathrm{T}_{w_{1}}\circ \mathrm{T}_{w_{2}}=c\cdot \mathrm{T}_{w_{1}w_{2}}
\end{equation*}
for some constant $c$. Moreover,if we let 
\begin{equation*}
  \begin{split}
    &    c_{\alpha}(\chag)=\frac{\zeta(\ppair{\chag}{\check{\alpha}})}{\zeta(\ppair{\chag}{\check{\alpha}}+1)}\\
    &c_{w}(\chag)=\prod_{\alpha>0,w\alpha<0}c_{\alpha}(\chag),
  \end{split}
\end{equation*}
then by theorem 3.1 in \cite{MR571057} 
\begin{equation*}
  \mathrm{T}_{w}\funf_{\chag}^{0}=c_{w}(\chag)\funf_{w\chag}^{0}.
\end{equation*}
So if we let
\begin{equation*}
  \overline{\mathrm{T}}_{w,\chag}=c_{w}(\chag)^{-1}\mathrm{T}_{w,\chag},
\end{equation*}
then
\begin{equation*}
  \overline{\mathrm{T}}_{w_{1},w_{2}\chag}\overline{\mathrm{T}}_{w_{2},\chag}=\overline{\mathrm{T}}_{w_{1}w_{2},\chag}
\end{equation*}

Following section 5 in \cite{MR581582} we state the following results without proof.
\begin{lem}
   Elements in $\{\overline{\mathrm{T}}_{w^{-1}}\iwag^{w\chag}\}_{w\in W_{\gpg}}$ form a basis of $\mathrm{Ind}^{\gpg}_{\gpb_{\gpg}}(\chag)^{\gpi_{\gpg}}$.
\end{lem}
By this lemma, we have 
\begin{lem}
one can express $\funf_{\chag}^{0}$ as a linear combination of $\{\overline{\mathrm{T}}_{w^{-1}}\iwag^{w\chag}\}_{w\in W_{\gpg}}$. In fact we have
\begin{equation*}
  \funf_{\chag}^{0}=\mathrm{Vol(I_{\gpg})^{-1}}\sum_{w\in W_{\gpg}}c_{w_{0}}(w\chag)\overline{\mathrm{T}}_{w^{-1}}\iwag^{w\chag}.
\end{equation*}
\end{lem}  

Next we have
\begin{pro}
\label{pro:aG}
  For $a\in \mathrm{\mathrm{T}}^{-}_{G}$,
  \begin{equation}
\label{eq:1}
     R(\mathrm{Ch}_{I_{\gpg}aI_{\gpg}})\iwag^{\chag}=\mathrm{Vol(I_{\gpg})}\chag\dnh(a)\iwag^{\chag}.
  \end{equation}
So as a consequence,
\begin{equation}
\label{eq:step5}
R(\mathrm{Ch}_{\gpi_{\gpg}a\gpi_{\gpg}})\funf_{\chag}^{0}=\sum_{w\in \gpw_{\gpg}}c_{w_{0}}(w\chag)(w\chag)\dnh(a)\overline{\mathrm{T}}_{w^{-1}}\iwag^{w\chag}.
\end{equation}
\end{pro}

\subsection{The $\overline \gpi_{\gpm}$-Invariant vectors in $\repse$.}

\hfill\\
\vspace{6pt}

A similar discussion can be applied to $\gpm^{\gpj}$ with some modification. Let $\mathrm{I}_{\gpm}$ be the Iwahori subgroup of $\gpm$, and let $\overline{\gpi}_{\gpm}=\gpi_{\gpm}\ltimes \gpj^{0}$. Consider the space
\begin{equation*}
  \repse^{\overline{\gpi}_{\gpm}}.
\end{equation*}
We have the following lemma.
\begin{lem}
\label{lem:4.2.1}
Any $f\in \repse^{\overline{\gpi}_{\gpm}}$ is supported on $\gpb_{\gpm^{\gpj}}\gpw_{\gpm}\overline{\gpi}_{\gpm}$.
\end{lem}
\begin{proof}

Since \begin{equation*}
  \gpm=\gpb_{\gpm}\gpw_{\gpm}\gpi_{\gpm}, 
\end{equation*}
we have 
\begin{equation*}
  \gpm^{\gpj}=\gpb_{\gpm^{\gpj}}\gpx \gpw_{\gpm}\gpi_{\gpm}.
\end{equation*}
Suppose $\mathrm{f}\in \repse^{\overline{\gpi}_{\gpm}}$ with $\mathrm{f}(\gpx(x)w)\neq 0$, then since $\gpw_{\gpm}$ normalizes $\gpj^{0}$, we have
\begin{equation*}
  f(\gpx(x)w)=f(\gpx(x)\gpy(y)w)
\end{equation*}
for any $y\in \mathcal{O}^{m}$. But then
\begin{equation*}
  \begin{split}
    &    f(\gpx(x)\gpy(y)w)\\=
    &f(\gpy(y)\gpx(x)\gpz(2\ppair{x}{y})w)\\=
    &\psi(\ppair{x}{y})f(\gpx(x)w).
  \end{split}
\end{equation*}
So $\psi(\ppair{x}{y})=1$ for any $y\in \mathcal{O}^{m}$, which means $x\in \mathcal{O}^{m}$, completing our proof since $\gpb_{\gpm^{\gpj}}\gpx^{0}\gpw_{\gpm}\gpi_{\gpm}=\gpb_{\gpm^{\gpj}}\gpw_{\gpm}\overline\gpi_{\gpm}$.
\end{proof}

By this lemma, we have
\begin{equation*}
  \mathrm{dim\, Ind}^{\gpm_{\gpj}}_{\gpb_{\gpm^{\gpj}}}(\cham, \psi)^{\overline{\gpi}_{\gpm}}\leq \mathrm{Card(W_{\gpm})}.
\end{equation*}
For any character $\cham$ on $\gpt_{\gpm}$, let $\mathrm{T}_{w}^{\cham,\psi}$ be the intertwining operator from $\mathrm{Ind}^{\gpm_{\gpj}}_{\gpb_{\gpm^{\gpj}}}(\cham, \psi)$ to $\mathrm{Ind}^{\gpm_{\gpj}}_{\gpb_{\gpm^{\gpj}}}(w\cham, \psi)$ defined as
\begin{equation*}
 \mathrm{T}_{w}^{\cham,\psi}(f)(g)= \int_{\gpn^{\gpj}\cap w\gpn^{\gpj}w^{-1}\backslash \gpn^{\gpj}}f(w^{-1}ng)\ud n.
\end{equation*}
Similar to the previous subsection, the integral is convergent when $Re(\cham)$ is sufficiently large, and continues holomorphically to generic $\cham$. In fact the only difference is that we are integrating on part of $\gpx$, on which smooth elements in $\repse$ is compactly supported. For generic $\cham$, the $\gpm^{\gpj}$-intertwining operator from  $Ind^{\gpm_{\gpj}}_{\gpb_{\gpm^{\gpj}}}(\cham, \psi)$ to $Ind^{\gpm_{\gpj}}_{\gpb_{\gpm^{\gpj}}}(w\cham, \psi)$ is unique, so we have
\begin{equation*}
\mathrm{T}_{w_{1}}^{w_{2}\cham,\psi}\circ \mathrm{T}_{w_{2}}^{\cham,\psi}=c\cdot \mathrm{T}_{w_{1}w_{2}}^{\cham,\psi}
\end{equation*}
for some constant $c$. By a similar method as in theorem 3.1 in \cite{MR571057} we have the following result.
\begin{lem}
  For $\alpha$ being a simple root in $\gpm$, if we let
\begin{equation*}
  \ti c_{\alpha}(\cham)=\frac{\zeta(\ppair{\cham}{\alpha})}{\zeta(\ppair{\cham}{\alpha}+1)},
\end{equation*}
and let
\begin{equation*}
  \ti c_{w}(\cham)=\prod_{\alpha>0,w\alpha<0}\ti c_{\alpha}(\cham),
\end{equation*}
then,
\begin{equation*}
  \mathrm{T}_{w}^{\cham,\psi}(\funf_{\cham,\psi}^{0})(\mathrm{e})=\ti c_{w}(\cham)
\end{equation*}

\end{lem}
\hfill\\

So if we define
\begin{equation*}
  \overline{\mathrm{T}}_{w}^{\cham,\psi}=(\ti c_{w}(\cham))^{-1}\mathrm{T}_{w}^{\cham,\psi},
\end{equation*}
then
\begin{equation*}
\overline{\mathrm{T}}_{w_{1}}^{w_{2}\cham,\psi}\circ \overline{\mathrm{T}}_{w_{2}}^{\cham,\psi}=\overline{\mathrm{T}}_{w_{1}w_{2}}^{\cham,\psi}
\end{equation*}
Let
\begin{equation*}
  \iwam^{\cham,\psi}=\funf_{\cham,\psi}(\mathrm{Ch}_{\overline\gpi_{\gpm}}),
\end{equation*}
then we have
\begin{equation*}
  \iwam^{\cham,\psi}(g)=
  \begin{cases}
    \cham\psi\dha_{\gpb_{\gpm^{\gpj}}}(b)&\text{if }g=bn^{-}x,\, b\in \gpb_{\gpm^{\gpj}}^{0},\, n^{-}\in \gpn^{-,1}_{\gpm},\, x\in \gpx^{0}.\\
    0 &\text{if } g\notin \gpb_{\gpm^{\gpj}}^{0} \gpn^{-,1}_{\gpm}\gpx^{0}
  \end{cases}
\end{equation*}

\begin{lem}
  The set $\{\mathrm{\overline T}_{w^{-1}}\iwam^{w\cham,\psi}\}_{w\in W_{\gpm}}$ forms a basis of $\mathrm{Ind}^{\gpm_{\gpj}}_{\gpb_{\gpm^{\gpj}}}(\cham, \psi)^{\overline{\gpi}_{\gpm}}$
\end{lem}
\begin{proof}
  Since  $\mathrm{dim\, Ind}^{\gpm_{\gpj}}_{\gpb_{\gpm^{\gpj}}}(\cham, \psi)^{\overline{\gpi}_{\gpm}}\leq \mathrm{Card}(\mathrm{W}_{\gpm})$, it suffices to prove that elements in $\{\mathrm{\overline T}_{w^{-1}}\iwam^{w\cham,\psi}\}_{w\in W_{\gpm}}$ are linear indepedent. Consider
 \begin{equation*}
    \mathrm{\overline T}_{w}(\iwam^{\cham,\psi})(w_{0}^{\gpm})=(\ti c_{w}(\cham))^{-1}
\int_{\gpn^{\gpj}\cap w\gpn^{\gpj}w^{-1}\backslash \gpn^{\gpj}}\ud n^{\gpj}\iwam^{\cham,\psi}(w^{-1}n^{\gpj}w_{0}^{\gpm})
  \end{equation*}
By the definition of $\iwam^{\cham,\psi}$, the integral is non-zero implies
\begin{equation*}
  w^{-1}\gpn^{\gpj}w_{0}^{\gpm}\cap \gpb_{\gpm^{\gpj}}\gpn_{1,\gpm}^{-}\gpx^{0}\neq \emptyset,
\end{equation*}
which implies
\begin{equation*}
  w^{-1}\gpn^{\gpj}\cap \gpb_{\gpm^{\gpj}}\gpn_{1,\gpm}^{-}\gpx^{0}(w_{0}^{\gpm})^{-1}\neq \emptyset.
\end{equation*}
Note that $w^{-1}\gpn^{\gpj}\subseteq \gpb_{\gpm}w^{-1}\gpn_{\gpm}\ltimes \gpj$, and $\gpb_{\gpm^{\gpj}}\gpn_{1,\gpm}^{-}\gpx^{0}(w_{0}^{\gpm})^{-1}\subseteq \gpb_{\gpm}(w_{0}^{\gpm})^{-1}\gpn_{\gpm}\ltimes \gpj$,
so by the Bruhat-decomposition on $\gpm$ we have
\begin{equation*}
  w=w_{0}^{\gpm}
\end{equation*}
if $\mathrm{\overline T}_{w}(\iwam^{\cham,\psi})(w_{0}^{\gpm})\neq 0$. On the other hand, when $w=w_{0}^{\gpm}$,
\begin{equation*}
  \mathrm{\overline T}_{w_{0}^{\gpm}}(\iwam^{\cham,\psi})(w_{0}^{\gpm})=(\ti c_{w_{0}^{\gpm}}(\cham))^{-1}\int_{\gpn^{-}\gpx}\ud n \ud x \,\iwam^{\cham,\psi}(nx)
\end{equation*}
which equals
\begin{equation*}
  (\ti c_{w_{0}^{\gpm}}(\cham))^{-1}\mathrm{Vol}(\gpi_{\gpm})
\end{equation*}
by the definition of $\iwam^{\cham,\psi}$. Now let  $\mathrm{\overline T}_{w_{0}^{\gpm}w}$ acts on all elements in $\{\mathrm{\overline T}_{w^{-1}}f_{1}^{w\cham,\psi}\}_{w\in W_{\gpm}}$ and evaluate them at $w_{0}^{\gpm}$. Only $\mathrm{\overline T}_{w_{0}^{\gpm}w}\circ\mathrm{\overline T}_{w^{-1}}f_{1}^{w\cham,\psi}$ is non-zero. So they are linear independent, and hence form a basis of $\mathrm{Ind}^{\gpm^{\gpj}}_{\gpb_{\gpm^{\gpj}}}(\cham,\psi)^{\overline I_{\gpm}}.$
\end{proof}

So then we have
\begin{lem}
We can write $\funf_{\chag,\psi}^{0}$ as a linear combination of $\{\mathrm{\overline T}_{w^{-1}}\iwam^{w\cham,\psi}\}_{w\in W_{\gpm}}$ as
\begin{equation}
\label{eq:4.2.1}
\funf_{\chag,\psi}^{0}=\mathrm{Vol}(\gpi_{\gpm})^{-1}\sum_{w\in \gpw_{\gpm}}\ti c_{w_{0}^{\gpm}}(w\chag)\mathrm{\overline T}_{w^{-1}}\iwam^{w_{\cham,\psi}}.
\end{equation}
\end{lem}

\begin{proof}
Since $\{\mathrm{\overline T}_{w^{-1}}\iwam^{w\cham,\psi}\}_{w\in W_{\gpm}}$ is the basis of $\mathrm{Ind}^{\gpm^{\gpj}}_{\gpb_{\gpm^{\gpj}}}(\cham,\psi)^{\overline I_{\gpm}}$, we assume
\begin{equation*}
  \funf_{\chag,\psi}^{0}=\sum_{w\in \gpw_{\gpm}}b_{w}\mathrm{\overline T}_{w^{-1}}\iwam^{w\cham,\psi}.
\end{equation*}
for some $b_{w}\in \mathbb{C}$. Let $\mathrm{\overline T}_{w_{0}^{\gpm}w}$ acts on both sides of the equation and take the value at $w_{0}^{\gpm}$, we obtain that
\begin{equation*}
  1=b_{w}(\ti c_{w_{0}^{\gpm}}(w\cham))^{-1}\mathrm{Vol}(\gpi_{\gpm}),
\end{equation*}
completing our proof.
\end{proof}

Next we consider the action $R(\mathrm{Ch}_{\overline \gpi_{\gpm}}a\mathrm{Ch}_{\overline\gpi_{\gpm}})$ on $\funf_{\chag,\psi}^{0}$ for $a\in \gpt_{\gpm}^{-}$. We have
\begin{pro}
  \label{pro:aM}
  \begin{equation*}
  R(\mathrm{Ch}_{\overline \gpi_{\gpm}}a\mathrm{Ch}_{\overline\gpi_{\gpm}})\funf_{\chag,\psi}^{0}=\sum_{w\in W_{\gpm}}\ti c_{w_{0}^{\gpm}}(w\cham)\cdot (w\cham)\dnh_{\gpb_{\gpm^{\gpj}}}(a)\cdot \mathrm{\overline T}_{w^{-1}}\iwam^{w\cham,\psi}
\end{equation*}
\end{pro}

\begin{proof}
Consider $R(\mathrm{Ch}_{\overline\gpi_{\gpm}}a\mathrm{Ch}_{\overline\gpi_{\gpm}})\iwam^{\cham,\psi}$. Note that it belongs to $\repse^{\overline{\gpi}_{\gpm}}$, so by lemma (\ref{lem:4.2.1}), we only need to consider its value on $\gpw_{\gpm}$. Note that when $a\in \gpt_{\gpm}^{-}$, we have the decomposition
\begin{equation*}
  \overline\gpi_{\gpm}a\overline\gpi_{\gpm}=\gpn_{\gpm}^{-,1}\gpx^{0}a\gpb_{\gpm^{\gpj}}^{0}.
\end{equation*}
Since $\iwam^{\cham,\psi}$ is $\overline \gpi_{\gpm}$-invariant, and $\mathrm{Vol}(\gpx^{0})=\mathrm{Vol}(\gpb_{\gpm^{\gpj}}^{0})=1$, so 
\begin{equation*}
  R( \overline\gpi_{\gpm}a \overline\gpi_{\gpm})\iwam^{\cham,\psi}(w)=\int_{N_{\gpm}^{-,1}\gpx^{0}}\ud n \iwam^{\cham,\psi}(wnxa)
\end{equation*}
Suppose it is non-zero, then by considering the support of $\iwam^{\cham,\psi}$, 
\begin{equation*}
  w\gpn_{1,\gpm}^{-}a\gpx\cap \gpb_{\gpm^{\gpj}}^{0}\gpn_{\gpm}^{-,1}\gpx^{0} \neq \emptyset
\end{equation*}
so
\begin{equation*}
  w\gpn_{1,\gpm}^{-}a\gpx w_{0}^{\gpm}\cap \gpb_{\gpm^{\gpj}}^{0}w_{0}^{\gpm}\gpn_{\gpm}^{1}\gpy^{0} \neq \emptyset
\end{equation*}
Note that
\begin{equation*}
  w\gpn_{1,\gpm}^{-}a\gpx w_{0}^{\gpm}\subseteq \gpb_{\gpm}ww_{0}^{\gpm}\gpn_{\gpm}\ltimes \gpj
\end{equation*}
and
\begin{equation*}
  \gpb_{\gpm^{\gpj}}^{0}w_{0}^{\gpm}\gpn_{\gpm}^{1}\gpy^{0}\subseteq \gpb_{\gpm}w_{0}^{\gpm}\gpn_{\gpm}\ltimes \gpj.
\end{equation*}
So by Bruhat decomposition of $\gpm$ we have
\begin{equation*}
  w=\mathrm{e}.
\end{equation*}
This implies that $R(\overline\gpi_{\gpm}a \overline\gpi_{\gpm})\iwam^{\cham,\psi}$ is propotional to $\iwam^{\cham,\psi}$. Consider $R(\overline\gpi_{\gpm}a \overline\gpi_{\gpm})\iwam^{\cham,\psi}(\mathrm{e})$, which is equal to 
\begin{equation*}
  \int_{N_{\gpm}^{-,1}\gpx^{0}}\ud n \ud x \iwam^{\cham,\psi}(nxa).
\end{equation*}
Note that $\iwam^{\cham,\psi}\in \mathrm{Ind}_{\gpb_{\gpm^{\gpj}}}^{\gpm^{\gpj}}(\cham,\psi)$, the integral is equal to
\begin{equation*}
  \cham\dha_{\gpb_{\gpm^{\gpj}}}(a)\int_{\gpn_{\gpm}^{-,1}\gpx^{0}}\ud n \ud x \, \iwam^{\cham,\psi}(a^{-1}nxa).
\end{equation*}
Considering the support of $\iwam^{\cham,\psi}$,
\begin{equation*}
  \cham\dha_{\gpb_{\gpm^{\gpj}}}(a)\int_{\gpn_{\gpm}^{-,1}\gpx^{0}}\ud n  \, \iwam^{\cham,\psi}(a^{-1}nxa)=\cham\dha_{\gpb_{\gpm^{\gpj}}}(a)\mathrm{Vol}(\gpn_{\gpm}^{-,1}\gpx^{0}\cap a\gpn_{1,\gpm}^{-}\gpx^{0}a^{-1})
\end{equation*}
When $a\in \gpt_{\gpm}^{-}$,  $a\gpn_{1,\gpm}^{-}\gpx^{0}a^{-1}\in \gpn_{\gpm}^{-,1}\gpx^{0}$, so
\begin{equation*}
  \begin{split}
    &    \mathrm{Vol}(\gpn_{\gpm}^{-,1}\gpx^{0}\cap a\gpn_{1,\gpm}^{-}\gpx^{0}a^{-1})\\
    =&\mathrm{Vol}(a\gpn_{1,\gpm}^{-}\gpx^{0}a^{-1})\\
    =&\dne_{\gpb_{\gpm^{\gpj}}}(a)\mathrm{Vol}(\gpi_{\gpm}).
  \end{split}
\end{equation*}
So we have
\begin{equation*}
  R(\overline\gpi_{\gpm}a \overline\gpi_{\gpm})\iwam^{\cham,\psi}=\cham\dnh_{\gpb_{\gpm^{\gpj}}}(a)\mathrm{Vol}(\gpi_{\gpm})\iwam^{\cham,\psi}.
\end{equation*}

Applying this to both sides of equation (\ref{eq:4.2.1}), our proposition is proved.
\end{proof}

\section{$\gamma$-factor}
\label{sec:step6}

In this section we assume $(\chag,\cham)$ is generic. We are going to show that
\begin{thm}
\label{thm:gamma}
  Let $\chag=(\chag_{1},...,\chag_{n})$ and $\cham=(\cham_{1},...\cham_{m})$. Let $\Gamma(\chag,\cham)$ be a function on $(\chag,\cham)$ given by
  \begin{equation*}
    \begin{split}
      \Gamma(\chag,\cham)&=\prod_{1\leq a < b \leq n}\zeta^{-1}(\chag_{a}-\chag_{b}+1)\zeta^{-1}(\chag_{a}+\chag_{b}+1)\cdot \prod_{i=1}^{n}\zeta^{-1}(\chag_{i}+1)\\
&\cdot \prod_{1\leq a < b \leq n}\zeta^{-1}(\cham_{a}-\cham_{b}+1)\zeta^{-1}(\cham_{a}+\cham_{b}+1)\cdot \prod_{j=1}^{m}(1+\cham_{j}(\w)|\w|^{\frac 1 2})\\
&\cdot \prod_{j=1}^{m}\prod_{i=1}^{(n-m)+j-1}\dfrac{\zeta(\chag_{i}-\cham_{j}+\frac 1 2)}{\zeta(-\chag_{i}+\cham_{j}+\frac 1 2) }\cdot \prod_{i=1}^{n}\prod_{j=1}^{m}\zeta(\chag_{i}+\cham_{j}+\frac 1 2)\zeta(-\chag_{i}+\cham_{j}+\frac 1 2)
  \end{split}  
  \end{equation*}
Then for any fixed $g\in \gpg$,
\begin{equation*}
  \frac{\lpair(R(g)\funf_{\chag}^{0},\funf_{\cham,\psi}^{0})}{\Gamma(\chag,\cham)}
\end{equation*}
is $\gpw_{\gpg}\times \gpw_{\gpm}$-invariant. 

\end{thm}

If we can prove that for $(w,w')=(w_{\alpha},1)$ and $(w,w')=(1,w_{\beta})$ where $\alpha$ is a simple root in $\gpg$ and $\beta$  is a simple root in $\gpm$,
\begin{equation*}
  \frac{\Gamma(w\chag,w'\cham)}{\Gamma(\chag,\cham)}=\frac{l_{w\chag,w'\cham,\psi}(R(g)\funf_{w\chag}^{0},\funf_{w'\cham,\psi}^{0})}{\lpair(R(g)\funf_{\chag}^{0},\funf_{\cham,\psi}^{0})}
\end{equation*}
then theorem (\ref{thm:gamma}) is implied. Since we can calculate the left hand side above directly, we only need to consider the ratio
\begin{equation}
  \label{eq:6.1}
  \frac{l_{w\chag,w'\cham,\psi}(R(g)\funf_{w\chag}^{0},\funf_{w'\cham,\psi}^{0})}{\lpair(R(g)\funf_{\chag}^{0},\funf_{\cham,\psi}^{0})}.
\end{equation}
We obtain its value by the uniqueness of the pair $\lpair$. To be precise, for $(w,w')\in \gpw_{\gpg}\times \gpw_{\gpm}$, let
\begin{equation*}
  \ti l_{\chag,\cham,w,w'}(\funf_{\chag},\funf_{\cham,\psi})=l_{w\chag,w'\cham,\psi}(\mathrm{T}_{w}^{\chag}\funf_{\chag},\mathrm{T}_{w'}^{\cham,\psi}\funf_{\cham,\psi}).
\end{equation*}
Then $\ti l_{\chag,\cham,w,w'}$ is also a pairing on $\repfi\otimes \repse$ satisfying \textrm{Condition A}. By the uniqueness of such pairing, there exists a constant, which we denote by $\gamma(\chag,\cham,w,w')$, such that
\begin{equation*}
  \ti l_{\chag,\cham,w,w'}=\gamma(\chag,\cham,w,w') \lpair. 
\end{equation*}
Then, since  $\mathrm{T}_{w}^{\chag}\funf_{\chag}^{0}=c_{w}(\chag)\funf_{w\chag}^{0}$ and $  \mathrm{T}_{w'}^{\cham,\psi}\funf_{\cham,\psi}^{0}=\ti c_{w'}(\cham)\funf_{w'\cham,\psi}^{0}$, we have
\begin{equation*}
  \begin{split}
    &    \frac{l_{w\chag,w'\cham,\psi}(R(g)\funf_{w\chag}^{0},\funf_{w'\cham,\psi}^{0})}{\lpair(R(g)\funf_{\chag}^{0},\funf_{\cham,\psi}^{0})}\\
    =&c_{w}^{-1}(\chag)\ti c_{w'}^{-1}(\cham)\cdot \frac{\ti l_{\chag,\cham,w,w'}(R(g)\funf_{\chag}^{0},\funf_{\cham,\psi}^{0})}{\lpair(R(g)\funf_{\chag}^{0},\funf_{\cham,\psi}^{0})}\\
    =&c_{w}^{-1}(\chag)\ti c_{w'}^{-1}(\cham)\gamma(\chag,\cham,w,w')
  \end{split}
\end{equation*}
So in the rest of this section we calculate $\gamma(\chag,\cham,w_{\alpha},1)$ and $\gamma(\chag,\cham,1,w_{\beta})$. The result will be stated in theorem (\ref{thm:alpha}) and (\ref{thm:beta}), which implies the theorem (\ref{thm:gamma}).
\subsection{The calculation of $ \gamma(\chag,\cham,w_{\alpha},1)$}
\label{sec:step6.1}
\hfill 

Let $\gpi_{\gpm^{\gpj}}=\gpi_{\gpm}\ltimes(\gpx^{1}\gpy^{0}\gpz^{0})$, and let 
\begin{equation*}
  \funf_{\cham,\psi}^{1}=\funf_{\cham,\psi}(\mathrm{Ch}_{\gpi_{\gpm^{\gpj}}}).
\end{equation*}. 
For $w\in \gpw_{\gpg}$, let $\iwagw{\chag,w}=\funf_{\chag}(\mathrm{Ch}_{\mathrm{I_{\gpg}}w\mathrm{I_{\gpg}}})$. Then by theorem 3.4 in \cite{MR571057}
\begin{equation*}
  \mathrm{T}_{w_{\alpha}}(\iwagw{\chag,1}+\iwagw{\chag,w_{\alpha}})=c_{\alpha}(\chag)(\iwagw{w_{\alpha}\chag,1}+\iwagw{w_{\alpha}\chag,w_{\alpha}}).
\end{equation*}
So we have
\begin{equation}
\label{eqrankone}
        \gamma(\chag,\cham,w_{\alpha},1) =c_{\alpha}(\chag)\cdot \frac{l_{w_{\alpha}\chag,\cham,\psi
}(R(\lambda w_{0}^{\gpg})\circ(\iwagw{w_{\alpha}\chag,1}+\iwagw{w_{\alpha}\chag,w_{\alpha}}),\funf_{\cham,\psi}^{1})}{\lpair(R(\lambda w_{0}^{\gpg})\circ(\iwagw{\chag,1}+\iwagw{\chag,w_{\alpha}}),\funf_{\cham,\psi}^{1})}
\end{equation}

Let $i'=i-(n-m)$. The result of the calculation is stated as the proposition below:
\begin{pro}
\label{pro:alpha}
For generic $(\chag,\cham)$, the value of $ \lpair(R(\lambda w_{0}^{\gpg})\circ(\iwagw{\chag,1}+\iwagw{\chag,w_{\alpha}}),\funf_{\cham,\psi}^{1})$ equals
  \begin{equation*}
   \mathrm{Vol}(\gpi_{\gpg})\mathrm{Vol}(\gpi_{\gpm^{\gpj}})\ww^{-1}(1-(\chag_{i}\chag_{i+1}^{-1})(\w)\ww)
\end{equation*}
for $\alpha=e_{i}-e_{i+1},1\leq i \leq n-m-1$, and equals
\begin{equation*}
      \mathrm{Vol}(\gpi_{\gpg})\mathrm{Vol}(\gpi_{\gpm^{\gpj}})\cdot(\ww^{-1}-1)\cdot\dfrac{1-\ww(\chag_{i}\chag_{i+1}^{-1})(\w)}{(1-(\chag_{i}\cham_{i'+1}^{-1})(\w)\ww^{\frac{1}{2}})(1-(\chag_{i+1}^{-1}\cham_{i'+1})(\w)\ww^{\frac{1}{2}})}
\end{equation*}
for $\alpha=e_{i}-e_{i+1},n-m \leq i \leq n-1$, and equals
\begin{equation*}
  \mathrm{Vol}(\gpi_{\gpg})\mathrm{Vol}(\gpi_{\gpm^{\gpj}})\ww^{-1}\cdot(1-\chag_{n}(\w)\ww^{})
\end{equation*}
for $\alpha=2e_{n}$.
\end{pro}
Substituting this in (\ref{eqrankone}), we have 
\begin{thm}
\label{thm:alpha}
  For generic $(\chag,\cham)$, the $\gamma$-factor $\gamma(\chag,\cham,w_{\alpha},1)$ is equal to
\begin{equation*}
      c_{\alpha}(\chag)\cdot\dfrac{\zeta(\chag_{i}-\chag_{i+1}+1)}{\zeta(\chag_{i+1}-\chag_{i}+1)}
\end{equation*}
for $\alpha=e_{i}-e_{i+1},1\leq i \leq n-m-1$, and 
\begin{equation*}
      c_{\alpha}(\chag)\cdot\dfrac{\zeta(\chag_{i}-\chag_{i+1}+1)}{\zeta(\chag_{i+1}-\chag_{i}+1)}\dfrac{\zeta(\chag_{i+1}-\cham_{i'+1}+\frac{1}{2})\zeta(-\chag_{i}+\cham_{i'+1}+\frac{1}{2})}{\zeta(\chag_{i}-\cham_{i'+1}+\frac{1}{2})\zeta(-\chag_{i+1}+\cham_{i'+1}+\frac{1}{2})}
\end{equation*}
for $\alpha=e_{i}-e_{i+1},n-m \leq i \leq n-1$, and 
\begin{equation*}
      c_{\alpha}(\chag)\cdot\dfrac{\zeta(\chag_{n}+1)}{\zeta(-\chag_{n}+1)}
\end{equation*}
for $\alpha=2e_{n}$.  
\end{thm}

Recall that for $(\chag,\cham)\in \mathcal{Z}_{c}$, we have 
\begin{equation*}
    \pair{R(g)\funf_{\chag}(\fifi)}{\funf_{\cham,\psi}(\fise)}=\iiib(\fifi,\fise)(g).
\end{equation*}
First we calculate $\iiib(\mathrm{Ch}_{\gpi_{\gpg}},\mathrm{Ch}_{\gpi_{\gpm^{\gpj}}})(\lambda w_{0})$, in fact, 
\begin{lem}
  \label{lemflat}
For $(\chag,\cham)\in \mathcal{Z}_{c}$, we have 
  \begin{equation}
\label{flat}
    \whi_{\chag,\cham}(\mathrm{Ch}_{\gpi_{\gpg}},\mathrm{Ch}_{\gpi_{\gpm^{\gpj}}})(\lambda w_{0})=\mathrm{Vol}(\gpi_{\gpg})\mathrm{Vol}(\gpi_{\gpm}).
\end{equation}
\end{lem}
\begin{proof}
  By definition, 
  \begin{equation*}
 \iiib(\mathrm{Ch}_{\gpi_{\gpg}},\mathrm{Ch}_{\gpi_{\gpm^{\gpj}}})(\lambda w_{0})=\int_{\gpi_{\gpg}\times \gpi_{\gpm^{\gpj}}}\funk_{\chag,\cham,\psi}(xw_{0}\lambda^{-1}x')dxdx'
\end{equation*}
From now on we write $\funk_{\chag,\cham,\psi}$ by $\funk$ for simplicity. Note that $$\lambda^{-1}=d_{m}(-1,-1,...,-1)\lambda d_{}(-1,-1,...,-1)\in T_{\gpm}^{0}\lambda T_{\gpm}^{0},$$ so we have 
\begin{equation*}
  \int_{\gpi_{\gpg}\times \gpi_{\gpm^{\gpj}}}dxdx' \funk(xw_{0}\lambda^{-1}x')=\int_{\gpi_{\gpg}\times \gpi_{\gpm^{\gpj}}}dxdx' \funk(xw_{0}\lambda x').
\end{equation*}
The proposition is implied if $\funk(xw_{0}\lambda x')=1$ for all $x\in \gpi_{\gpg}$ and $x'\in \gpi_{\gpm^{\gpj}}$. To show this we note that $\gpi_{\gpg}=\gpb_{\gpg}^{0}N_{\gpg}^{-,1}$ and $\gpi_{\gpm^{\gpj}}=\gpx^{1}N^{-,1}_{\gpm}\gpb_{\gpm}^{0}\gpy^{0}\gpz^{0}$. So by the definition of $\funk$, we only need to show
\begin{equation*}
   \funk(n'w_{0}\lambda xn^{-})=1
\end{equation*}
for $n'\in N^{-,1}_{\gpg}$, $x\in \gpx^{1}$ and $n^{-}\in N^{-,1}_{\gpm}$. Note that when $x\in \gpx^{1}$, we have $\lambda x\in T_{\gpm}^{0}\lambda T_{\gpm}^{0}$, and note that $T_{\gpm}^{0}$ normalizes both $N_{\gpg}^{-,1}$ and $N_{\gpm}^{-,1}$, so it reduces to show that 
\begin{equation*}
  \funk(n'w_{0}\lambda n^{-})=1
\end{equation*}
for $n'\in N^{-,1}_{\gpg}$, and $n^{-}\in N^{-,1}_{\gpm}$. The proof for this is similar to that in lemma (\ref{lem:basic}), so we just skip it here.
\end{proof}
Now we consider the calculation of $\iiib(\mathrm{Ch}_{\gpi_{\gpg}w_{\alpha}\gpi_{\gpg}},\mathrm{Ch}_{\gpi_{\gpm^{\gpj}}})(\lambda w_{0})$.
First, by a similar method as in lemma (\ref{lem:basic}), we have
\begin{equation*}
  \iiib(\mathrm{Ch}_{\gpi_{\gpg}w_{\alpha}\gpi_{\gpg}},\mathrm{Ch}_{\gpi_{\gpm^{\gpj}}})(\lambda w_{0})=\ww^{-1}\mathrm{Vol}(\gpi_{\gpg})\mathrm{Vol}(\gpi_{\gpm})\int_{N^{0}_{\alpha}}dn_{\alpha}\funk(w_{\alpha}n_{\alpha}w_{0}\lambda).
\end{equation*}
Combining this with (\ref{flat}),  we have 
\begin{equation*}
  \iiib(\mathrm{Ch}_{\gpi_{\gpg}}+\mathrm{Ch}_{\gpi_{\gpg}w_{\alpha}\gpi_{\gpg}},\mathrm{Ch}_{\gpi_{\gpm^{\gpj}}})(\lambda w_{0})=\mathrm{Vol}(\gpi_{\gpg})\mathrm{Vol}(\gpi_{\gpm})(1+\ww^{-1}\int_{N^{0}_{\alpha}}dn_{\alpha}\funk(w_{\alpha}n_{\alpha}w_{0}\lambda))
\end{equation*}
Note that $w_{\alpha}n_{\alpha}(t)\in \gpt_{\gpg}^{0}\gpt_{\alpha}(t^{-1})\gpn_{\alpha}(-t)\gpn_{-\alpha}(t^{-1})$, so it is equal to 
\begin{equation*}
  \mathrm{Vol}(\gpi_{\gpg})\mathrm{Vol}(\gpi_{\gpm})\cdot\left(1+\ww^{-1}\int_{|t|\leq 1}dt\funk(T_{\alpha}(t^{-1})w_0n_{\alpha}(t^{-1})\lambda)\right).
\end{equation*}
We consider the calculation of this case by case.
\hfill \\
\textbf{Case a.} When $\alpha=e_{i}-e_{i+1}$ with $1\leq i \leq n-m-1$ we have
$n_{\alpha}(t^{-1})\lambda=\lambda n_{\alpha}(t^{-1})$ if $i<n-m-1$, and $n_{\alpha}(t^{-1})\lambda=\lambda n_{\alpha}(t^{-1}) u$ for some $u\in \gpu$ with $\psi_{\gpu}(u)=1$ if $i=n-m-1$. 
So
\begin{equation*}
  \funk(T_{\alpha}(t^{-1})w_0n_{\alpha}(t^{-1})\lambda)=\funk(T_{\alpha}(t^{-1})w_0\lambda n_{\alpha}(t^{-1})), 
\end{equation*}
and so
\begin{equation*}
  \int_{|t|\leq 1}dt\funk(T_{\alpha}(t^{-1})w_0n_{\alpha}(t^{-1})\lambda)=\int_{|t|\leq 1}dt (\chag_{i}\chag_{i+1}^{-1})(t)|t|^{-1}\psi(t^{-1}).
\end{equation*}

We let 
\begin{align*}
  \gpi(n)=&\int_{|t|=\ww^{n}}dt(\chag_{i}\chag_{i+1}^{-1})(t)|t|^{-1}\psi(t^{-1})\\
=&(\chag_{i}\chag_{i+1}^{-1}(\w))^{n}\ww^{-n}\int_{|t|=\ww^{n}}dt \psi(t^{-1}).
\end{align*}
Using lemma (\ref{lem911}) below we have $\mathrm{I}(0)=1-\ww$, $\mathrm{I}(1)=(-1)(\chag_{i}\chag_{i+1}^{-1})(\w)\ww$, and $\mathrm{I}(n)=0$ if $n\geq 2$. Combining these we have 
\begin{equation*}
  \begin{split}
    &\iiib(\mathrm{Ch}_{\gpi_{\gpg}}+\mathrm{Ch}_{\gpi_{\gpg}w_{\alpha}\gpi_{\gpg}},\mathrm{Ch}_{\gpi_{\gpm^{\gpj}}})(\lambda w_{0})\\=
    &\mathrm{Vol}(\gpi_{\gpg})\mathrm{Vol}(\gpi_{\gpm^{\gpj}})\ww^{-1}(1-(\chag_{i}\chag_{i+1}^{-1})(\w)\ww),
  \end{split}
\end{equation*}
which is the first part of proposition (\ref{pro:alpha})
\\
\textbf{Case b.} When $\alpha=e_{i}-e_{i+1}$ with $n-m\leq i \leq n-1$, we let\\ $\mathbf{t}_{j}=(1,1,...,\overset{\text{j-th}}{1+t^{-1}},1,...,1)\in \m F^{m}$. Then $n_{\alpha}(t^{-1})\lambda=\gpx(\m t_{1})$ if $ i=n-m$, and $n_{\alpha}(t^{-1})\lambda=\gpx(\m t_{i'+1})n_{\alpha}(t^{-1})$ if $i\geq n-m+1$. Here $i'=i-(n-m)$. Note that $w_{0}\gpx(\m t_{j})=d'(\m t_{j})w_{0}\lambda d'(\m t_{j})$, so 
\begin{equation*}
  \funk(T_{\alpha}(t^{-1})w_0n_{\alpha}(t^{-1})\lambda)=\int_{|t|\leq 1}\ud t\, |t|^{\chag_{i}-\cham_{i'+1}-\frac{1}{2}}|t+1|^{-\chag_{i+1}+\cham_{i'+1}-\frac{1}{2}}.
\end{equation*}
To calculate this we apply the lemma 8.6 in \cite{MR1956080}.
\begin{lem}
  Suppose $\chi$ and $\chi'$ are two unramified character on $F^{*}$, then
  \begin{equation*}
    1+\ww^{-1}\int_{\mathcal O}\ud t\  \chi(t)\chi'(1+t)=(\ww^{-1}-1)\dfrac{1-\ww^{2}(\chi\chi')(\w)}{(1-\ww\chi(\w))(1-\ww\chi'(\w))}.
  \end{equation*}
\end{lem}
Applying the lemma for $\chag=\chag_{i}\cham_{i'+1}^{-1}|\cdot|^{-\frac{1}{2}}$ and
$\chag'=\chag_{i+1}^{-1}\cham_{i'+1}|\cdot|^{-\frac{1}{2}}$, we have, when $\alpha=e_{i}-e_{i+1}$ with $ n-m\leq i \leq n-1$, 
\begin{equation*}
  \begin{split}
    &\iiib(\mathrm{Ch}_{\gpi_{\gpg}}+\mathrm{Ch}_{\gpi_{\gpg}w_{\alpha}\gpi_{\gpg}},\mathrm{Ch}_{\gpi_{\gpm^{\gpj}}})(\lambda w_{0})\\
    =&\mathrm{Vol}(\gpi_{\gpg})\mathrm{Vol}(\gpi_{\gpm^{\gpj}})(\ww^{-1}-1)\dfrac{1-\ww(\chag_{i}\chag_{i+1}^{-1})(\w)}{(1-(\chag_{i}\cham_{i'+1}^{-1})(\w)\ww^{\frac{1}{2}})(1-(\chag_{i+1}^{-1}\cham_{i'+1})(\w)\ww^{\frac{1}{2}})},
  \end{split}
\end{equation*}
which is the second part of proposition (\ref{pro:alpha}).
\\
\textbf{Case c.} When $\alpha=2e_{n}$, we have 
\begin{equation*}
  n_{\alpha}(t^{-1})\lambda=\lambda \gpz(t^{-1})\gpy_{1}(-t^{-1})n_{\alpha}(t^{-1})
\end{equation*}
So 
\begin{equation*}
  \int_{|t|\leq 1}\ud t\,\funk(T_{\alpha}(t^{-1})w_0n_{\alpha}(t^{-1})\lambda)=\int_{|t|\leq 1}\ud t\, \chag_{n}(t)|t|^{-1}\psi(t^{-1}).
\end{equation*}

Similar to \textbf{Case a}, if we let 
\begin{equation*}
  \begin{split}
    &    \ti I(i)=\int_{|t|=\ww^{i}}dt \chag_{n}(t)|t|^{-1}\psi(t^{-1})\\=
    &\chag_{n}(\w)^{i}\ww^{-i}\int_{|t|=\ww^{i}}dt\psi(t^{-1}),
  \end{split}
\end{equation*}
then by lemma (\ref{lem911}), 
\begin{equation*}\ti I(i)=
  \begin{cases}
    1-\ww &\text{ if }i=0;\\
-\chag_{n}(\w)\ww &\text{ if } i=1;\\
0 &\text{ if }i\geq 2.
  \end{cases}
\end{equation*}
So 
\begin{align*}
  &\iiib(\mathrm{Ch}_{\gpi_{\gpg}}+\mathrm{Ch}_{\gpi_{\gpg}w_{\alpha}\gpi_{\gpg}},\mathrm{Ch}_{\gpi_{\gpm^{\gpj}}})(\lambda w_{0})\\
=&\mathrm{Vol}(\gpi_{\gpg})\mathrm{Vol}(\gpi_{\gpm^{\gpj}})\ww^{-1}\cdot(1-\chag_{n}(\w)\ww^{}),
\end{align*}
which is part 3 of proposition (\ref{pro:alpha}).

So by the calculation in the \textbf{Case a, b} and \textbf{c}, we have proved proposition (\ref{pro:alpha}), which implies theorem (\ref{thm:alpha}).

\subsection{Calculation of $\gamma(\chag,\cham,1,w_{\beta})$}
\hfill\\
\vspace{3pt} 
Let $\ti \chag_{i}=\chag_{n-m+i}$ for $1\leq i \leq m$. Let $\iwagw{\chag,1}=\funf_{\chag}(\mathrm{Ch}(\gpi_{\gpg}))$ as before. Recall that $\overline \gpi_{\gpm}=\gpi_{\gpm}\ltimes \gpj^{0}$. Let 
\begin{equation*}
  \iwamw{\cham,\psi,w}=\funf_{\cham,\psi}(\mathrm{Ch}_{\overline \gpi_{\gpm}w\overline \gpi_{\gpm}}).
\end{equation*}
Similar to theorem 3.4 in \cite{MR571057}, we have
\begin{equation}
  \mathrm{T}_{w_{\beta}}(\iwamw{\cham,\psi,1}+\iwamw{\cham,\psi,w_{\beta}})=\ti c_{\beta}(\cham)(\iwamw{w_{\beta}\cham,\psi,1}+\iwamw{w_{\beta}\cham,\psi,w_{\beta}}).
\end{equation}
So we have
\begin{equation}
\label{eqgammab}
  \gamma(\chag,\cham,1,w_{\beta})=\ti c_{\beta}(\cham)\cdot \dfrac{l_{\chag,w_{\beta}\cham,\psi}(R(\lambda w_{0}^{\gpg})\iwagw{\chag,1},\iwamw{w_{\beta}\cham,\psi,1}+\iwamw{w_{\beta}\cham,\psi,w_{\beta}})}{\lpair(R(\lambda w_{0}^{\gpg})\iwagw{\chag,1},\iwamw{\cham,\psi,1}+\iwamw{\cham,\psi,w_{\beta}})}
\end{equation}
What we are going to show is
\begin{pro} For generic $(\chag,\cham)$, the value of $\lpair(R(\lambda w_{0}^{\gpg})\iwagw{\chag,1},\iwamw{\cham,\psi,1}+\iwamw{\cham,\psi,w_{\beta}})$ equals
  \label{pro:beta}
  \begin{equation*}
\begin{split}
     &\ww^{-1}(1-\ww)^{m}\mathrm{Vol}(I_{\gpg})\mathrm{Vol}(I_{\gpm})\prod_{j\neq i,i+1}\zeta(\frac{1}{2}-\ti\chag_{j}+\cham_{j})\\
    &\cdot \dfrac{\zeta(\cham_{i}-\ti{\chag}_{i}+\frac{1}{2})\zeta(\cham_{i+1}-\ti{\chag}_{i+1}+\frac{1}{2})\zeta(\cham_{i}-\ti{\chag}_{i+1}+\frac{1}{2})\zeta(-\cham_{i+1}+\ti{\chag}_{i}+\frac{1}{2})}{\zeta(\ti{\chag}_{i}-\ti{\chag}_{i+1}+1)\zeta(\cham_{i}-\cham_{i+1}+1)}
\end{split}
\end{equation*}
when $\beta=e_{i}-e_{i+1}$, and equals
\begin{equation*}
  \begin{split}
    &\ww^{-1}(1-\ww)^{m}\mathrm{Vol}(I_{\gpg})\mathrm{Vol}(I_{\gpm})\prod_{j=1}^{m-1}\zeta(-\ti\chag_{j}+\cham_{j}+\frac 1 2)\\
        &\cdot \dfrac{(1-\ti{\chag}_{m}|\cdot|)(\w)(1+\cham_{m}|\cdot|^{\frac{1}{2}}(\w))}{(1-\ti{\chag}_{m}^{-1}\cham_{m}|\cdot|^{\frac{1}{2}}(\w))(1-\ti{\chag}_{m}\cham_{m}|\cdot|^{\frac{1}{2}}(\w))}
  \end{split}
\end{equation*}
when $\beta=2e_{m}$. 
\end{pro}
Substituting this in (\ref{eqgammab}) we have
\begin{thm}
  \label{thm:beta}
 For generic $(\chag,\cham)$, the $\gamma$-factor $\gamma(\chag,\cham,1,w_{\beta})$ equals
  \begin{equation*}
     \ti c_{\beta}(\cham)\dfrac{\zeta(\cham_{i}-\cham_{i+1}+1)\zeta(-\ti{\chag}_{i}+\cham_{i+1}+\frac{1}{2})\zeta(\ti{\chag}_{i}-\cham_{i}+\frac{1}{2})}{\zeta(-\cham_{i}+\cham_{i+1}+1)\zeta(\ti{\chag}_{i}-\cham_{i+1}+\frac{1}{2})\zeta(-\ti{\chag}_{i}+\cham_{i}+\frac{1}{2})}\\
\end{equation*}
when $\beta=e_{i}-e_{i+1}$, and equals
\begin{equation*}
\ti c_{\beta}(\cham)\dfrac{\zeta(-\cham_{m}-\ti{\chag}_{m}+\frac 1 2)\zeta(-\cham_{m}+\ti{\chag}_{m}+\frac 1 2)\zeta(2\cham_{m}+1)\zeta(-\cham_{m}+\frac 1 2)}{\zeta(\cham_{m}-\ti{\chag}_{m}+\frac 1 2)\zeta(\cham_{m}+\ti{\chag}_{m}+\frac 1 2)\zeta(-2\cham_{m}+1)\zeta(\cham_{m}+\frac 1 2)}
\end{equation*}
when $\beta=2e_{m}$.
\end{thm}

To prove \ref{pro:beta}, first we have
\begin{pro} For $(\chag,\cham)\in \mathcal{Z}_{C}$, 
  \label{pro:beta.1}
  \begin{equation*}
    \iiib(\mathrm{Ch}_{\gpi_{\gpg}},\mathrm{Ch}_{\overline\gpi_{\gpm}})(\lambda w_{0}^{\gpg})=\mathrm{Vol}(\gpi_{\gpg})\mathrm{Vol}(\gpi_{\gpm})\prod_{j=1}^{m}\dfrac{1-\ww}{1-\ti{\chag}_j^{-1}\cham_{j}(\w)\ww^{\frac{1}{2}}}
  \end{equation*}
\end{pro}
\begin{proof}

By definition, we have
\begin{equation*}
  \iiib(\mathrm{Ch}_{\gpi_{\gpg}},\mathrm{Ch}_{\overline\gpi_{\gpm}})(\lambda w_{0}^{\gpg})=\int_{\gpi_{\gpg}}\int_{\overline \gpi_{\gpm}}\ud g \ud m\funk(gw_{0}^{\gpg}\lambda^{-1}m).
\end{equation*}
Since $\lambda\in \overline \gpi_{\gpm}$, and $\funk$ is left $\gpb_{\gpg}^{0}$-invariant and right $\gpb_{\gpm}^{0}$-invariant, and $\mathrm{Vol}(\gpb_{\gpg}^{0})=\mathrm{Vol}(\gpb_{\gpm}^{0})=1$, we have 
\begin{equation*}
  \int_{\gpi_{\gpg}}\int_{\overline \gpi_{\gpm}}\ud g \ud m\funk(gw_{0}^{\gpg}\lambda^{-1}m)=\int_{N_{\gpg}^{-,1}}\ud n\int_{N_{\gpm}^{-,1}\gpx^{0}}\ud n' \ud x'\funk(nw_{0}^{\gpg} n'x').
\end{equation*}
By a similar discussion as in lemma (\ref{lem:basic}), which is given later, we have,  for all $n\in N_{\gpg}^{-,1}$ and $n'\in N_{\gpm}^{-,1}$,
\begin{equation*}
  \int_{\gpx^{0}}\ud x'\funk(nw_{0}^{\gpg} n'x')=\int_{\gpx^{0}}\ud x'\funk(w_{0}^{\gpg}x').
\end{equation*}
So
\begin{equation*}
  \iiib(\mathrm{Ch}_{\gpi_{\gpg}},\mathrm{Ch}_{\overline\gpi_{\gpm}})(\lambda w_{0}^{\gpg})=\mathrm{Vol}(\gpi_{\gpg})\mathrm{Vol}(\gpi_{\gpm})\int_{\gpx^{0}}\ud x'\funk(w_{0}^{\gpg} x').
\end{equation*}
Parametrize $\gpx^{0}$ as $\gpx^{0}=\{\gpx(\m x)=\gpx(x_{1},...,x_{m})|x_{i}\in \mathcal O \text{ for all } i \}$, then,\\ $w_{0}^{\gpg}\gpx(x_{1},...,x_{m})=d_{m}(x_{1},...,x_{m})w_{0}^{\gpg}\lambda d_{m}(x_{1},...,x_{m})$ if all $x_{i}$ are non-zero, and\\ $\funk(w_{0}^{\gpg} x')=0$ otherwise. So 
\begin{align*}
  \int_{\gpx^{0}}dx'\funk(w_{0}^{\gpg} x')&=\prod_{j=1}^{m}\int_{x_{j}\in \mathcal O}\ud x_{j}(\ti{\chag}_j^{-1}\cham_{j}|\cdot|^{-\frac{1}{2}})(x_{j})\\
  &=\prod_{j=1}^{m}\dfrac{1-\ww}{1-\ti{\chag}_j^{-1}\cham_{j}(\w)\ww^{\frac{1}{2}}},
\end{align*}
where the last equality follows from the following lemma.
\end{proof}
\begin{lem}
\label{lem:basic1}
  Let $\chi$ be an unramified character on $F^{*}$. Then
  \begin{equation*}
    \int_{x\in \mathcal O^{*}}\chi(x)\ud x=\dfrac{1-\ww}{1-\chi(\w)\ww}.
  \end{equation*}
\end{lem}
We skip its proof.

Next we consider $\iiib(\mathrm{Ch}_{\gpi_{\gpg}},\mathrm{Ch}_{\overline\gpi_{\gpm}w_{\beta}\overline\gpi_{\gpm}})(\lambda w_{0}^{\gpg})$. By definition, it is equal to
\begin{equation*}
  \int_{\gpi_{\gpg}}\ud x\int_{\overline \gpi_{\gpm}w_{\beta}\overline \gpi_{\gpm}}\ud x'\funk(xw_{0}^{\gpg}\lambda^{-1}x').
\end{equation*}
Note that $\lambda\in \mathrm{\overline I}_{m}$, and $\overline \gpi_{\gpm}w_{\beta}\overline \gpi_{\gpm}=\gpn_{\gpm,\hat{\beta}}^{-,1}w_{\beta}\gpn_{-\beta}^{0}\gpj^{0}\gpb_{\gpm}^{0}$. So the above integral is equal to
\begin{equation*}
  \int_{\gpn_{\gpg}^{-,1}}\ud n\int_{\gpn_{\gpm}^{-,1}}\ud n'\int_{\gpn_{-\beta}^{0}}\ud n_{-\beta}\int_{\gpj^{0}}\ud u \,\funk(nw_{0}^{\gpg}n'w_{\beta}n_{-\beta}u)
\end{equation*}
We claim that
\begin{lem}
\label{lem:basic}
  For any $n \in \gpn_{\gpg}^{-,1}$ and $n'\in \gpn_{\gpm}^{-,1}$,
  \begin{equation*}
    \int_{\gpn_{-\beta}^{0}}\ud n_{-\beta}\int_{\gpj^{0}}\ud u \,\funk(nw_{0}^{\gpg} n'w_{\beta}n_{-\beta}u)=\int_{\gpn_{-\beta}^{0}}\ud n_{-\beta}\int_{\gpj^{0}}\ud u \,\funk(w_{0}^{\gpg} w_{\beta}n_{-\beta}u).
  \end{equation*}
\end{lem}
\begin{proof}
  First we have
  \begin{equation}
    \label{eq:basic11}
    \int_{\gpn_{-\beta}^{0}}\ud n_{-\beta}\int_{\gpj^{0}}\ud u \,\funk(nw_{0}^{\gpg} n'w_{\beta}n_{-\beta}u)=\int_{\gpn_{-\beta}^{0}}\ud n_{-\beta}\int_{\gpj^{0}}\ud u \,\funk(\overline nw_{0}^{\gpg} w_{\beta}n_{-\beta}u).
  \end{equation}
  for some $\overline n\in \gpn_{\gpg}^{-,1}$. This is true because
  \begin{enumerate}
  \item $w_{0}^{\gpg}\gpn_{\gpm}^{-,1}w_{0}^{\gpg}=\gpn_{\gpm}^{1}$,
  \item $\gpn_{\gpg}^{-,1}\gpn_{\gpm}^{1}\subset \gpb_{\gpm}^{0}\gpn_{\gpg}^{-,1}$,
  \item $\funk$ is left $\gpb_{\gpm}^{0}$-invariant.
  \end{enumerate}
Continue the calculation, the right hand side of equation (\ref{eq:basic11}) is equal to 
\begin{align}
\label{eq:basic3}
  &\int_{\gpn_{-\beta}^{0}}\ud n_{-\beta}\int_{\gpj^{0}}\ud u \,\funk(w_{0}^{\gpg}\ti n w_{\beta}n_{-\beta}u)  
\end{align}
for some $\ti n\in \gpn_{\gpg}^{1}$. We write $\ti n$ as $\ti n=n_{1}n_{2}\gpj(x,y,z)n_{3}$ where $n_{1}\in \gpn_{\beta}^{1}$, $n_{2}\in \gpn_{\gpm,\check{\beta}}^{1}$, $\gpj(x,y,z)\in \gpj^{1}$, and $n_{3}\in \gpu^{1}$. Here $\gpn_{\gpm,\check{\beta}}$ is the subgroup of $\gpn_{\gpm}$ generated by all positive roots in $\gpm$ except $\beta$. Then (\ref{eq:basic3}) is equal to
\begin{equation*}
  \int_{\gpn_{-\beta}^{0}}\ud n_{-\beta}\int_{\gpj^{0}}\ud u \,\funk(w_{0}^{\gpg}n_{1}n_{2}\gpj(x,y,z)n_{3} w_{\beta}n_{-\beta}u) 
\end{equation*}
Our lemma will be proved if we remove $n_{1}n_{2}\gpj(x,y,z)n_{3}$ in the integral, which is by the following steps
\begin{enumerate}
\item First, $n_{3}$ can be removed because $\gpu^{1}$ is normalized by $w_{\beta}$, $\gpn_{-\beta}^{0}$, and $\gpj^{0}$, and $\funk$ is right $\gpu^{1}$-invariant. 
\item Note that $\gpj^{1}$ is normalized by $w_{\beta}\gpn_{-\beta}^{0}$, we can remove $\gpj^{1}(x,y,z)$ by a change of variable in the integral of $\gpj^{0}$.
\item We can remove $n_{2}$ because $\gpn_{\gpm,\check{\beta}}^{1}$ is normalized by $w_{\beta}\gpn_{-\beta}^{0}$, and $\gpn_{\gpm,\check{\beta}}^{1}$ normalizes $\gpj^{0}$, and $\funk$ is right $\gpn_{\gpm,\check{\beta}}^{1}$-invariant.
\item Finally, note that $\gpn_{\beta}^{1}w_{\beta}=w_{\beta}\gpn_{-\beta}^{1}$, we can remove $n_{1}$ by a change of variable in the integral of $\gpn_{-\beta}^{0}$
\end{enumerate}
So the lemma is proved.
\end{proof}

By the lemma we have
\begin{equation*}
 \iiib(\mathrm{Ch}_{\gpi_{\gpg}},\mathrm{Ch}_{\overline\gpi_{\gpm}w_{\beta}\overline\gpi_{\gpm}})(\lambda w_{0}^{\gpg})=\ww^{-1}\mathrm{Vol}(\gpi_{\gpg})\mathrm{Vol}(\gpi_{\gpm})\int_{|t|\leq 1}\ud t\int_{\gpx^{0}}\ud x \funk(w_{0}^{\gpg} w_{\beta}n_{-\beta}(t)x)
\end{equation*}

Combining this with proposition (\ref{pro:beta.1}), we have
\begin{align}
  &\iiib(\mathrm{Ch}_{\gpi_{\gpg}},\mathrm{Ch}_{\overline\gpi_{\gpm}}+\mathrm{Ch}_{\overline\gpi_{\gpm}w_{\beta}\overline\gpi_{\gpm}})(\lambda w_{0}^{\gpg})\notag\\
  =&\mathrm{Vol}(\gpi_{\gpg})\mathrm{Vol}(\gpi_{\gpm})\left(\prod_{j=1}^{m}\dfrac{1-\ww}{1-\ti{\chag}_j^{-1}\cham_{j}(\w)\ww^{\frac{1}{2}}}+\ww^{-1}\int_{|t|\leq 1}\ud t\int_{\gpx^{0}}\ud x \funk(w_{0}^{\gpg} w_{\beta}n_{-\beta}(t)x)\right).
\label{eq:beta1}
\end{align}
Now we calculate
\begin{equation*}
  \int_{|t|\leq 1}\ud t\int_{\gpx^{0}}\ud x \funk(w_{0}^{\gpg}w_{\beta}n_{-\beta}(t)x).
\end{equation*}
Note that %\begin{equation*}
$w_{\beta}n_{-\beta}(t)=n_{-\beta}(-t^{-1})\gpn_{\beta}(t)T_{\beta}(t)$, so
%\end{equation*}

\begin{align*}
  &\int_{|t|\leq 1}\ud t\int_{\gpx^{0}}\ud x \funk(w_{0}^{\gpg} w_{\beta}n_{-\beta}(t)x)\\
=&\int_{|t|\leq 1}\ud t\int_{\gpx^{0}}\ud x \funk(w_{0}^{\gpg} n_{-\beta}(-t^{-1})n_{\beta}(t)T_{\beta}(t) x)\\
=&\int_{|t|\leq 1}\ud t\int_{\gpx^{0}}\ud x \funk(w_{0}^{\gpg}x^{n_{\beta}(t)T_{\beta}(t)})(\cham\dnh_{\gpb_{\gpm}^{\gpj}})(T_{\beta}(t))
\end{align*}
Below we discuss this integral case by case.
\hfill\\
\textbf{Case a.} When $\beta=e_{i}-e_{i+1}$. 

We parametrize $\gpx^{0}$ as $\gpx(x)=\gpx(x_{1},...,x_{m})$ with $x_{j}\in \mathcal{O}$. Then we have
\begin{equation*}
  \gpx(x_{1},...,x_{m})^{n_{\beta}(t)T_{\beta}(t)}=\gpx(x_{1},...,x_{i-1},t^{-1}x_{i},-x_{i}+tx_{i+1},x_{i+2},\ldots,x_{m})
\end{equation*}
and
\begin{equation*}
  (\cham\dnh_{\gpb_{\gpm}^{\gpj}})(T_{\beta}(t))=(\cham_{i}\cham_{i+1}^{-1}|\cdot|^{-1})(t).
\end{equation*}
Let $\mathrm{x}(i,t)=(x_{1},...,x_{i-1},t^{-1}x_{i},-x_{i}+tx_{i+1},x_{i+2},...,x_{m})$, then
\begin{equation*}
   \begin{split}
    &\int_{|t|\leq 1}\ud t\int_{\gpx^{0}}\ud x \funk(w_{0}^{\gpg}x^{n_{\beta}(t)T_{\beta}(t)})(\cham\dnh)(T_{\beta}(t))\\
&=\int_{|t|\leq 1}\ud t\int_{|x_{i}|\leq 1,i=1,2,...,m}\ud x_{i} \, \, (\cham_{i}\cham_{i+1}^{-1})(t)|t|^{-1}\funk(w_{0}^{\gpg}\gpx(\mathrm{x}(i,t))).
\end{split}
\end{equation*}
Note that if none of the components of $\mathrm{x}(i,t)$ is zero, then
\begin{equation*}
w_{0}^{\gpg}\gpx(\mathrm{x}(i,t))=d_{m}(\mathrm{x}(i,t))w_{0}^{\gpg}\lambda d_{m}(\mathrm{x}(i,t))
\end{equation*}
 and $\funk(w_{0}^{\gpg}\gpx(\mathrm{x}(i,t)))=0$ if otherwise. So the integral above is equal to
\begin{align*}
 & \prod_{j\in \{1,2,...,i-1,i+2,...,m\} }\int_{x_{j}\in \mathcal O}(\ti{\chag}_{j}^{-1}\cham_{j}|\cdot|^{-\frac{1}{2}})(x_{j})\ud x_{j}\cdot\\
&\int_{|t|\leq 1,|x_{i}|\leq 1,|x_{i+1}|\leq 1}\ud t\ud x_{i}\ud x_{i+1} (\cham_{i}\cham^{-1}_{i+1}|\cdot|^{-1})(t)\\
&(\ti{\chag}^{-1}_{i}\cham_{i}|\cdot|^{-\frac{1}{2}})(t^{-1}x_{i})\,(\ti{\chag}^{-1}_{i+1}\cham_{i+1}|\cdot|^{-\frac{1}{2}})(-x_{i}+tx_{i+1})\\
\end{align*}
By lemma (\ref{lem:basic1}), it is equal to
\begin{equation}
\label{eq18}
  \begin{split}
    & \prod_{j\in \{1,2,...,i-1,i+2,...,m\} }\dfrac{1-\ww}{1-\ti{\chag}_j^{-1}\cham_{j}(\w)\ww^{\frac{1}{2}}}\\
  &\int_{|t|\leq 1,|x_{i}|\leq 1,|x_{i+1}|\leq 1}\ud t\ud x_{i}\ud x_{i+1} (\cham_{i}\cham^{-1}_{i+1}|\cdot|^{-1})(t)\\
&(\ti{\chag}^{-1}_{i}\cham_{i}|\cdot|^{-\frac{1}{2}})(t^{-1}x_{i})\,(\ti{\chag}^{-1}_{i+1}\cham_{i+1}|\cdot|^{-\frac{1}{2}})(-x_{i}+tx_{i+1})
  \end{split}  
\end{equation}
Let $\gpi(a,b)$ be part of the integral above for $|t|=\ww^{a}$ and $|x_{i}|=\ww^{b}$ with $a,b\geq 0$, then the integral is equal to $\sum_{a,b\geq 0}\gpi(a,b)$.  When $a>b$. i.e., $|x_{i}|>|t|$,  we have $|-x_{i}+tx_{i+1}|=|x_{i}|$. So
\begin{equation*}
  \gpi(a,b)=(1-\ww)^{2}\ww^{\frac a 2}\ti\chag_{i}^{a-b}\ti\chag_{i+1}^{-b}\cham_{i}^{b}\cham_{i+1}^{-a+b}.
\end{equation*}
When $a\leq b$, i.e., $|x_{i}|\leq |t|$ we can change the variable $x_{i+1}\to x_{i+1}+\frac{x_{i}}{t}$, we have
\begin{equation*}
  \gpi(a,b)=\sum_{c\geq 0}(1-\ww)^{3}\ww^{\frac{(b+c)}{2}}\ti\chag_{i}^{a-b}\ti\chag_{i+1}^{-a-c}\cham_{i}^{b}\cham_{i+1}^{c},
\end{equation*}
where the summand for $c$ corresponds to $|x_{i+1}|=\ww^{c}$ after the change of variable. So applying this to equation (\ref{eq:beta1}) and by some calculation,
\begin{equation*}
  \begin{split}
    &\iiib(\mathrm{Ch}_{\gpi_{\gpg}},\mathrm{Ch}_{\overline\gpi_{\gpm}}+\mathrm{Ch}_{\overline\gpi_{\gpm}w_{\beta}\overline\gpi_{\gpm}})(\lambda w_{0}^{\gpg})=\mathrm{Vol}(\gpi_{\gpg})\mathrm{Vol}(\gpi_{\gpm})\\
    &\cdot\left(\prod_{j=1}^{m}\dfrac{1-\ww}{1-\ti{\chag}_j^{-1}\cham_{j}(\w)\ww^{\frac{1}{2}}}+\ww^{-1}\prod_{j\neq i,i+1 }\dfrac{1-\ww}{1-\ti{\chag}_j^{-1}\cham_{j}(\w)\ww^{\frac{1}{2}}}\sum_{a,b\geq 0}\gpi(a,b)\right)\\
    =&\ww^{-1}(1-\ww)^{m}\mathrm{Vol}(I_{\gpg})\mathrm{Vol}(I_{\gpm})\prod_{j\neq i,i+1}\zeta(\frac{1}{2}-\ti\chag_{j}+\cham_{j})\\
    \cdot &\dfrac{\zeta(\cham_{i}-\ti{\chag}_{i}+\frac{1}{2})\zeta(\cham_{i+1}-\ti{\chag}_{i+1}+\frac{1}{2})\zeta(\cham_{i}-\ti{\chag}_{i+1}+\frac{1}{2})\zeta(-\cham_{i+1}+\ti{\chag}_{i}+\frac{1}{2})}{\zeta(\ti{\chag}_{i}-\ti{\chag}_{i+1}+1)\zeta(\cham_{i}-\cham_{i+1}+1)}
  \end{split}
\end{equation*}

which is the first part of proposition (\ref{pro:beta})

\textbf{Case b.} When $\beta=2e_{m}$, we have
\begin{equation*}
   (\cham\dnh_{\gpb_{\gpm^{\gpj}}})(T_{\beta}(t))=\cham_{m}(t)|t|^{-\frac{3}{2}},
\end{equation*}
and 
\begin{equation*}
  \gpx(x_{1},...,x_{m})^{n_{\beta}(t)T_{\beta}(t)}=\gpx(x_{1},...,x_{m-1},t^{-1}x_{m})\gpy_{1}(-x_{m})\gpz(t^{-1}x_{m}^{2}).
\end{equation*}
So
\begin{align}
  &  \int_{|t|\leq 1}\ud t\int_{\gpx^{0}}\ud x \funk(w_{0}^{\gpg}x^{n_{\beta}(t)T_{\beta}(t)})(\cham\dnh)(T_{\beta}(t))\notag\\
  =&\int_{|t|\leq 1}\ud t\int_{x_{i}\in \mathcal O}\ud x_{i}\, \psi(t^{-1}x_{m}^{2})\cham_{m}(t)|t|^{-\frac{3}{2}}\funk(w_{0}^{\gpg}\gpx(x_{1},...,x_{m-1},t^{-1}x_{m}))\\
=&\prod_{j=1}^{m-1}\dfrac{1-\ww}{1-\ti {\chag}_{j}^{-1}\cham_{j}|\cdot|^{\frac{1}{2}}(\w)}\\
&\cdot \int_{t,x_{m}\in \mathcal O}\ud t \ud x_{m}\,(\ti{\chag}_{m}^{-1}\cham_{m}|\cdot|^{-\frac{1}{2}})(t^{-1}x_{m})(\cham_{m}|\cdot|^{-\frac{3}{2}})(t)\psi(t^{-1}x_{m}^{2}).
\end{align}
Let $\gpi(a,b)$ be part of the integral above when $|t|=\ww^{a}$ and $|x_{m}|=\ww^{b}$. To calculate $\gpi(a,b)$, we apply the following lemma which can be proved by direct calculation
\begin{lem}
\label{lem911}
Suppose $|x|=\ww^{i}$, then
  \begin{equation*}
\int_{|t|=\ww^{j}}\psi(t^{-1}x)dt=
    \begin{cases}
      \ww^{j}(1-\ww) &\text{ if $j\leq i$}\\
      -\ww^{i+2} &\text{ if $j=i+1$}\\
      0         &\text{ if $j\geq i+2$}  
    \end{cases}
  \end{equation*}
\end{lem}
So $\gpi(a,b)=(1-\ww)^{2}\ti\chag_{m}^{a-b}\cham_{m}^{b}\ww^{\frac 1 2 b}$ 
when $0\leq a \leq 2b$, and $\gpi(a,b)=-(1-\ww)\ti\chag_{m}\cham_{m}^{b}\ww^{\frac 1 2 b+1}$ when $a=2b+1$, and $\gpi(a,b)=0$ when $a\geq 2b+2$. Applying these to equation (\ref{eq:beta1}) and by some calculation, we have
\begin{equation*}
  \begin{split}
    &   \iiib(\mathrm{Ch}_{\gpi_{\gpg}},\mathrm{Ch}_{\overline\gpi_{\gpm}}+\mathrm{Ch}_{\overline\gpi_{\gpm}w_{\beta}\overline\gpi_{\gpm}})(\lambda w_{0}^{\gpg})\\=
    &\ww^{-1}(1-\ww)^{m}\mathrm{Vol}(I_{\gpg})\mathrm{Vol}(I_{\gpm})\prod_{j=1}^{m-1}\zeta(-\ti\chag_{j}+\cham_{j}+\frac 1 2)\\
&\cdot \frac{(1-\ti{\chag}_{m}|\cdot|)(\w)(1+\cham_{m}|\cdot|^{\frac{1}{2}}(\w))}{(1-\ti{\chag}_{m}^{-1}\cham_{m}|\cdot|^{\frac{1}{2}}(\w))(1-\ti{\chag}_{m}\cham_{m}|\cdot|^{\frac{1}{2}}(\w))}
  \end{split}
\end{equation*}
which is the second part of the proposition (\ref{pro:beta}). Combining \textbf{Case a} and \textbf{b}, we have proposition \ref{pro:beta}, which implies theorem \ref{thm:beta}.

\section{Formula for generic $(\chag,\cham)$}
\label{sec:step7}

In this section we discuss the value of $\iii(\w^{\m d}\lambda\w^{\m f})$ for generic $(\chag,\cham)$ where $\m d\in \Lambda^{+}_{m}$ and $\m f \in \Lambda^{+}_{n}$.  First we claim that
\begin{lem}
\label{lem7.1}
  For any $g\in \gpi_{\gpg}$ and $m\in \gpi_{\gpm}$, we have
  \begin{equation*}
    \w^{\m d}\lambda w_{0}^{\gpg}\w^{\m -f}~\w^{\m d}m\lambda w_{0}^{\gpg} g\w^{\m -f} 
  \end{equation*}
in the double coset $U\gpk_{\gpm^{\gpj}}\backslash \gpg\slash \gpk_{\gpg}$.
\end{lem}
\begin{proof}
  When $\m d\in \Lambda^{+}_{m}$ and $\m f\in \Lambda^{+}_{n}$, we have $\w^{\m d}\gpb_{\gpm}^{0}\w^{-\m d}\in \gpb_{\gpm}^{0}$ and $\w^{\m f}\gpb_{\gpg}^{0}\w^{-\m f}\in \gpb_{\gpg}^{0}$, so it suffices to prove the lemma when $g\in \gpn_{\gpg}^{-,1}$ and $m\in \gpn_{\gpm}^{-,1}$, which is true by the method in lemma (\ref{lem:basic}).
\end{proof}
Now we consider
\begin{equation*}
\Gamma^{-1}(\chag,\cham)\cdot \pair{R(\lambda w_{0}^{\gpg})R(\gpi_{\gpg}\w^{\m -f}\gpi_{\gpg})\funf_{\chag}^{0}}{R( \overline\gpi_{\gpm}\w^{\m -d}\overline\gpi_{\gpm})\funf_{\cham,\psi}^{0}}.  
\end{equation*}
By lemma (\ref{lem7.1}) it is equal to 
\begin{equation*}
    \Gamma^{-1}(\chag,\cham)\mathrm{Vol}(\gpi_{\gpg})\mathrm{Vol}(\gpi_{\gpm})\int_{\gpx^{0}}\ud x\,\iii(\w^{\m d}x\w^{\m f}).
\end{equation*}
On the other hand, by lemma (\ref{pro:aG}) and (\ref{pro:aM}), it is equal to 
\begin{equation*}
\begin{split}
    &\Gamma^{-1}(\chag,\cham)\sum_{w\in W_{\gpg},w'\in W_{\gpm}}c_{w_{0}^{\gpg}}(w\chag)(w\chag\dnh_{\gpb_{\gpg}})(\w^{-\m f})\ti c_{w_{0}^{\gpm}}(w'\cham)(w'\cham\dnh_{\gpb_{\gpm^{\gpj}}})(\w^{-\m d})\\
    &\cdot \pair{R(\lambda w_{0})\overline T_{w^{-1}}\iwag^{w\chag}}{\overline T_{(w')^{-1}}\iwam^{w'\cham,\psi}}
 \end{split}
\end{equation*}
So we have
\begin{equation}
\label{eqluan}
  \begin{split}
    \Gamma^{-1}(\chag,\cham)\int_{\gpx^{0}}&\iii(\w^{\m d} x \w^{\m f}) =\mathrm{Vol}(\gpi_{\gpg})^{-1}\mathrm{Vol}(\gpi_{\gpm})^{-1}\\
    \cdot\sum_{w\in W_{\gpg},w'\in W_{\gpm}}&\left(c_{w_{0}^{\gpg}}(w\chag)(w\chag\dnh_{\gpb_{\gpg}})(\w^{-\m f})\ti c_{w_{0}^{\gpm}}(w'\cham)(w'\cham\dnh_{\gpb_{\gpm^{\gpj}}})(\w^{-\m d})\right.\\
    & \cdot \left.\pair{R(\lambda w_{0})\overline T_{w^{-1}}\iwag^{w\chag}}{\overline T_{(w')^{-1}}\iwam^{w'\cham,\psi}}\Gamma^{-1}(\chag,\cham)\right).
  \end{split}
\end{equation}
To calculate the right hand side of the equation above, we use the following lemma without proof.
\begin{lem}
  \label{1childlem}
For generic $(\chag,\cham)$, and $(\m d, \m f)\in \Lambda^{+}_{m}\times \Lambda^{+}_{n}$, let
 \begin{equation*}
  \mathcal{B}(\chag,\cham,\m d, \m f)=\sum_{w\in \gpw_{\gpg},w'\in\gpw_{\gpm}}\mathcal{A}(\chag,\cham,w,w')(w\chag)(\w^{\m f})(w'\cham)(\w^{\m d}).
\end{equation*}
Suppose $\mathcal{B}(\chag,\cham,\m d, \m f)$ is $\gpw_{\gpg}\times \gpw_{\gpm}$-invariant for all $(\m d,\m f)\in \Lambda^{+}_{m}\times \Lambda^{+}_{n}$, then we have
\begin{equation*}
  \mathcal{A}(\chag,\cham,w,w')=\mathcal{A}(w\chag,w'\cham, e, e).
\end{equation*}
\end{lem}

By this lemma, the summation on the right hand side of (\ref{eqluan}) is determined by its summand at $(w,w')=(e,e)$, which, by proposition \ref{pro:beta.1}, is equal to
\begin{equation*}
  \begin{split}
    &c_{w_{0}^{\gpg}}(\chag)(\chag\dnh_{\gpb_{\gpg}})(\w^{-\m f})\ti c_{w_{0}^{\gpm}}(\cham)(\cham\dnh_{\gpb_{\gpm^{\gpj}}})(\w^{-\m d})\cdot \pair{R(\lambda w_{0})\iwag^{\chag}}{\iwam^{\cham,\psi}}\cdot \Gamma^{-1}(\chag,\cham)\\
    =&(1-\ww)^{m}\mathrm{b}(\chag,\cham)\mathrm{d}(\chag)\mathrm{d'}(\cham)(\chag\dnh_{\gpb_{\gpg}})(\w^{-\m f})(\cham\dnh_{\gpb_{\gpm^{\gpj}}})(\w^{-\m d})\mathrm{Vol}(\gpi_{\gpg})\mathrm{Vol}(\gpi_{\gpm}),
  \end{split}
\end{equation*}
where
\begin{equation*}
  \mathrm{d}(\chag)=\prod_{1 \leq a< b \leq n }\zeta(\chag_{a}-\chag_{b})\zeta(\chag_{a}+\chag_{b})\prod_{i=1}^{n}\zeta(\chag_{i}),
\end{equation*}
\begin{equation*}
  \mathrm{d'}(\cham)=\prod_{1\leq a<b \leq m}\zeta(\cham_{a}-\cham_{b})\zeta(\cham_{a}+\cham_{b})\prod_{j=1}^{m}\zeta(2\cham_{j}),
\end{equation*}
and
\begin{equation*}
  \begin{split}
    \mathrm{b}(\chag,\cham)=&\prod_{i<j+n-m}\zeta^{-1}(\chag_{i}-\cham_{j}+\frac{1}{2})\cdot \prod_{i>j+n-m}\zeta^{-1}(-\chag_{i}+\cham_{j}+\frac{1}{2})\\
    &\prod_{ 1\leq j \leq m}\zeta^{-1}(\cham_{j}+\frac 1 2)\cdot \prod_{\substack{1 \leq i \leq n\\ 1\leq j \leq m}}\zeta^{-1}(\chag_{i}+\cham_{j}+\frac{1}{2})
  \end{split}   
\end{equation*}
Applying the lemma $\ref{1childlem}$, we have 
\begin{equation}
  \label{eq:final}
  \begin{split}
    &\int_{\gpx^{0}}\ud x\,\iii(\w^{\m d} x\w^{\m f})=(1-\ww)^{m}\Gamma(\chag,\cham)\\
    &\cdot\sum_{w\in \gpw_{\gpg},w'\in\gpw_{\gpm}}\mathrm{b}(w\chag,w'\cham)\mathrm{d}(w\chag)\mathrm{d'}(w'\cham)(w\chag\dnh_{\gpb_{\gpg}})(\w^{-\m f})(w'\cham\dnh_{\gpb_{\gpm^{\gpj}}})(\w^{-\m d}).
  \end{split}
\end{equation}
Though this is not a direct formula for $\iii(\w^{\m d}\lambda \w^{\m f})$, we have the following theorem.
\begin{thm}
  \label{usefulthm}
  Let $l(\m d,\m f)=\iii(\w^{\m d}\lambda \w^{\m f})$, and let
  \begin{equation}
    \label{eq:expand}
L(\m d,\m f)=\int_{\gpx^{0}}\ud x\,\iii(\w^{\m d} x \w^{\m f}),
  \end{equation}
then $l(\m 0, \m f)=L(\m 0,\m f)$. Moreover, there exists $a(\m d')\geq 0$ such that
\begin{equation*}
  l(\m d,\m f)=\sum_{\m d'} a(\m d')L(\m d',\mathbf{f+d-d'}) 
\end{equation*}
where the sum is over the set $\{\m d'\mid \mathbf{d'\leq d}, \mathbf{d'}\in \Lambda^{+}_{m}, \mathbf{f+d-d'}\in \Lambda^{+}_{n}\}$, and we have $a(\m d)> 0$.
\end{thm}
\begin{proof}
 We have $l(\m 0,\m f)=L(\m 0,\m f)$ by the $\gpk_{\gpm^{\gpj}}$-invariance on the left. Because of this, if we can prove
  \begin{equation}
\label{biglsmalll}
   L(\m d,\m f)=\sum_{\m d'}b(\m d')l(\m d',\mathbf{f+d-d'}) 
  \end{equation}
where $\mathbf{d'}$ runs over $\{\m d'\mid \mathbf{d'\leq d}, \mathrm{d'}\in \Lambda^{+}_{m}, \mathbf{f+d-d'}\in \Lambda^{+}_{n}\}$ with $b(\m d')\geq 0$ and $b(\m d)> 0$, then the theorem is implied. Expand the right hand side of equation (\ref{eq:expand}), and then by lemma (\ref{lem:zero}), $L(\m d,\m f)$ is a linear combination of $\iii(\w^{\m d}\gpx(\mathbf{c})\w^{\m f})$ with non-negative coefficients where $\mathbf{0\leq c\leq d}$. Note that $\iii(\w^{\m d}\gpx(\mathbf{c})\w^{\m f})=\iii(\w^{\mathbf{d-c}}\lambda \w^{\mathbf{f+c}})$, and that $\mathbf{(d-c;f+c)}$ satisfies the assumption in lemma (\ref{lem:closure}). So by lemma (\ref{lem:closure}), $\iii(\w^{\mathbf{d-c}}\lambda\w^{\mathbf{f+c}})=\iii(\w^{\m d'}\lambda \w^{\mathbf{f+d-d'}})$  where $\m d'$ satisfies the conditions in our theorem, so we have (\ref{biglsmalll}). To show that $b(\m d)\neq 0$, note that when $x\in (\mathcal{O}^{*})^{m}$, \\$\iii(\w^{\m d} \gpx(x) \w^{\m f})=l(\mathbf{d,f})$, so $b(\m d)\geq (1-\ww)^{m}>0$, completing our proof.
\end{proof}

\section{Normalization}
\label{sec:step8}

By equation (\ref{eq:final}), we have
\begin{equation*}
     \iii(e)=(1-\ww)^{m}\Gamma(\chag,\cham)\cdot\sum_{w\in \gpw_{\gpg},w'\in\gpw_{\gpm}}\mathbf{b}(w\chag,w'\cham)\mathbf{d}(w\chag)\mathbf{d'}(w'\cham).
\end{equation*}
Now we calculate the value of $\iii(e)$. The method is similar to that in section 11 in \cite{MR1956080}. Our conclusion is
\begin{thm}
\label{thmid}
  The value of $\iii$ at identity is given by
  \begin{equation*}
%    \begin{split}
       \iii(e)=(1-\ww)^{m}\Gamma(\chag,\cham)\zeta(1)^{m}\prod_{i=1}^{m}\zeta^{-1}(2i).
 %   \end{split}
  \end{equation*}
\end{thm}

Let $\m C=\sum_{w\in \gpw_{\gpg},w'\in\gpw_{\gpm}}\mathbf{b}(w\chag,w'\cham)\mathbf{d}(w\chag)\mathbf{d'}(w'\cham)$, then what we need to show is
\begin{equation*}
  \m C=\zeta(1)^{m}\prod_{i=1}^{m}\zeta^{-1}(2i).
\end{equation*}
In this section we simply denote $\chag_{i}(\w)$ (resp. $\cham_{j}(\w)$) as $\chag_{i}$ (resp. $\cham_{j}$)  for convenience. Let $\mathbf{b}_{1}(\chag)$ be a function defined on $\chag\in\mathbb{C}^{n}$ and $\m b_{2}(\cham)$ a function on $\cham\in\mathbb{C}^{m}$, we define
\begin{equation*}
  \aveg \left(\m b_{1}(\chag)\right)=\sum_{w\in \gpw_{\gpg}}sgn(w)\m b_{1}(w\chag).
\end{equation*}
and
\begin{equation*}
  \avem (\m b_{2}(\cham))=\sum_{w'\in \gpw_{\gpm}}sgn(w')\m b_{2}(w'\cham).
\end{equation*}
Then it is not hard to see that
\begin{lem}
  For $w\in \gpw_{\gpg}$ and $w'\in \gpw_{\gpm}$, we have
  \begin{equation*}
    \begin{split}
      &\aveg(\m b_{1}(w\chag))=sgn(w)\aveg(\m b_{1}(\chag)),\\
      &\avem(\m b_{2}(w'\cham))=sgn(w')\avem(\m b_{2}(\cham)).
    \end{split}
  \end{equation*}

\end{lem}

  For $\epsilon\in \mathbb{Z}^{n}$ and $\mu\in \mathbb{Z}^{m}$, let  $\chag^{\epsilon}=\prod_{i}\chag_{i}^{\epsilon_{i}}$ and $\cham^{\mu}=\prod_{j}\cham_{j}^{\mu_{j}}$. We say $\epsilon$ (resp. $\mu$) is regular if $w\epsilon=\epsilon$ (resp. $w'\mu=\mu$) implies  $w=e$ (resp. $w'=e$). Then we have 
\begin{lem}
 For non-regular $\epsilon$ and $\mu$,
  \begin{equation*}
    \aveg(\chag^{\epsilon})=0;\qquad \avem(\cham^{\mu})=0.
  \end{equation*}
\end{lem}
Let $\rho_{1}=(n-\frac 1 2, n-\frac 3 2,..., \frac 1 2)$ and let $\rho_{2}=(m,m-1,...,1)$. Let
\begin{equation*}
  \mathrm{P}(\chag)=\chag^{\rho_{1}}; \qquad \mathrm{Q}(\cham)=\cham^{\rho_{2}}.
\end{equation*}
Note that $\rho_{1}$ is the half sum of postive roots in $\hat{\gpg}=SO_{2n+1}(\mathbb{C})$, and $\rho_{2}$ is that of $\gpm=SP_{2m}$, and note that $\gpw_{\gpg}=\gpw_{\hat{\gpg}}$, so, by the Weyl character formula on $\hat{\gpg}$ and on $\gpm$, we have 
\begin{equation*}
 \m d(\chag)=\dfrac{(-1)^{n}}{\mathrm{P}(\chag)\aveg(\mathrm{P}(\chag))},\quad \m d'(\cham)=\dfrac{(-1)^{m}}{\mathrm{Q}(\cham)\avem(\mathrm{Q}(\cham))}.
\end{equation*}
From which we know that
\begin{equation*}
\m d(w\chag)=\dfrac{(-1)^{n}sgn(w)}{\mathrm{P}(w\chag)\aveg(\mathrm{P}(\chag))},\quad \m d'(w'\cham)=\dfrac{(-1)^{m}sgn(w')}{\mathrm{Q}(w'\cham)\avem(\mathrm{Q}(\cham))}.
\end{equation*}
So we have
\begin{equation}
\label{eq:8.1}
  \mathrm{C}=(-1)^{m+n}\dfrac{(\aveg\circ\avem)(\m b(\chag,\cham)\mathrm{P}(\chag^{-1})\mathrm{Q}(\cham^{-1}))}{\aveg(\mathrm{P}(\chag))\avem(\mathrm{Q}(\cham))}.
\end{equation}
Note that $sgn(w_{0}^{\gpg})sgn(w_{0}^{\gpm})=(-1)^{m+n}$, $w_{0}^{\gpg}(\chag)=\chag^{-1}$ and $w_{0}^{\gpm}(\cham)=\cham^{-1}$. So
\begin{equation*}
  \m C=\dfrac{(\aveg\circ\avem)(\m b(\chag^{-1},\cham^{-1})\mathrm{P}(\chag)\mathrm{Q}(\cham))}{\aveg(\mathrm{P}(\chag))\avem(\mathrm{Q}(\cham))}.
\end{equation*}

To calculate $\m C$, we need to simplify $(\aveg\circ\avem)(\m b(\chag^{-1},\cham^{-1})\mathrm{P}(\chag)\mathrm{Q}(\cham))$. 
\begin{lem}
\label{lem:104}
  Let \begin{equation*}
  \begin{split}
    A(\chag,\cham)=\dfrac{\m b(\chag^{-1},\cham^{-1})\mathrm{P}(\chag)\mathrm{Q}(\cham)}{\prod_{i=1}^{n-m}\prod_{j=1}^{m}\zeta^{-1}(-\chag_{i}+\cham_{j}+\frac{1}{2})\zeta^{-1}(-\chag_{i}-\cham_{j}+\frac{1}{2})}
   \end{split}
 \end{equation*}
then 
\begin{equation*}
  \m C=\dfrac{(\aveg\circ\avem)(A(\chag,\cham))}{\aveg(\mathrm{P}(\chag))\avem(\mathrm{Q}(\cham))}.
\end{equation*}
\end{lem}
\begin{proof}
Consider $\m b(\chag^{-1},\cham^{-1})\mathrm{P}(\chag)\mathrm{Q}(\cham)$, wihch is equal to
\begin{equation*}
  \begin{split}
    &\prod_{i<j+n-m}\zeta^{-1}(-\chag_{i}+\cham_{j}+\frac{1}{2})\cdot \prod_{i>j+n-m}\zeta^{-1}(+\chag_{i}-\cham_{j}+\frac{1}{2})\\
    &\prod_{ 1\leq j \leq m}\zeta^{-1}(-\cham_{j}+\frac 1 2)\cdot \prod_{\substack{1\leq i \leq n\\1\leq j \leq m}}\zeta^{-1}(-\chag_{i}-\cham_{j}+\frac{1}{2})\cdot \chag^{\rho_{1}}\cham^{\rho_{2}}.
  \end{split}
\end{equation*}
Let $\m b(\chag^{-1},\cham^{-1})\mathrm{P}(\chag)\mathrm{Q}(\cham)=\sum_{\epsilon,\mu}c_{\epsilon,\mu}\cdot \chag^{\epsilon}\cham^{\mu}$, then
\begin{equation*}
  (\aveg\circ\avem)(\m b(\chag^{-1},\cham^{-1})\mathrm{P}(\chag)\mathrm{Q}(\cham))=\sum_{\epsilon,\mu \text{ regular}}c_{\epsilon,\mu}\cdot (\aveg\circ\avem)(\chag^{\epsilon}\cham^{\mu}).
\end{equation*}
By considering the power of $\chag$, for those $(\epsilon,\mu)$ with $c_{\epsilon,\mu}\neq 0$, we have
\begin{equation*}
\epsilon_{i}\in
  \begin{cases}
    [\frac 1 2+(n-i)-2m, \frac{1}{2}+(n-i)]& \text{ when }1\leq i \leq n-m\\
    [\frac{1}{2}-m, m-\frac{1}{2}]&\text { when } n-m+1\leq i \leq n
      \end{cases}
\end{equation*}
Suppose $\epsilon$ is regular, then  for $n-m+1\leq i \leq n$, $  \{|\epsilon_{i}|\}$ must be a permutation of $\{ \frac 1 2, \frac 3 2,..., m-\frac 1 2\}$, which implies $\epsilon_{n-m}=\frac{1}{2}+m$, which then implies $\epsilon_{n-m-1}=\frac 1 2 + (m+1)$, and so on. Eventually we have, $\epsilon_{i}=\frac{1}{2}+n-i$ for $1\leq i \leq n-m$. In other words, for $1\leq i \leq n-m$, $\epsilon_{i}$ should attain its upper bound. This implies that $ (\aveg\circ\avem)(\m b(\chag^{-1},\cham^{-1})\mathrm{P}(\chag)\mathrm{Q}(\cham))$ remains the same if we divides $ \prod_{i=1}^{n-m}\prod_{j=1}^{m}\zeta^{-1}(-\chag_{i}+\cham_{j}+\frac{1}{2})\zeta^{-1}(-\chag_{i}-\cham_{j}+\frac{1}{2})$ from $\m b(\chag^{-1},\cham^{-1})\mathrm{P}(\chag)\mathrm{Q}(\cham)$, which equals $A(\chag,\cham)$.
\end{proof}

Let $\ti \chag=(\ti \chag_{1},\ldots, \ti \chag_{m})\in \mathbb{C}^{m}$ such that $\ti \chag_{i}=\chag_{n-m+i}$. Let
\begin{equation*}
  \begin{split}
    &\mathrm{\ti A}(\ti\chag,\cham)=\dfrac{A (\chag,\cham)}{\prod_{i=1}^{n-m} \chag_{i}^{\frac{1}{2}+(n-i)}}\\
    =&\prod_{i<j}\zeta^{-1}(-\ti{\chag}_{i}+\cham_{j}+\frac{1}{2})\prod_{i>j}\zeta^{-1}(\ti{\chag}_{i}-\cham_{j}+\frac{1}{2})\\
    &\prod_{ 1\leq j \leq m}\zeta^{-1}(-\cham_{j}+\frac 1 2)\prod_{1 \leq i,j \leq m}\zeta^{-1}(-\ti{\chag}_{i}-\cham_{j}+\frac{1}{2}) {\ti\chag}^{\ti \rho_{1}} \cham^{\rho_{2}},
  \end{split}
\end{equation*}
where $\ti \rho_{1}=(m-\frac 1 2, n-\frac 3 2,..., \frac 1 2)$. Then we have the following lemma.
\begin{lem}
\label{lem:105}
  If we let $\mathrm{P}(\ti \chag)=\ti\chag^{\ti\rho_{1}}$, and let $$(\avem\circ\avem)(\mathrm{\ti A}(\ti\chag,\cham))=\sum_{w_{1},w_{2}\in \gpw_{\gpm}}sgn(w_{1})sgn(w_{2})\mathrm{\ti A}(w_{1}\ti\chag,w_{2}\cham),$$ then

\begin{equation*}
  \m C=\dfrac{(\avem\circ\avem)(\mathrm{\ti A}(\ti\chag,\cham))}{\avem (\mathrm{\ti P}(\ti \chag))\avem (\mathrm{Q}(\cham))}.
\end{equation*}
Moreover, the value of $\mathbf{C}$ is independent of $(\chag,\cham)$.
\end{lem}
\begin{proof}
What we actually need to prove is
 \begin{equation*}
  \dfrac{(\aveg\circ\avem)(A(\chag,\cham))}{\aveg(\mathrm{P}(\chag))\avem(\mathrm{Q}(\cham))}=\dfrac{(\avem\circ\avem)(\mathrm{\ti A}(\ti\chag,\cham))}{\avem (\mathrm{\ti P}(\ti \chag))\avem (\mathrm{Q}(\cham))}
\end{equation*}
and that it is independent of $(\chag,\cham)$.
Let $A(\chag,\cham)=\sum_{\epsilon,\mu}d_{\epsilon,\mu}\chag^{\epsilon}\cham^{\mu}$. Then by the discussion in the previous lemma,
\begin{equation*}
  \begin{split}
    &\epsilon_{i}=\frac 1 2 +n -i, \text{ for }1 \leq i \leq n-m\\
    &\epsilon_{i}\in [-(m-\frac 1 2),m-\frac 1 2], \text{ for } i>n-m.
  \end{split}
\end{equation*}
So for each regular $\epsilon$, if we let $\epsilon'=(\epsilon_{n-m+1},\ldots,\epsilon_{n})\in \mathbb{Z}^{m}$, then
\begin{enumerate}
\item $\epsilon=(n-\frac 1 2, n-\frac 3 2,\ldots ,m+\frac 1 2; \epsilon')$.
\item Note that $\chag^{\epsilon}=\prod_{i=1}^{n-m} \chag_{i}^{\frac{1}{2}+(n-i)}\cdot\ti\chag^{\epsilon'}$, so
  \begin{equation*}
    \mathrm{\ti A}(\ti\chag,\cham)=\sum_{\epsilon,\mu}d_{\epsilon,\mu}\ti\chag^{\epsilon'}\cham^{\mu}
  \end{equation*}
\item $\epsilon$ is regular if and only if $\epsilon'$ is regular with respect to the action of $\gpw_{\gpm}$. So when $\epsilon$ is regular, $\{|\epsilon_{n-m+1}|,\ldots,|\epsilon_{n}|\}$ is a permutation of $\{\frac 1 2, \ldots, m-\frac 1 2\}$, and hence there exists $w_{1}\in \gpw_{\gpm}$, such that $\epsilon=w_{1}\rho_{1}$ and $\epsilon'=w_{1}\ti\rho_{1}$.
\end{enumerate}

By a similar consideration on the power of $\cham$, if $\mu$ is regular, we have
\begin{equation*}
  \{|\mu_{1}|,\ldots,|\mu_{m}|\}\text{ is a permutation of }\{1,\ldots,m\}.
\end{equation*}
So for $(\epsilon,\mu)$ being regular, we let $\epsilon=w_{1}\rho_{1}$ (so $\epsilon'=w_{1}\ti\rho_{1}$), and $\mu=w_{2}\rho_{2}$, where $w_{1},w_{2}\in \gpw_{\gpm}$, and we let $d_{w_{1},w_{2}}=d_{\epsilon,\mu}$. Then
\begin{equation*}
  \begin{split}
    &(\aveg\circ\avem)(A(\chag,\cham))=\sum_{\epsilon,\mu\text{ regular}}d_{\epsilon,\mu}(\aveg\circ\avem)(\chag^{\epsilon}\cham^{\mu})\\
    =&\sum_{w_{1},w_{2}\in \gpm}d_{w_{1},w_{2}}(\aveg\circ\avem)\left((\chag)^{w_{1}\rho_{1}}(\cham)^{w_{2}\rho_{2}}\right)\\
    =&\aveg(\mathrm{P}(\chag))\avem(\mathrm{Q}(\cham))\sum_{w_{1},w_{2}\in \gpm}d_{w_{1},w_{2}}sgn(w_{1})sgn(w_{2}).
  \end{split}
\end{equation*}
and
\begin{equation*}
   \begin{split}
    &(\avem\circ\avem)(\mathrm{\ti A}(\ti \chag,\cham))=\sum_{\epsilon,\mu\text{ regular}}d_{\epsilon,\mu}(\avem\circ\avem)(\ti \chag^{\epsilon'}\cham^{\mu})\\
    =&\sum_{w_{1},w_{2}\in \gpm}d_{w_{1},w_{2}}(\avem\circ\avem)\left((\ti \chag)^{w_{1}\ti\rho_{1}}(\cham)^{w_{2}\rho_{2}}\right)\\
    =&\avem(\mathrm{\ti P}(\ti \chag))\avem(\mathrm{Q}(\cham))\sum_{w_{1},w_{2}\in \gpm}d_{w_{1},w_{2}}sgn(w_{1})sgn(w_{2}).
  \end{split}
\end{equation*}
So
\begin{equation*}
  \begin{split}
    \m C=&\dfrac{(\aveg\circ\avem)(A(\chag,\cham))}{\aveg(\mathrm{P}(\chag))\avem(\mathrm{Q}(\cham))}\\
    =&\sum_{w_{1},w_{2}\in \gpm}d_{w_{1},w_{2}}sgn(w_{1})sgn(w_{2})\\=
    &\dfrac{(\avem\circ\avem)(\mathrm{\ti A}(\ti \chag,\cham))}{\avem(\mathrm{\ti P}(\ti \chag))\avem(\mathrm{Q}(\cham))},
  \end{split}
\end{equation*}
completing our proof since $d_{w_{1},w_{2}}$ is independent of $(\chag,\cham)$.
\end{proof}

So to find the value of $\m C$ one only needs to find the value of
\begin{equation*}
  \dfrac{(\avem\circ\avem)(\mathrm{\ti A}(\ti \chag,\cham))}{\avem(\mathrm{\ti P}(\ti \chag))\avem(\mathrm{Q}(\cham))}
\end{equation*}
for a special $(\chag,\cham)$. From now on we let $\ti\chag_{i}=-(m+1-i)$ and $\cham_{j}=-(m+\frac 1 2 -j)$, then we have the following lemma
\begin{lem}
  If $\mathrm{\ti A}(w\ti\chag,w'\cham)\neq 0$, then  $w=w'=e$.
\end{lem}
\begin{proof}
We denote $\ti\chag_{-i}=-\ti\chag_{i}$ and $\cham_{-j}=\cham_{j}$. For any weyl element $w,w'\in W_{\gpm}$, there exists $\sigma,\sigma'$, permutations of the set $\{\pm 1,\pm 2,\ldots,\pm m\}$ such that
\begin{enumerate}
\item $\sigma(-i)=-\sigma(i)$, $\sigma'(-j)=-\sigma'(j)$.
\item $(w{\ti\chag})_{i}={\ti\chag}_{\sigma_{i}}$, $(w'\cham)_{j}=\cham_{\sigma'_{j}}$.
\end{enumerate}
So $w=w'=e$ if and only if $\sigma=\sigma'=\mathrm{id}$. Now suppose $\mathrm{\ti A}(w\ti\chag,w'\cham)\neq 0$, since $\zeta^{-1}(0)=0$, we $\sigma$ and $\sigma'$ should satisfy the following properties:
\begin{enumerate}
\item For any $i<j$, $\ti{\chag}_{\sigma_{i}}-\cham_{\sigma'(j)}\neq \frac 1 2$. 
\item For any $i>j$, $\ti{\chag}_{\sigma_{i}}-\cham_{\sigma'(j)}\neq -\frac 1 2$.
\item For any $1\leq j\leq m$, $\cham_{\sigma'(j)}\neq \frac 1 2$.
\item For any $1\leq i,j \leq m$, $\ti{\chag}_{\sigma_{i}}+\cham_{\sigma'(j)}\neq \frac 1 2$ 
\end{enumerate}
These four conditions actually imply $\sigma=\sigma'=1$. To see this we consider $ A= \{\ti{\chag}_{\sigma(i)},1\leq i \leq m\}$ and $ B= \{\cham_{\sigma'(j)},1\leq j \leq m\}$. Note that $\{|\ti{\chag}_{\sigma(1)}|,\ldots,|\ti{\chag}_{\sigma(m)}|\}=\{1,\ldots,m\}$ and $\{|\cham_{\sigma(1)}|,\ldots,|\cham_{\sigma(m)}|\}=\{\frac 1 2,\ldots, (m-\frac 1 2)\}$. By property (3), $\frac 1 2\notin B$, which implies $-\frac 1 2\in B$. Then by property (4), $1 \notin A$. So then $-1 \in A$. Then again by property (4), $\frac 3 2\notin B$, so then $-\frac{3}{2}\in B$. And then by property (4), $2\notin A$, so $-2\in A$. Continuing this process, we will eventually have $A=\{-m,-(m-1),...,-1\}$ and $  B=\{-m+\frac 1 2, -m+\frac 3 2,..., -\frac 1 2\}$. So $\sigma$ and $\sigma'$ are actually permutations of $\{1,2,...,m\}$. Note that $\ti{\chag}_{k}-\cham_{k}=-\frac 1 2$, so by property (2), $\sigma(i)\neq \sigma'(j)$ for any $i>j$, which implies that
\begin{equation}
\label{eq:8.5}
  \sigma^{-1}(k)\leq (\sigma')^{-1}(k)
\end{equation}
when $1 \leq k \leq m$. On the other hand, since $\ti{\chag}_{k+1}-\cham_{k}=\frac 1 2$, we have
\begin{equation}
\label{eq:8.6}
  \sigma^{-1}(k+1)\geq (\sigma')^{-1}(k).
\end{equation}
by property (1). Combining equation (\ref{eq:8.5}) and (\ref{eq:8.6}), we have 
\begin{equation*}
  (\sigma')^{-1}(m)\geq \sigma^{-1}(m)\geq (\sigma')^{-1}(m-1)\geq \sigma^{-1}(m-1)\geq... \geq (\sigma')^{-1}(1)\geq \sigma^{-1}(1),
\end{equation*}
which implies that $\sigma=\sigma'=\mathbf{id}$.
\end{proof}

By the lemma,
\begin{equation*}
  \m C=\dfrac{\mathrm{\ti A}(\ti\chag,\cham)}{\avem(\mathrm{\ti P}(\ti\chag))\avem(\mathrm{Q}(\cham))}
\end{equation*}
where $\ti\chag_{i}=-(n+1-i)$ and $\cham_{j}=-(m+\frac 1 2 -j)$.
Note that by the Weyl character formula,
\begin{equation*}
  \begin{split}
    &{\avem(\mathrm{\ti P}(\ti\chag))}=\prod_{1 \leq a< b \leq m }\zeta^{-1}(-\ti\chag_{a}+\ti\chag_{b})\zeta^{-1}(-\ti\chag_{a}-\ti\chag_{b})\prod_{i=1}^{n}\zeta^{-1}(-\ti\chag_{i}){\mathrm{\ti P}(\ti\chag)};\\
    &{\avem(\mathrm{Q}(\cham))}=\prod_{1\leq a<b \leq m}\zeta^{-1}(-\cham_{a}+\cham_{b})\zeta^{-1}(-\cham_{a}-\cham_{b})\prod_{j=1}^{m}\zeta^{-1}(-2\cham_{j}){\mathrm{Q}(\cham)}.
    \end{split}
\end{equation*}
By direct calculation we have
\begin{equation*}
  \mathbf{C}=\zeta(1)^{m}\prod_{i=1}^{m}\zeta^{-1}(2i).
\end{equation*}

\section{Uniqueness of the Whittaker Shintani function }
\label{sec:step9}

In this section we show the following theorem.
\begin{thm}
\label{thmunique}
  Let $\wsfb$ be a Whittaker Shintani function on $\gpg$. Let $\mathcal{W}(\m d, \m f)=\wsfb(\w^{\m d}\lambda \w^{\m f})$ with $\m d\in \Lambda^{+}_{m}$ and $\m f\in \Lambda^{+}_{n}$. If $\mathcal{W}(\m 0,\m 0)=0$, then $\mathcal{W}(\m d,\m f)=0$ for every $\m d\in \Lambda^{+}_{m}$ and $\m f\in \Lambda^{+}_{n}$. 
\end{thm}
Combine this with theorem (\ref{thm:step4}) we know that for all $(\chag,\cham)$, the space of Whittaker Shintani function $\ws$ is of at most one dimensional. The method we use in the proof is from \cite{MR1956080} and \cite{MR1121142}. First we define an order on $\Lambda^{+}_{n}\times \Lambda^{+}_{m}$ as
\begin{define}
  Let $\varpi_{k}$, $\varpi_{l}'$ be the dominant weights of $\gpg$ and $\gpm$. For any\\ $\mathbf{(d',f'),(d,f)}\in \Lambda^{+}_{m}\times \Lambda^{+}_{n}$, we write $\mathbf{(d,f)\geq_{\mathcal{WS}} (d',f')}$ if
  \begin{enumerate}
  \item $\ppair{\varpi_{l}}{\m f}\geq\ppair{\varpi_{l}}{\m f'}$ for $1\leq k\leq n-m$
  \item $\ppair{\varpi_{l}}{\m f}+\ppair{\varpi_{l-(n-m)}'}{\m d}\geq \ppair{\varpi_{l}}{\m f'}+\ppair{\varpi_{l-(n-m)}}{\m d'}$ for $n-m+1\leq l\leq n$
  \item $\ppair{\varpi_{n-m+l-1}}{\m f}+\ppair{\varpi_{l}'}{\m d}\geq \ppair{\varpi_{n-m+l-1}}{\m f'}+\ppair{\varpi_{l}}{\m d'}$ for $1\leq l\leq m$
   \end{enumerate}
\end{define}
Then we have the following lemma
\begin{lem}
\label{lemgroup}
  Suppose $\mathbf{(d,f)}\in \Lambda^{+}_{m}\times \Lambda^{+}_{n}$. 
  \begin{enumerate}
  \item If $\mathbf{(d',f')}\in \Lambda^{+}_{m}\times \Lambda^{+}_{n}$ satisfies
    \begin{equation*}
      \gpk_{\gpm^{\gpj}}\w^{\m d}\gpk_{\gpg}\w^{\m f}\gpk_{\gpg}\cap \gpz\gpu\gpk_{\gpm^{\gpj}}\w^{\m d'}\lambda\w^{\m f'}\gpk_{\gpg}\neq \emptyset,
    \end{equation*}
then $\mathbf{(d,f)}\geq_{\mathcal{WS}} \mathbf{(d',f')}$.
\item If $u\in \gpu$ and $z\in \gpz$ satisfies
  \begin{equation*}
    \gpk_{\gpm^{\gpj}}\w^{\m d}\gpk_{\gpg}\w^{\m f}\gpk_{\gpg}\cap zu\gpk_{\gpm^{\gpj}}\w^{\m d}\lambda\w^{\m f}\gpk_{\gpg}\neq \emptyset,
  \end{equation*}
then $\psi_{\gpu}(u)=1$ and $\psi(z)=1$.
  \end{enumerate}
\end{lem}

Before proving the lemma (\ref{lemgroup}), we first show it implies theorem (\ref{thmunique}).

\begin{proof}[Proof of theorem (\ref{thmunique})]
 Consider $\int_{\gpk_{\gpm^{\gpj}}\w^{\m d}\gpk_{\gpg}\w^{\m f}\gpk_{\gpg}}\ud g \whi_{\chag,\cham,\psi}(g)$. 
On the one hand it is equal to
\begin{equation*}
  \omega_{\chag}(\mathrm{Ch}_{\gpk_{\gpg}\w^{\m f}\gpk_{\gpg}})\omega_{\cham}(\mathrm{Ch}_{\gpk_{\gpm^{\gpj}}\w^{\m d}\gpk_{\gpm^{\gpj}}})\mathcal{W}(\m 0,\m 0).
\end{equation*}
On the other hand, by lemma (\ref{lemgroup}) and themrem (\ref{thm:step4}) it is equal to
\begin{equation*}
  \sum_{(\m d',\m f')\leq_{\mathcal{WS}} (\m d,\m f), (\m d',\m f')\in \Lambda^{+}_{m}\times \Lambda^{+}_{n}}C_{\m d',\m f'}\mathcal{W}(\m d',\m f').
\end{equation*}
Where $C_{\m d, \m f}$ is positive by the second part of the lemma (\ref{lemgroup}). So if $\mathcal{W}(\m 0,\m 0)=0$, we have
\begin{equation*}
  \sum_{(\m d',\m f')\leq_{\mathcal{WS}} (\m d,\m f), (\m d',\m f')\in \Lambda^{+}_{m}\times \Lambda^{+}_{n}}C_{\m d',\m f'}\mathcal{W}(\m d',\m f')=0
\end{equation*}
for all $(\m d, \m f)\in \Lambda^{+}_{m}\times \Lambda^{+}_{n}$. Taking the induction on $(\m d,\m f)$ by the order $\leq_{\mathcal{WS}}$ we have
  \begin{equation*}
    \mathcal{W}(\m d, \m f)=0
  \end{equation*}
for all $(\m d, \m f)\in \Lambda^{+}_{m}\times \Lambda^{+}_{n}$, completing our proof. 
\end{proof}

So in the rest of this section we only need to prove lemma (\ref{lemgroup}). To prove the first part of lemma (\ref{lemgroup}), we need the following lemma.
\begin{lem}
  Let $\mathcal{N}_{2n}=\{1,2,\ldots,2n\}$, and let  $g,g^{1},g^{2},g^{3}\in \gpg$.  For $I=(i_{1},\ldots, i_{k})$, $J=(j_{1},\ldots, j_{k})\in (\mathcal{N}_{2n})^{k}$, we denote $f_{I,J}(g)=\prod_{s=1}^{k}g_{i_{s},j_{s}}$, and
\begin{equation*}
  \Delta_{I,J}(g)=\det(g_{I,J})=\sum_{\sigma\in S_{k}}sgn(\sigma)\prod_{s=1}^{k}g_{i_{s},j_{\sigma(s)}}.
\end{equation*}
If $g=g^{1}g^{2}g^{3}$, then we have
\begin{equation}
  \label{eqhopf}
  \Delta_{I,J}(g)=\sum_{A,C\in \mathcal{N}_{2n}^{k}}f_{I,A}(g^{1})\Delta_{A,C}(g^{2})f_{C,J}(g^{3}).
\end{equation}
\end{lem}
\begin{proof}
 Since $g=g^{1}g^{2}g^{3}$, we have
\begin{equation}
  \label{eq123}
  g_{i,j}=\sum_{a,b}g^{1}_{i,a}g^{2}_{a,b}g^{3}_{b,j},
\end{equation}
where $a,b$ runs over $\mathcal{N}_{2n}^{k}$. So
\begin{equation}
\label{eqggg}
  \Delta_{I,J}(g)=\sum_{\sigma\in S_{k}}sgn(\sigma)\sum_{(a_{1},\ldots,a_{k}),(b_{1},\ldots,b_{k})\in (\mathcal{N}_{2n})^{k}}\ \prod_{s=1}^{k}g^{1}_{i_{s},a_{s}}g^{2}_{a_{s},b_{s}}g^{3}_{b_{s},j_{\sigma(s)}}.
\end{equation}
Note that $S_{k}$ acts on $(\mathcal{N}_{2n})^{k}$. If we define $c_{s}=b_{\sigma^{-1}(s)}$, then 
\begin{equation*}
  \begin{split}
    &\sum_{(a_{1},\ldots,a_{k}),(b_{1},\ldots,b_{k})\in (\mathcal{N}_{2n})^{k}}\ \prod_{s=1}^{k}g^{1}_{i_{s},a_{s}}g^{2}_{a_{s},b_{s}}g^{3}_{b_{s},j_{\sigma(s)}}\\
    =&\sum_{(a_{1},\ldots,a_{k}),(c_{1},\ldots,c_{k})\in (\mathcal{N}_{2n})^{k}}\ \prod_{s=1}^{k}g^{1}_{i_{s},a_{s}}g^{2}_{a_{s},c_{\sigma(s)}}g^{3}_{c_{\sigma(s)},j_{\sigma(s)}}.
  \end{split}
\end{equation*}
Note that for any $\sigma\in S_{k}$, $$\prod_{s=1}^{k}g^{3}_{c_{\sigma(s)},j_{\sigma(s)}}=\prod_{s=1}^{k}g^{3}_{c_s,j_{s}}.$$ So (\ref{eqggg}) is equal to
\begin{equation*}
  \begin{split}
    &\sum_{(a_{1},\ldots,a_{k}),(c_{1},\ldots,c_{k})\in (\mathcal{N}_{2n})^{k}}\prod_{s=1}^{k}\left(g^{1}_{i_{s},a_{s}}g^{3}_{c_s,j_{s}}\right)\cdot \sum_{\sigma\in S_{k}}sgn(\sigma)\prod_{s=1}^{k}g^{2}_{a_{s},c_{\sigma(s)}}\\
   =&\sum_{(a_{1},\ldots,a_{k}),(c_{1},\ldots,c_{k})\in (\mathcal{N}_{2n})^{k}}\prod_{s=1}^{k}\left(g^{1}_{i_{s},a_{s}}g^{3}_{c_s,j_{s}}\right)\cdot \Delta_{(a_{1},\ldots,a_{k}),(c_{1},\ldots,c_{k})}(g^{2}).
 \end{split}
\end{equation*}
which is the formula we want to prove.
\end{proof}

By this lemma we can prove the first part of lemma (\ref{lemgroup}). 
\begin{proof}[Proof of first part of lemma (\ref{lemgroup})]
Suppose 
\begin{equation*}
  \gpk_{\gpm^{\gpj}}\w^{\m d}\gpk_{\gpg}\w^{\m f}\gpk_{\gpg}\cap \gpz\gpu\gpk_{\gpm^{\gpj}}\w^{\m d'}\lambda\w^{\m f'}\gpk_{\gpg}\neq \emptyset,
\end{equation*}
then $\w^{\m f'}w_{0}^{\gpg}\lambda\w^{-\m d'}uz=k_{1}\w^{-\m f}k_{2}\w^{-\m d}k$ for some $k_{1},k_{2}\in \gpk_{\gpg}$ and $k\in \gpk_{\gpm^{\gpj}}$. Apply $\alpha_{l}$ and $\beta_{l}$, which are defined as (\ref{eqalpha}) and (\ref{eqbeta}) on page \pageref{eqalpha} to both sides of the equation. When $l\geq n-m+1$, we have
\begin{equation}
\label{in1}
  v(\alpha_{l}(\w^{\m f'}w_{0}^{\gpg}\lambda\w^{-\m d'}uz))=-\ppair{\varpi_{l}}{\m f'}-\ppair{\varpi_{l-(n-m)}'}{\m d'}.
\end{equation}
On the other hand, note that
\begin{equation}
  \alpha_{l}(g)=\Delta_{I_{l},J_{l}}(g).
\end{equation}
where $I_{l}=(2n+1-l,2n+1-(l-1),\ldots, 2n)$, and $J_{l}=(1,2,\ldots,l)$. By (\ref{eqhopf}), we have 

\begin{equation*}
  \alpha_{l}(k_{1}\w^{-\m f}k_{2}\w^{-\m d}k)=\sum_{A,C\in (\mathcal{N}_{2n})^{l}}f_{I,A}(k_{1})\Delta_{A,C}(\w^{-\m f}k_{2}\w^{-\m d})f_{C,J}(k).
\end{equation*}
Note that for $f_{I,A}(k_{1})\Delta_{A,C}(\w^{-\m f}k_{2}\w^{-\m d})f_{C,J}(k)\neq 0$, both $A$ and $C$ should contain distinct coordinates. Moreover, note that $k\in \gpk_{\gpm^{\gpj}}$, so $f_{C,J}(k)\neq 0$ implies $c_{j}=j$ for all $1\leq j\leq n-m$. Under these two restrictions it is not hard to see that
\begin{equation*}
  v(\Delta_{A,C}(\w^{-\m f}k_{2}\w^{-\m d}))\geq -\ppair{\varpi_{l}}{\m f}-\ppair{\varpi_{l-(n-m)}'}{\m d}. 
\end{equation*}
So
\begin{equation}
\label{in2}
  v(\alpha_{l}(k_{1}\w^{-\m f}k_{2}\w^{-\m d}k))\geq -\ppair{\varpi_{l}}{\m f}-\ppair{\varpi_{l-(n-m)}'}{\m d}. 
\end{equation}
Comparing (\ref{in1}) and (\ref{in2}) we have
\begin{equation*}
  \ppair{\varpi_{l}}{\m f'}+\ppair{\varpi_{l-(n-m)}'}{\m d'}\leq \ppair{\varpi_{l}}{\m f}+\ppair{\varpi_{l-(n-m)}'}{\m d}
\end{equation*}
for all $n-m+1\leq l\leq n$. 
Similarly, if we apply $\alpha_{l}$ for $1\leq l \leq n-m$ or $\beta_{l}$ for $1\leq l \leq m$, we will have
\begin{equation*}
  \ppair{\varpi_{l}}{\m f'}\leq \ppair{\varpi_{l}}{\m f}
\end{equation*}
for all $1\leq l \leq n-m$ and 
\begin{equation*}
  \ppair{\varpi_{n-m+l-1}}{\m f'}+\ppair{\varpi_{l}'}{\m d'}\leq \ppair{\varpi_{m-m+l-1}}{\m f}+\ppair{\varpi_{l}'}{\m d} 
\end{equation*}
for all $1\leq l \leq m$. 
So we have the first part of the lemma (\ref{lemgroup}).
\end{proof}

Next we prove the second part of the lemma (\ref{lemgroup}). 
\begin{proof}[Proof of second part of the lemma (\ref{lemgroup})]
Suppose
\begin{equation}
  \label{eq:333}
  \w^{\m d}k^{1}\w^{\m f}=zuk\w^{\m d}\lambda\w^{\m f}k^{2}
\end{equation}
for some $u\in \gpu$, $z\in \gpz$, $k\in \gpk_{\gpm^{\gpj}}$ and $k_{1},k_{2}\in \gpk_{\gpg}$. We need to show that $\psi(z)=\psi_{\gpu}(u)=1$. Consider the element $r_{2n}=(0,0,\ldots,0,1)\in F^{2n}$. Multiplying $r_{2n}$ from the left to both sides of (\ref{eq:333}), we have
\begin{equation*}
  \m k^{1}_{2n}\w^{\m f}=\w^{-f_{1}}\m k^{2}_{2n}.
\end{equation*}
Here $\m k^{1}_{2n}$, $\m k^{2}_{2n}$ are the 2n-th row of $k^{1}$ and $k^{2}$. Suppose $\m k^{1}_{2n}=(k^{1}_{2n,1},\ldots,k^{1}_{2n,2n})$, then
\begin{equation}
\begin{split}
\label{eqpri}
&\m k^{2}_{2n}\\=
&(\w^{f_{1}+f_{1}}k_{2n,1}^{1},\w^{f_{1}+f_{2}}k_{2n,2}^{1}, \ldots,\w^{f_{1}+f_{n}}k_{2n,n}^{1},\w^{f_{1}-f_{n}}k_{2n,n+1}^{1},\ldots,\w^{f_{1}-f_{2}}k_{2n,2n-1}^{1},k_{2n,2n}^{1}).
\end{split}
\end{equation}
Note that when $k^{2}\in \gpk_{\gpg}$, each row of it is primitive, that is, it belongs to $\mathcal{O}^{2n}$, but not $(\w\mathcal{O})^{2n}$. So suppose $f_{1}=f_{2}=\cdots f_{k}>f_{k+1}\geq \cdots \geq f_{n}$ for some $k$, then by (\ref{eqpri}), at least one element in $\{k_{2n,2n-k+1}^{1},\ldots,k_{2n,2n}^{1}\}$ belongs to $\mathcal{O}^{*}$. Let's say it is $k^{1}_{2n,2n-i+1}$. Let $w$ be an weyl element of $\gpg$ transposing $1$ and $i$, then $k^{1}w^{-1}$ has element in $\mathcal{O}^{*}$ at the $(2n,2n)$ position. Then, by the Bruhat decomposition of $\gpk\ (mod\ \w)$, there exists $x_{1},x_{2},y_{1},y_{2}\in \mathcal{O}^{n-1}$, and $z_{1},z_{2}\in \mathcal{O}$, such that
\begin{equation*}
  k^{1}w^{-1}=E_{1}(x_{1},y_{1},z_{1})
  \begin{pmatrix}
    \epsilon&& \\ & k' & \\ && \epsilon^{-1}
  \end{pmatrix}
E_{1}(x_{2},y_{2},z_{2})^{w_{0}^{G}},
\end{equation*}
where $\epsilon\in \mathcal{O^{*}}$, $k'\in \gpk_{Sp_{2n-2}}$, and
\begin{equation*}
  E_{1}(x,y,z)=
  \begin{pmatrix}
    1&
    \begin{matrix}
      x&y
    \end{matrix}
    &z\\
    &I_{2n-2}&
    \begin{matrix}
      ^{t}y\\-^{t}x
    \end{matrix}
    \\
    &&1
  \end{pmatrix}
\end{equation*}
So (\ref{eq:333}) becomes
\begin{equation*}
  E_{1}(x_{1},y_{1},z_{1})^{\w^{\m d}}\left(\w^{\m d}\begin{pmatrix}
    \epsilon&& \\ & k' & \\ && \epsilon^{-1}
  \end{pmatrix}
\w^{\m f}\right)\w^{-\m f}E_{1}(x_{2},y_{2},z_{2})^{w_{0}^{G}}w\w^{\m f}=uzk\w^{\m d}\lambda \w^{\m f}k^{2}.
\end{equation*}
By the definition of $w$, it commutes with $\w^{\m f}$, so $\w^{-\m f}E_{1}(x_{2},y_{2},z_{2})^{w_{0}^{G}}w\w^{\m f}\in \gpk_{\gpg}$. So we just need to show that  
\begin{equation}
\label{eqhaha}
  E_{1}(x_{1},y_{1},z_{1})^{\w^{\m d}}\left(\w^{\m d}\begin{pmatrix}
    1&& \\ & k' & \\ && 1
  \end{pmatrix}
\w^{\m f}\right)=uzk\w^{\m d}\lambda \w^{\m f}k^{2}
\end{equation} implies $\psi(z)=\psi_{\gpu}(u)=1$. We prove this by induction on $n-m$. 

When $n-m=1$, $\gpu$ is trivial, and $E_{1}(x,y,z)=\gpj(x,y,z)$. By (\ref{eqhaha}), we have $\w^{\m f}k^{2}\w^{-\m f} \in \gpm^{\gpj}$. So $k^{2}\in \gpk_{\gpm^{\gpj}}$. Suppose $k^{2}=n_{1}k''$ where $n_{1}\in \gpj^{0}$ and $k''\in \gpk_{\gpm}$, and suppose $\mathbf{\tilde{f}}=(f_{2},\ldots,f_{n})$, then we have
\begin{equation*}
   \gpj(x_{1},y_{1},z_{1})^{\w^{\m d}}\left(\w^{\m d}\begin{pmatrix}
    1&& \\ & k' & \\ && 1
  \end{pmatrix}
\w^{\mathbf{\tilde{f}}}\right)=zk\w^{\m d}\lambda n_{1}^{\w^{\m f}} \w^{\mathbf{\tilde{f}}}k''
\end{equation*}
Note that now both sides belongs to $\gpm^{\gpj}$, we write both sides in the form of $\gpj\rtimes \gpm$. Then by comparing the $\gpj$-part of both sides we have
\begin{equation*}
  \gpj(x_{1},y_{1},z_{1})^{\w^{\m d}}=z(\lambda n_{1}^{\w^{\m f}})^{k\w^{\m d}}.
\end{equation*}
When $\m f \in \Lambda^{+}_{n}$, $n_{1}^{\w^{\m f}}\in \gpj^{0}$. So we assume $\lambda n_{1}^{\w^{\m f}}=\gpj(x_{3},y_{3},z_{3})$ with $x_{3},y_{3}\in \mathcal{O}^{n-1}$ and $z_{3}\in \mathcal{O}$. Then we have
\begin{equation*}
  z\cdot \gpz(z_{3}-z_{1})=\gpj(x_{1},y_{1},0)^{\w^{\m d}}\gpj(-x_{3},-y_{3},0)^{k\w^{\m d}}
\end{equation*}
When $\gpj(x_{1},y_{1},0)^{\w^{\m d}}\gpj(-x_{3},-y_{3},0)^{k\w^{\m d}}\in \gpz$, we have
\begin{equation*}
  \gpj(-x_{3},-y_{3},0)^{k\w^{\m d}}=\gpj(-x_{1},-y_{1},0)^{\w^{\m d}}.
\end{equation*}
So $z\cdot\gpz(z_{3}-z_{1})=x_{1}\cdot ^{t}\!\!y_{1}\in \gpz^{0}$. Since $z_{1},z_{3}\in \mathcal{O}$, so $z\in \gpz^{0}$, and hence $\psi(z)=1$, completing the proof for $n-m=1$. 

Assume the lemma is true for $n-m=r-1$, and suppose now $n-m=r>1$. Then $E_{1}(x_{1},y_{1},z_{1})^{\w^{\m d}}\in \gpu$. Since $E_{1}(x_{1},y_{1},z_{1})\in \gpu^{0}$ and $\w^{\m d}$ stablizes $\psi_{\gpu}$, we have $\psi_{\gpu}(E_{1}(x_{1},y_{1},z_{1})^{\w^{\m d}})=1$. So from (\ref{eqhaha}) reduces to show that
\begin{equation}
\label{eqhehe}
  \w^{\m d}
\begin{pmatrix}
    1&& \\ & k' & \\ && 1
  \end{pmatrix}
\w^{\m f}=uzk\w^{\m d}\lambda \w^{\m f}k_{2}
\end{equation} implies $\psi(z)=\psi_{\gpu}(u)=1$.
Let
\begin{equation*}
  \gpg'=\{g\in \gpg\mid g=
\begin{pmatrix}
  1 & * & *\\ & g' &*\\ && 1
\end{pmatrix}\}\cong \mathrm{Sp}_{2n-2}\ltimes \mathcal{H}_{2n-1}.
\end{equation*}
Then by (\ref{eqhehe}), $\w^{\m f}k_{2}\w^{-\m f}\in \gpg'$, so $k_{2}\in \gpk_{\gpg'}$. Let $k_{2}=\tilde{u}k''$ where $\tilde{u}\in \mathcal{H}_{2n-1}$ and $k''\in \gpk_{\mathrm{Sp}_{2n-2}}$. Then we have
\begin{equation*}
  \w^{\m d}
\begin{pmatrix}
    1&& \\ & k' & \\ && 1
  \end{pmatrix}
\w^{\tilde{\mathbf{f}}}=uzk\w^{\m d}\lambda\cdot  (\tilde{u})^{\w^{\m f}}\w^{\tilde{\mathbf{f}}}k''
\end{equation*}
Suppose $u=u^{1}u^{2}$ where $u^{1}\in \mathcal{H}_{2n-1}$ and $u^{2}\in \mathrm{Sp}_{2n-2}$. Now both sides belongs to $\gpg'$. Write them in the form $\mathcal{H}_{2n-1}\mathrm{Sp}_{2n-2}$ we have
\begin{equation*}
  u^{1}(\tilde{u})^{zk\w^{\m d}\w^{\m f}}=1  
\end{equation*}
and
\begin{equation*}
  u^{2}zk\w^{\m d}\lambda \w^{\tilde{\mathbf{f}}}k''=\w^{\m d}
  \begin{pmatrix}
   1 &  & \\ & k' &\\ && 1
 \end{pmatrix}\w^{\tilde{\m f}}.
\end{equation*}
Note that $\tilde{u}\in \gpu^{0}$, $\m f\in \Lambda^{+}_{n}$, and $zk\w^{\m d}\in \gpm^{J}$ which stablizes $\psi_{\gpu}$,we have $$\psi_{\gpu}(u_{1})=\psi_{\gpu}^{-1}(\tilde{u}^{zu\w^{\m d}\w^{\m f}})=\psi_{\gpu}^{-1}(\tilde{u}^{\w^{\m f}})=1. $$ 
On the other hand, by assumption of the induction, when 
$$ u^{2}zk\w^{\m d}\lambda \w^{\tilde{\mathbf{f}}}k''=\w^{\m d}
  \begin{pmatrix}
   1 &  & \\ & k' &\\ && 1
 \end{pmatrix}\w^{\tilde{\m f}},$$ we have $\psi_{\gpu}(u_{2})=\psi(z)=1$. So $\psi_{\gpu}(u)=\psi_{\gpu}(u_{1}u_{2})=1$ and  $\psi(z)=1$, completing our proof for $n-m=r>1$.
\end{proof}

\section{The formula for normalized Whittaker-Shintani function}
\label{sec:step10}
Let $\overline{l}_{\chag,\cham,\psi}$ be a pairing between $\repfi$ and $\repse$ satisfying \textrm{Condition A} and
\begin{equation*}
  \overline{l}_{\chag,\cham,\psi}(\funf_{\chag}^{0},\funf_{\cham,\psi}^{0})=1
\end{equation*}

Since every such pairing produces a normalized  Whittaker-Shintani function, so it is unique for every $(\chag,\cham)$. By theorem (\ref{thmid}), for those generic $(\chag,\cham)$ where $\Gamma(\chag,\cham)$ has no zeroes or poles, such pairing exists. So applying the Bernstein theorem, and the corollary in section 1 in \cite{MR1671189}, formula (\ref{eq:final}) can be extended to a regular function on $(\chag,\cham)$. Now we summerize our result. 
\begin{thm}
  For every $(\chag,\cham)\in \mathbb{C}^{n}\times \mathbb{C}^{m}$, the normalzied Whittaker-Shintani function is given by
      \begin{equation*}
  \begin{split}
    \int_{\gpx^{0}}\ud x\wsf(\w^{\mathbf{d}} x \w^{\mathbf{f}})&=\zeta(1)^{-m}\prod_{i=1}^{m}\zeta(2i)\cdot\\
    \sum_{w\in W_{\gpg},w'\in W_{\gpm}}&b(w\chag,w'\cham)d(w\chag)d'(w'\cham)((w\chag)^{-1}\dha)(\w^{\m f})((w'\cham)^{-1}\dha)(\w^{\m d}) 
  \end{split}
\end{equation*}
for $\m{d} \in \Lambda_{m}^{+}$ and $\m{f} \in \Lambda_{n}^{+}$. If we let $\mathcal{L}(\m d', \m f')=  \int_{\gpx^{0}}\ud x\wsf(\w^{\mathbf{d'}} x \w^{\mathbf{f'}})$, then 
\begin{equation*}
  \wsf(\w^{\mathbf{d}} \lambda \w^{\mathbf{f}})=\sum_{\m d'}a(\m d')\mathcal{L}(\m d', \mathbf{f+d-d'} ).
\end{equation*}
with $a(\m d')\geq 0$ and $a(\m d)>0$. Here in the summation on  $\m d'$ runs over the set
\begin{equation*}
   \{\m d'\mid\mathbf{d'\in }\Lambda^{+}_{m}, \mathbf{f+d-d'}\in \Lambda^{+}_{n}, \mathbf{d'\leq d}\}.
\end{equation*}
In particuler,
\begin{equation*}
  \wsf(\w^{\mathbf{f}})=\mathcal{L}(\m 0, \mathbf{f} ).
\end{equation*}

\end{thm}

\section{Application}
\label{sec:13}

Using the formula for the Shintani function, we can give an alternative proof of the \textbf{Theorem 6.1} in \cite{MR1121142}. We rewrite the theorem as below.

\begin{thm}[Theorem 6.1 in \cite{MR1121142}, conjectured by T.Shintani]
\label{thm:13}
Let $\gpg=\mathrm{Sp}_{2n}$ and $\gpm=\mathrm{Sp}_{2m}$ as defined in our paper, and suppose $n=m+1$. Let $\pi$ and $\tilde\sigma$ be the unramified representation of $\gpg(F_{v})$ and $\tilde\gpm(F_{v})$ respectively. Let \\$z_{\pi}=(\w^{\chag_{1}},\ldots,\w^{\chag_{m+1}},1,\w^{-\chag_{m+1}},\ldots,\w^{-\chag_{1}})$ be the Satake parameters of $\pi$ and $z_{\tilde\sigma}=(\w^{\cham_{1}},\ldots, \w^{\cham_{m}},\w^{-\cham_{m}},\ldots, \w^{-\cham_{1}})$ be the Satake parameters of $\tilde\sigma$ with respect to $\psi$ so that $\tilde\sigma\otimes \omega_{\psi}\cong \repse$. Let $\wsf$ be the Whittaker-Shintani function as we defined in this paper. Then we have
\begin{equation}\label{eq:a1}
  \int_{\mathrm{GL}_{1}} \wsf
  \begin{pmatrix}
    t && \\ &I_{2m} & \\ && t^{-1}
  \end{pmatrix}
|t|^{s-m-1}\ud t=\frac{L(\pi,s)}{L_{\psi}(\tilde\sigma,s+\frac 1 2)\zeta(2s)}
\end{equation}
\end{thm}

We denote by $\mathrm{LHS}$ and $\mathrm{RHS}$ the left hand side and right hand side of above respectively. First we have 

  \begin{lem} Recall that $\rho_{1}=(n-\frac 1 2, n-\frac 3 2, \ldots, \frac 1 2)$, where $n=m+1$ now, corresponding to the half sum of positive roots in $SO_{2n+1}=SO_{2m+3}$. Then  
    \begin{equation*}
    \begin{split}
      \mathrm{LHS}=\sum_{l\geq 0}\frac{\mathcal{A}_{\gpw_{\gpg}}(\prod_{1\leq j \leq m}(1-\w^{-\chag_{1}\pm \cham_{j}+\frac 1 2})\w^{\ppair{\chag}{l+\rho_{1}}})}{\mathcal{A}_{\gpw_{\gpg}}(\w^{\ppair{\chag}{\rho_{1}}})}\cdot \ww^{ls}.
    \end{split}
  \end{equation*}
Note that by abusing the notation the $l$ in $l+\rho_{1}$ is regarded as $(l,0,\ldots,0)\in \mathbb{C}^{n}$.
  \end{lem}
  \begin{proof}
 Since all the data are unramified, the integral on the left is actually a sum over $t=\w^{l}$ with $l\in \mathbb{Z}$. By the discussion of section \ref{sec:step4}, the Whittaker-Shintani function vanishes unless $l\geq 0$. Substituting $\wsf$ by the formula we developped in the previous sections, and note that $\dha
_{\gpb_{\gpg}}\begin{pmatrix}
    t && \\ &I_{2m} & \\ && t^{-1}
  \end{pmatrix}=|t|^{m+1}$, we have
\begin{equation*}
  \mathrm{LHS}=c^{-1}\cdot \sum_{l\geq 0}\left[\ww^{sl}\cdot\sum_{w\in W_{\gpg},w'\in W_{\gpm}}b(w\chag,w'\cham)d(w\chag)d'(w'\cham)((w\chag)^{-1})(\w^{l})\right].
\end{equation*}
Here $c=\zeta(1)^{m}\prod_{i=1}^{m}\zeta^{-1}(2i)$. Recall that $\rho_{2}=(m,m-1,\ldots, 1)$. Then
\begin{align*}
  &d(\chag)=(-1)^{m+1}\w^{-\ppair{\chag}{\rho_{1}}}\mathcal{A}_{W_{\gpg}}^{-1}(\w^{\ppair{\chag}{\rho_{1}}}).\\
  &d'(\cham)=(-1)^{m}\w^{-\ppair{\cham}{\rho_{2}}}\mathcal{A}_{W_{\gpm}}^{-1}(\w^{\ppair{\cham}{\rho_{2}}}).
\end{align*}
So 
\begin{equation}
\label{eq:n2}
  \begin{split}
   \mathrm{LHS}=&\frac{c^{-1}}{\mathcal{A}_{W_{\gpg}}(\w^{\ppair{\chag}{\rho_{1}}})\mathcal{A}_{W_{\gpm}}(\w^{\ppair{\cham}{\rho_{2}}})}\\\cdot &\sum_{l\geq 0}\left[\ww^{ls}\sum_{w\in W_{\gpg},w'\in W_{\gpm}}(-1)^{2m+1}sgn(w)sgn(w')\w^{-\ppair{w\chag}{l+\rho_{1}}}\w^{-\ppair{w'\cham}{\rho_{2}}}b(w\chag,w'\cham)\right].
  \end{split}
\end{equation}
 In fact the summation on $\gpw_{\gpg}\times \gpw_{\gpm}$ is equal to, by a change of variable $w\mapsto ww_{0}^{\gpg}$ and $w'\mapsto w'w_{0}^{\gpm}$, 
\begin{equation*}
  \sum_{w\in W_{\gpg},w'\in W_{\gpm}}sgn(w)sgn(w')\w^{\ppair{w\chag}{l+\rho_{1}}}\w^{\ppair{w'\cham}{\rho_{2}}}b(-w\chag,-w'\cham)
\end{equation*}

We can further simplify this summation by a similar discussion as in \textbf{Section} \ref{sec:step8}. By the definition of $b(\chag,\cham)$, we have
\begin{equation*}
  \begin{split}
    &\w^{\ppair{w\chag}{\rho_{1}+l}}\w^{\ppair{w'\cham}{\rho_{2}}}b(-w\chag,-w'\cham)\\
=&\left[\prod_{j=1}^{m}(1-\w^{-\chag_{1}\pm\cham_{j}+\frac 1 2}\w^{(l+m+\frac 1 2)\chag_{1}})\right]\left[\tilde{A}(\tilde\chag,\cham)\right].
  \end{split}
\end{equation*}
Here $\tilde{A}(\tilde\chag,\cham)$ is defined right after \textbf{lemma 10.4}. If we let $\gpw_{1}$ be the subgroup of $\gpw_{\gpg}$ stablizing $\chag_{2},\ldots, \chag_{m+1}$ and $\gpw_{2}$ the subgroup of $\gpw_{\gpg}$ stablizing $\chag_{1}$, and let $\gpw_{0}$ be a set of representatives in $(\gpw_{1}\times \gpw_{2})\backslash \gpw_{\gpg}$, then the first bracket is invariant under $\gpw_{\gpm}\times\gpw_{2}$, and the second bracket is invariant under $\gpw_{1}$. So the summation over $\gpw_{\gpg}\times \gpw_{\gpm}$ is equal to
\begin{align*}
  \sum_{w_{0},w_{1}}& sgn(w_{0}w_{1})[(w_{1}w_{0})\circ (\prod_{j=1}^{m}(1-\w^{-\chag_{1}\pm\cham_{j}+\frac 1 2}\w^{(l+m+\frac 1 2)\chag_{1}}))]\\
  &\cdot [\sum_{w_{2},w_{\gpm}}sgn(w_{2}w_{\gpm})\tilde A(w_{2}w_{0}\tilde\chag,w_{\gpm}\cham)]
\end{align*}
By \textbf{Lemma \ref{lem:105}}, the second bracket is equal to
\begin{equation*}
  c\cdot \sum_{w_{2},w_{\gpm}}sgn(w_{2}w_{\gpm})\w^{\ppair{w_{2}w_{0}\chag}{\tilde\rho_{1}}}\w^{\ppair{w_{\gpm}\cham}{\rho_{2}}},
\end{equation*} where $\tilde\rho_{1}=(0,m-\frac 1 2, m-\frac 3 2, \ldots, \frac 1 2)$. From this it is not hard to see that the summation over $\gpw_{\gpg}\times \gpw_{\gpm}$ in (\ref{eq:n2}) is equal to
\begin{align*}
  c\cdot \left[\sum_{w\in \gpw_{\gpg}}sgn(w)\cdot w\circ\left(\prod_{j=1}^{m}(1-\w^{-\chag_{1}\pm \cham_{j}+\frac 1 2})\w^{\ppair{\chag}{l+\rho_{1}}}\right)\right]\mathcal{A}_{\gpw_{\gpm}}(w^{\ppair{\cham}{\rho_{2}}})
\end{align*}
Substituting the formula to equation \ref{eq:n2} we obtain our lemma. 
\end{proof} 
Now we can prove \textbf{Theorem \ref{thm:13}}.
\begin{proof}[(Proof of \textbf{Theorem \ref{thm:13}})]
By Weyl's character formula, for the representation of $SO_{2N+1}(\mathbb{C})$ whose highest weight is $\lambda=(\lambda_{1},\ldots, \lambda_{N})\in \Lambda_{N}^{+}$,  the trace of $x=diag(x_{1},\ldots, x_{N},1,x_{N}^{-1},\ldots, x_{1}^{-1})$ is $\schh_{N}(\lambda;x)=\frac{\det (x_{i}^{\lambda_{j}+N-j+\frac 1 2}-x_{i}^{-(\lambda_{j}+N-j+\frac 1 2)})}{\det (x_{i}^{N-j+\frac 1 2}-x_{i}^{-(N-j+\frac 1 2)})}$. The function $\schh_{N}(\lambda;x)$ is in fact defined for all $\lambda\in \mathbb{Z}^{N}$. For any set $A=\{a_{1},\ldots, a_{N}\}$, we let $\wedge^{i}(A)=\sum_{S\subset A, |S|=i}(\prod_{s\in S}a_{s})$. Using these notation, we can express $\mathrm{LHS}$ as
\begin{equation}
  \label{eq:n3}
  \mathrm{LHS}=\sum_{l\geq 0, r \in \{0,1,\ldots, 2m\}}(-1)^{r}\wedge^{r}(\Gamma_{\tilde\sigma})\cdot\schh_{m+1}((l-r,0,\ldots,0);z_{\pi})\ww^{ls}.
\end{equation}
Here $\Gamma_{\tilde\sigma}$ is the set $\{\cham_{1}+\frac 1 2, \ldots, \cham_{m}+\frac 1 2, -\cham_{m}+\frac 1 2,\ldots, -\cham_{1}+\frac 1 2\}$. 
  Next we consider $\mathrm{RHS}$. By the discussion  in \cite[Theorem 3.1]{MR1675971}, we have
\begin{equation*}
  \frac{L(\pi,s)}{\zeta(2s)}=\sum_{a\geq 0}\schh_{m+1}( a;z_{\pi})\ww^{as}.
\end{equation*}
So by the notation introduced above, we have
\begin{equation}
\label{eq:rhs}
  \mathrm{RHS}=\sum_{a\geq 0, r\in \{0,1,\ldots, 2m\}^{m}}(-1)^{r}\wedge^{r}(\Gamma_{\tilde\sigma})\schh_{m+1}((a,0,\ldots,0);z_{\pi})\ww^{(a+r)s}.
\end{equation}
To show that (\ref{eq:n3}) equals (\ref{eq:rhs}), note that in (\ref{eq:n3}) if $l<r$, then $l-r\in \{-1,\ldots,-2m\}$ since $0\leq r \leq 2m$. Then it is not hard to see that $\schh_{m+1}(l-r,0,\ldots,0)=0$ by its definition. So one can replace the summation from $l \geq 0$ to $l\geq r$. Then by a change of the variable $l=a+r$ with $a\geq 0$ we have (\ref{eq:n3}) equals (\ref{eq:rhs}). 
\end{proof}
\bibliographystyle{plain}
\bibliography{math}

\begin{thebibliography}{10}

\bibitem{MR1671189}
William~D. Banks.
\newblock A corollary to {B}ernstein's theorem and {W}hittaker functionals on
  the metaplectic group.
\newblock {\em Math. Res. Lett.}, 5(6):781--790, 1998.

\bibitem{MR571057}
W.~Casselman.
\newblock The unramified principal series of {${p}$}-adic groups. {I}. {T}he
  spherical function.
\newblock {\em Compositio Math.}, 40(3):387--406, 1980.

\bibitem{casselman's_notes}
W.~Casselman.
\newblock {I}ntroduction to the theory of admissible representations of
  ${p}$-adic reductive groups.
\newblock \url{http://www.math.ubc.ca/~cass/research.html}, 1995.

\bibitem{MR581582}
W.~Casselman and J.~Shalika.
\newblock The unramified principal series of {$p$}-adic groups. {II}. {T}he
  {W}hittaker function.
\newblock {\em Compositio Math.}, 41(2):207--231, 1980.

\bibitem{fourierjacobi}
David Ginzburg, Dihua Jiang, Stephen Rallis, and David Soudry.
\newblock {$L$}-functions for symplectic groups using {F}ourier-{J}acobi
  models.
\newblock In {\em Arithmetic geometry and automorphic forms}, volume~19 of {\em
  Adv. Lect. Math. (ALM)}, pages 183--207. Int. Press, Somerville, MA, 2011.

\bibitem{MR1675971}
David Ginzburg, Stephen Rallis, and David Soudry.
\newblock {$L$}-functions for symplectic groups.
\newblock {\em Bull. Soc. Math. France}, 126(2):181--244, 1998.

\bibitem{MR1956080}
Shin-ichi Kato, Atsushi Murase, and Takashi Sugano.
\newblock Whittaker-{S}hintani functions for orthogonal groups.
\newblock {\em Tohoku Math. J. (2)}, 55(1):1--64, 2003.

\bibitem{MR1018057}
Atsushi Murase.
\newblock {$L$}-functions attached to {J}acobi forms of degree {$n$}. {I}.
  {T}he basic identity.
\newblock {\em J. Reine Angew. Math.}, 401:122--156, 1989.

\bibitem{MR1121142}
Atsushi Murase and Takashi Sugano.
\newblock Whittaker-{S}hintani functions on the symplectic group of
  {F}ourier-{J}acobi type.
\newblock {\em Compositio Math.}, 79(3):321--349, 1991.

\bibitem{xinshen2}
Xin {Shen}.
\newblock {Local theory for tensor L-function for symplectic groups}.
\newblock {\em preprint}, September 2012.

\end{thebibliography}

\end{document}